\documentclass[]{article}
\usepackage{graphicx}
\usepackage{amsmath,amsfonts,amssymb,amsthm,mathtools}
\usepackage{fancyhdr}
\usepackage{titlesec}
\usepackage{indentfirst}
\usepackage{booktabs}

\usepackage{hyperref}
\usepackage{xcolor}

\usepackage{algorithm}
\usepackage{algorithmic}

\usepackage{float}
\usepackage{subcaption}
\usepackage[export]{adjustbox}
\usepackage{placeins}

\usepackage[numbers,sort&compress,square]{natbib}

\mathtoolsset{showonlyrefs}

\newcommand{\R}{\mathbb{R}}
\newcommand{\E}{\mathbb{E}}

\numberwithin{equation}{section}
\newtheorem{theorem}{Theorem}[section]

\newtheorem{remark}[theorem]{Remark}

\begin{document}

\title{A Neural Score-Based Particle Method for \\the Vlasov--Maxwell--Landau System}

\author{Vasily Ilin\footnote{Department of Mathematics, University of Washington, Seattle, WA, USA. vilin@uw.edu.}
\quad
Jingwei Hu\footnote{Department of Applied Mathematics, University of Washington, Seattle, WA, USA. hujw@uw.edu.}}

\maketitle

\begin{abstract}
Plasma modeling is central to the design of nuclear fusion reactors, yet simulating collisional plasma kinetics from first principles remains a formidable computational challenge: the Vlasov--Maxwell--Landau (VML) system describes six-dimensional phase-space transport under self-consistent electromagnetic fields together with the nonlinear, nonlocal Landau collision operator.  A recent deterministic particle method for the full VML system \cite{bailo2024collisional} estimates the velocity score function via the blob method, a kernel-based approximation with $O(n^2)$ cost.  In this work, we replace the blob score estimator with score-based transport modeling (SBTM), in which a neural network is trained on-the-fly via implicit score matching at $O(n)$ cost.  We prove that the approximated collision operator preserves momentum and kinetic energy, and dissipates an estimated entropy.  We also characterize the unique global steady state of the VML system and its electrostatic reduction, providing the ground truth for numerical validation.  On three canonical benchmarks -- Landau damping, two-stream instability, and Weibel instability -- SBTM is more accurate than the blob method, achieves correct long-time relaxation to Maxwellian equilibrium where the blob method fails, and delivers $50\%$ faster runtime with $4\times$ lower peak memory.
\end{abstract}

\section{Introduction}
\label{sec:intro}

Controlled nuclear fusion promises a virtually limitless source of clean energy, but achieving it requires confining a plasma -- an ionized gas of electrons and ions -- at temperatures exceeding $10^8$\,K, where charged particles undergo frequent Coulomb collisions while generating and responding to electromagnetic fields \citep{chen2016introduction}.  Accurate numerical simulation of such plasmas is essential for reactor design and remains one of the grand challenges in computational physics.  The Vlasov--Maxwell--Landau (VML) system provides the first-principles kinetic description \citep{landau1936kinetische, villani2002review}, coupling six-dimensional phase-space transport with self-consistent electromagnetic fields and the Landau collision operator.  Solving the VML equations is exceptionally challenging due to the high dimensionality, multiscale dynamics, and the need to preserve conservation laws and entropy structure.

Particle-in-cell (PIC) methods \citep{hockney1988computer, birdsall2018plasma} are the dominant approach for high-dimensional kinetic simulations: particles are advanced in the Lagrangian frame while fields are solved on a spatial grid.  However, incorporating the Landau collision operator into PIC is nontrivial.  Classical Monte Carlo approaches \citep{takizuka1977binary, bird1994molecular, manheimer1997langevin} introduce statistical noise, and because they are not compatible with the standard Lagrangian formulation of PIC, collisions must be incorporated in a time-splitting manner, resulting in additional discretization errors.

A deterministic alternative is possible because the Landau collision operator can be rewritten as a \emph{transport equation} driven by the velocity score function $s(x,v) = \nabla_v \log f(x,v)$ \citep{villani1998spatially}, reducing the problem to estimating the score from particle data.  The first deterministic particle method for the homogeneous Landau equation was introduced in \citep{carrillo2020particle}, where the score was estimated via a kernel density estimate (the ``blob method'').  This approach was extended to the full VML system in \citep{bailo2024collisional}. By viewing the Landau operator as a ``collision force" in velocity space, the method naturally incorporates collisions into the velocity update of the standard PIC framework, preserving the main physical structures. However, the blob method requires regularization with a kernel in velocity space. Choosing an appropriate bandwidth for the kernel can be challenging. A poor choice could lead to inaccurate scores in low-density tail regions, causing systematic artifacts in long-time relaxation to equilibrium. Furthermore, the blob method requires computing a few pairwise kernel sums over all $n$ particles in each spatial cell, resulting in $O(n^2)$ computational cost and memory usage. In fact, due to the high computational cost, \citep{bailo2024collisional} employs a random batch method to reduce the per-step cost to $O(n^2/R)$ \citep{carrillo2022random}, but this comes at the expense of introducing additional random error.

Score-based transport modeling (SBTM) \citep{boffi2023probability} offers an alternative: a neural network $s_\theta$ is trained on-the-fly to approximate the score by minimizing an implicit score matching loss \citep{hyvarinen2005estimation}, at $O(n)$ cost per gradient step.  This was applied to the spatially homogeneous Landau equation in \citep{ilin2025transport, huang2024score}, demonstrating superior accuracy especially in sparse particle regimes, and error bounds for SBTM with Coulomb collisions were established in \citep{ilin2025stability}.

\textbf{Other related work.}  A convergence analysis of the blob method for the homogeneous Landau equation was conducted in \citep{carrillo2023convergence}. Score matching \citep{hyvarinen2005estimation} was introduced for learning unnormalized statistical models.  SBTM has been applied to various (spatially homogeneous) Fokker-Planck equations \citep{shen2022self, shen2023entropy, lu2024score} and density sampling \citep{ilin2025score}.  \citep{ilin2026kernels} proposed DiScoFormer, a pretrained Transformer for score estimation without per-distribution retraining.  


\textbf{Contributions.} Our contributions are as follows:
\begin{itemize}
\item We apply SBTM to the full Vlasov--Maxwell--Landau system, extending the spatially homogeneous work \citep{ilin2025transport} to the spatially inhomogeneous case with electromagnetic field coupling.  Our implementation builds on the collisional PIC framework of \citep{bailo2024collisional}, replacing \emph{only} the score estimation module.  This extension requires spatial localization of the score network via the hat kernel $\psi_\eta$, coupling to the Maxwell field solver, and handling distinct velocity distributions across spatial cells with a single network.
\item We prove that the approximated collision operator preserves \emph{momentum and kinetic energy} exactly, and dissipates an estimated entropy (Theorem~\ref{thm:collision}).
\item We characterize the unique global steady state of both the full VML system (Theorem~\ref{thm:equilibrium}) and its electrostatic reduction, the VPL system (Theorem~\ref{thm:equilibrium_vpl}), providing the theoretical ground truth against which we validate our numerical results.  A formal verification of Theorem~\ref{thm:equilibrium} in the Lean~4 proof assistant is carried out in \citep{ilin2026formalization}.
\item Through experiments on three canonical benchmarks, we demonstrate that SBTM achieves correct long-time \emph{relaxation to Maxwellian equilibrium} (where the blob method fails), \emph{converges faster in the number of particles}, and delivers ${\sim}50\%$ faster runtime with $2$--$4\times$ lower peak memory.
\end{itemize}

\section{Background and Method}
\label{sec:method}

\textbf{Notation.}  We write $x \in \Omega_x \subset \R^{d_x}$ for spatial position, $v \in \R^{d_v}$ for velocity, and $f(t,x,v)$ for the particle distribution function.  The spatial domain is discretized into $M$ cells of width $\eta = L/M$, with $\psi_\eta$ denoting the hat kernel used for spatial localization.  The system contains $n$ particles with positions $x_p$, velocities $v_p$, and weights $w_p$.  The collision frequency is $\nu$, and $K$ denotes the number of implicit score matching (ISM) gradient steps per time step $\Delta t$.  We use $s(x,v) = \nabla_v \log f(x,v)$ for the velocity score function and $s_\theta$ for its neural network approximation.

\subsection{The Vlasov--Maxwell--Landau System}
\label{sec:vml}
The Vlasov--Maxwell--Landau (VML) system describes the evolution of a collisional, magnetized plasma.  After non-dimensionalization, for a single electron species with an immobile neutralizing ion background, the distribution function $f(t, x, v)$ evolves according to
\begin{align}
\label{eq:vlasov}
&\partial_t f + v \cdot \nabla_x f + (E + v \times B) \cdot \nabla_v f = \nu \, Q[f, f],  \\
\label{eq:maxwell_full}
&\partial_t E = \nabla_x \times B - J, \qquad \partial_t B = -\nabla_x \times E, \qquad \nabla_x \cdot E = \rho - \rho_{\mathrm{ion}}, \qquad \nabla_x \cdot B = 0, 
\end{align}
where $t\geq 0$ is the time, $x\in \Omega_x\subset \R^{d_x}$ is the position, and $v\in \R^{d_v}$ is the particle velocity. The electric and magnetic fields $E(t,x)$ and $B(t,x)$ in \eqref{eq:vlasov}   are determined by the Maxwell's equations \eqref{eq:maxwell_full}, where $\rho(t,x)=\int_{\R^{d_v}}f \, dv$ and $J(t,x)=\int_{\R^{d_v}} v f \, dv$ are the charge density and current density, and $\rho_{\mathrm{ion}}$ is the uniform background ion density satisfying $\int_{\Omega_x}(\rho-\rho_{\mathrm{ion}})\, dx=0$. Finally, $\nu \geq 0$ is the collision frequency and $Q[f, f]$ is the Landau operator modeling particle collisions \citep{landau1936kinetische}:
\begin{equation}
\label{eq:landau}
Q[f, f](t,x, v) = \nabla_v \cdot \int_{\R^{d_v}} A(v - w) \big[ f(t,x,w) \nabla_v f(t,x,v) - f(t,x,v) \nabla_w f(t,x,w) \big] \, dw,
\end{equation}
where the collision kernel $A$ is given by $A(z) = |z|^{\gamma+2} \Pi(z)$ with $\Pi(z) = I_{d_v} - z z^\top / |z|^2$, and $\gamma = -d_v$ corresponds to the Coulomb interaction ($\gamma=-3$ for $d_v=3$ and $\gamma=-2$ for $d_v=2$).

In the electrostatic case, $B=0$, so $E$ is curl free. Then \eqref{eq:vlasov}-\eqref{eq:maxwell_full} reduce to the Vlasov--Poisson--Landau (VPL) system:
\begin{align}
\label{eq:vpl_1}
&\partial_t f + v \cdot \nabla_x f + E \cdot \nabla_v f = \nu\, Q[f,f], \\
\label{eq:vpl_2}
& \partial_t E=-J, \qquad E = -\nabla_x \phi, \qquad -\Delta_x \phi = \rho - \rho_{\mathrm{ion}}.
\end{align}

The VML system \eqref{eq:vlasov}-\eqref{eq:maxwell_full} possesses many important physical properties. We highlight a few that are relevant to the following discussion. First, the Landau operator \eqref{eq:landau} satisfies local conservation of mass, momentum, and energy, as well as the H-theorem \cite{villani2002review}:
\begin{equation}
\label{eq:original_Landau}
\int_{\R^{d_v}} Q[f,f]\,dv=\int_{\R^{d_v}} Q[f,f]v_i\,dv=\int_{\R^{d_v}} Q[f,f]|v|^2\,dv=0, \quad \int_{\R^{d_v}} Q[f,f]\log f\,dv\leq 0,
\end{equation}
where $v_i$ is the $i$-th component of $v$.
Using these properties, together with periodic boundary conditions in $x$, one can derive global entropy decay and conservation of mass and energy:
\begin{align}
\label{eq:H-theorem}
&\partial_t \iint_{\Omega_x\times \R^{d_v}} f\log f\,dv\,dx\leq 0, \quad \partial_t\iint_{\Omega_x\times \R^{d_v}} f\,dv\,dx=0, \\
\label{eq:energy-consv}
& \partial_t \left(\iint_{\Omega_x\times \R^{d_v}} \frac{1}{2}|v|^2f \,dv\,dx+\frac{1}{2}\int_{\Omega_x}(|E|^2+|B|^2)\,dx\right)=0.
\end{align}
Regarding momentum, one has
\begin{equation}
\partial_t \left( \iint_{\Omega_x\times \R^{d_v}} vf\,dv\,dx+\int_{\Omega_x} (E\times B)\,dx\right)=\int_{\Omega_x} \rho_{\text{ion}}E\,dx,
\end{equation}
so momentum is generally not conserved, except in the electrostatic case $B=0$, where $E=-\nabla_x\phi$ and $\rho_{\text{ion}}$ is constant. In this case, 
\begin{equation}
\label{eq:moment-consv}
\partial_t  \iint_{\Omega_x\times \R^{d_v}} vf\,dv\,dx=0.
\end{equation}

Due to the complexity of the VML system, constructing analytical solutions for numerical validation is challenging. As a result, the long-time behavior -- namely, the equilibrium distribution corresponding to a given initial condition -- serves as an important metric. We therefore establish the following theorems characterizing the steady state of the VML system and its reduced VPL version. Although the proofs are elementary, we have not found them in standard textbooks or the existing literature; we therefore provide complete proofs in  Appendix~\ref{app:proofs_equilibrium}.

\begin{theorem}[Global steady state of the VML system]
\label{thm:equilibrium}
Consider the Vlasov--Maxwell--Landau system \eqref{eq:vlasov}-\eqref{eq:maxwell_full} with collision frequency $\nu > 0$, periodic spatial domain $x \in \mathbb{T}^{3}$, $v \in \R^{3}$, and a uniform neutralizing background ion density $\rho_{\mathrm{ion}} > 0$.  Let $f(x,v) > 0$ be a sufficiently smooth steady-state solution with finite total energy and finite entropy, and $f$ and its derivatives decay sufficiently fast as $|v| \to \infty$.  Then $f$ must be a spatially uniform, zero-drift global Maxwellian:
\begin{equation}
\label{eq:equilibrium}
f_\infty(v) = \frac{\rho_{\mathrm{ion}}}{(2\pi T_\infty)^{3/2}} \exp\!\left(-\frac{|v|^2}{2T_\infty}\right), \qquad E_\infty = 0, \qquad B_\infty = \frac{1}{|\mathbb{T}^3|}\int_{\mathbb{T}^3} B_{\mathrm{init}}(x) \, dx,
\end{equation}
where $B_\infty$ is the spatial mean of the initial magnetic field (a conserved quantity), and the equilibrium temperature is
\begin{equation}
\label{eq:T_infty_statement}
T_\infty = \frac{2}{3\, \rho_{\mathrm{ion}} \, |\mathbb{T}^3|} \left(\mathcal{E}_0 - \frac{|B_\infty|^2}{2} |\mathbb{T}^3|\right) > 0,
\end{equation}
with $\mathcal{E}_0 = \iint \frac{1}{2}|v|^2 f_{\mathrm{init}} \, dv \, dx + \frac{1}{2}\int (|E_{\mathrm{init}}|^2 + |B_{\mathrm{init}}|^2) \, dx$ denoting the initial total energy.
\end{theorem}

In the electrostatic case ($B = 0$), the VML system reduces to the VPL system. The argument above requires a genuine modification: the bulk velocity is no longer forced to vanish.

\begin{theorem}[Global steady state of the VPL system]
\label{thm:equilibrium_vpl}
Consider the Vlasov--Poisson--Landau system \eqref{eq:vpl_1}-\eqref{eq:vpl_2} with 
$\nu > 0$, $x \in \mathbb{T}^3$, $v \in \R^3$, and a uniform neutralizing background $\rho_{\mathrm{ion}} > 0$.  Under the same regularity and decay assumptions as in Theorem~\ref{thm:equilibrium}, every steady-state solution is a spatially uniform, drifting global Maxwellian:
\begin{equation}
\label{eq:equilibrium_vpl}
f_\infty(v) = \frac{\rho_{\mathrm{ion}}}{(2\pi T_\infty)^{3/2}} \exp\!\left(-\frac{|v - u_\infty|^2}{2T_\infty}\right), \qquad E_\infty = 0,
\end{equation}
where the bulk drift and temperature are determined by the initial conditions:
\begin{equation}
\label{eq:u_T_vpl_statement}
u_\infty = \frac{\iint v \, f_{\mathrm{init}} \, dv \, dx}{\rho_{\mathrm{ion}} \, |\mathbb{T}^3|}, \qquad T_\infty = \frac{2}{3\, \rho_{\mathrm{ion}} \, |\mathbb{T}^3|}\!\left(\mathcal{E}_0 - \frac{1}{2}\rho_{\mathrm{ion}} \, |u_\infty|^2 \, |\mathbb{T}^3|\right) > 0,
\end{equation}
with $\mathcal{E}_0 = \iint \frac{1}{2}|v|^2 f_{\mathrm{init}} \, dv \, dx + \frac{1}{2}\int |E_{\mathrm{init}}|^2 \, dx$ denoting the initial total energy.
\end{theorem}

\subsection{Regularized Landau Operator and Particle Method}
\label{sec:particles}

In this section, we introduce the particle method for the VML system \eqref{eq:vlasov}-\eqref{eq:maxwell_full}. For convenience of presentation, we assume a one-dimensional spatial domain $\Omega_x=[0,L]$ with periodic boundary conditions, while the velocity domain remains $\R^{d_v}$, where $d_v=2$ or $3$. Our discussion can be straightforwardly extended to multiple spatial dimensions.

The Landau collision operator \eqref{eq:landau} can be rewritten in terms of the velocity score $\nabla_v \log f$ \citep{villani1998spatially, carrillo2020particle}:
\begin{align}
\label{eq:landau_score1}
Q[f, f](t, x, v) &= \nabla_v \cdot \bigg\{ f(t,x,v) \int_{\R^{d_v}} A(v - w)\, \big[ \nabla_v \log f(t,x,v) - \nabla_w \log f(t,x,w) \big]f(t,x,w)  \, dw \bigg\}\\
&:=\nabla_v \cdot \bigg\{ f(t,x,v) U(t,x,v)\bigg\}.
\end{align}
This formulation forms the basis of the particle method: the collision operator can be viewed as a transport operator with ``collision force" $U(t,x,v)$, which depends only on the score function rather than the distribution function itself. 

However, the spatial dependence in $U(t,x,v)$ does not admit a direct particle approximation. Following the idea in \citep{bailo2024collisional}, we introduce a regularized ``collision force":
\begin{equation}
\label{eq:landau_score2}
U^{\eta}(t,x,v) :=\iint_{\Omega_x\times\R^{d_v}} \psi_{\eta}(x-y)A(v - w)\,  \big[ \nabla_v \log f(t,x,v) - \nabla_w \log f(t,y,w) \big] f(t,y,w)\, dw \,dy,
\end{equation}
where $\psi_\eta$ is a symmetric, positive kernel function used for spatial localization. Typical choices include B-spline kernels. This localization allows only particles within a certain neighborhood to undergo collisions. It is an easy exercise to show that the regularized Landau collision operator $Q^{\eta}[f,f]:=\nabla_v\cdot \{ f U^\eta\}$ satisfies conservation of mass, momentum, and energy:
\begin{align}
\iint Q^\eta[f,f]\, dv\,dx=\iint Q^\eta[f,f] v_i\, dv\,dx=\iint Q^\eta[f,f]|v|^2\, dv\,dx=0,
\end{align}
as well as the H-theorem:
\begin{equation}
\iint Q^{\eta}[f,f]\log f\,dv \,dx\leq 0.
\end{equation}
However, note the difference from \eqref{eq:original_Landau}.

With the above regularization, \eqref{eq:vlasov} can be written in conservative form:
\begin{equation}
\label{eq:conservative}
\partial_t f + \nabla_x \cdot (v f) + \nabla_v \cdot ((E + v \times B) f) = \nu \, \nabla_v \cdot (U^\eta f).
\end{equation}
We now represent the distribution function $f$ as a sum of $n$ weighted particles:
\begin{equation}
\label{eq:weaksoln}
f(t, x, v) \approx \sum_{p=1}^{n} w_p \, \delta(x - x_p(t)) \, \delta(v - v_p(t)).
\end{equation}
\eqref{eq:weaksoln} is a weak solution to \eqref{eq:conservative} provided that 
the particle positions $x_p$ and velocities $v_p$ evolve according to the coupled ODE system (where $v_{1,p}$ denotes the first component of $v_p$):
\begin{align}
\label{eq:particle_ode}
\frac{dx_p}{dt} &= v_{1,p} \pmod{L}, \\
\label{eq:collision_particle}
\frac{dv_p}{dt} &= \big(E(x_p) + v_p \times B(x_p)\big) - \nu U^{\eta}(x_p,v_p).
\end{align}
Here, the collision force $U^{\eta}(x_p,v_p)$ is given by
\begin{equation}
\label{eq:Landau_particle}
U^{\eta}(x_p,v_p):=\sum_{q=1}^{n} w_q \, \psi_\eta(x_p - x_q) \, A(v_p - v_q) \big[ s(x_p, v_p) - s(x_q, v_q) \big],
\end{equation}
where $s(x,v) = \nabla_v \log f(x,v)$ is the velocity \emph{score function}. The score function is the key quantity in the algorithm, and its approximation will be discussed in Section~\ref{sec:score}. For the moment, we assume that, given particles $\{x_p,v_p\}_{p=1}^n$, a good estimator $\{s(x_p,v_p)\}_{p=1}^n$ is available.

To couple the ODE system \eqref{eq:particle_ode}-\eqref{eq:collision_particle} with the Maxwell's equations \eqref{eq:maxwell_full}, we adopt the standard particle-in-cell (PIC) procedure. We partition the spatial domain $[0,L]$ into $M$ cells of size $\eta=L/M$, with cell centers $x_j = (j - \tfrac{1}{2})\eta$, $j = 1, \ldots, M$. In the collision operator, we choose $\psi_\eta$ as the degree-1 B-spline (hat function):
\begin{equation}
\label{eq:hat_kernel}
\psi_\eta(x) = \frac{1}{\eta}\!\left(1 - \frac{|x|}{\eta}\right)_{\!\!+}\!, \quad (z)_+ = \max(z, 0).
\end{equation}
Since $\psi_\eta$ has compact support $[-\eta, \eta]$, only particles within distance $\eta$ of each other undergo collisions, thereby localizing interactions to particles in the same or adjacent spatial cells.  

The same $\psi_{\eta}$ is used to deposit charge and current densities:
\begin{equation}
\label{eq:deposit}
\rho_j = \sum_{p=1}^{n} w_p\psi_\eta(x_j - x_p), \qquad J_{i,j} =  \sum_{p=1}^{n}w_p v_{i,p}\, \psi_\eta(x_j - x_p), \quad i = 1, \ldots, d_v.
\end{equation}
We then solve the Maxwell's equations using the grid values $\rho_j$ and $J_{i,j}$. The updated electric and magnetic fields $E_j$ and $B_j$ are then interpolated back to the particle positions via
\begin{equation}
\label{eq:interpolate}
E(x_p) = \eta \sum_{j=1}^{M} \psi_\eta(x_p - x_j)\, E_j, \qquad B(x_p) = \eta \sum_{j=1}^{M} \psi_\eta(x_p - x_j)\, B_j.
\end{equation}


The proposed discretization of the collision force \eqref{eq:Landau_particle} inherits the main physical properties of the regularized Landau operator, such as conservation and entropy decay. This is most naturally seen through the \emph{weak formulation}: for any test function $\phi(x,v)\in \R^{d_v}$,
\begin{align}
&\nu\sum_p w_p  \phi (x_p,v_p)\cdot U^{\eta}(x_p,v_p)=\nu\sum_{p,q} w_p w_q \psi_\eta(x_p - x_q)  \phi(x_p,v_p)^\top A(v_p-v_q)\,[s(x_p,v_p) - s(x_q,v_q)]\\
=&\frac{\nu}{2}\sum_{p,q} w_p w_q \, \psi_\eta(x_p - x_q) \big[\phi(x_p,v_p) - \phi(x_q,v_q)\big]^\top A(v_p-v_q)\,[s(x_p,v_p) - s(x_q,v_q)],\label{eq:D_symmetrized}
\end{align}
where the last line is obtained by swapping $p \leftrightarrow q$ in the double sum and averaging.

We can state the following result.
\begin{theorem}
\label{thm:collision}
For any score approximation $s(x,v)$ and approximate force $U^\eta$ defined by \eqref{eq:Landau_particle}, the following properties hold:
\begin{enumerate}
\item[(i)] \textbf{Conservation of momentum and energy:}
\begin{equation}
\sum_p w_p \, U^{\eta}(x_p,v_p) = 0, \qquad \sum_p w_p \, v_p \cdot U^{\eta}(x_p,v_p) = 0.
\end{equation}
\item[(ii)] \textbf{Estimated entropy dissipation:} 
\begin{equation}
-\nu\sum_p w_p s (x_p,v_p)\cdot U^{\eta}(x_p,v_p)\leq 0.
\end{equation}
\end{enumerate}
\end{theorem}

\begin{proof}
\textbf{(i)} Using \eqref{eq:D_symmetrized} with $\phi(x,v) = e_i$, conservation of momentum is immediate. Conservation of energy follows from $(v_p - v_q)^\top A(v_p - v_q) = 0$, since for any $z \in \R^{d_v}$,
\begin{equation}
\label{eq:zA_zero}
z^\top A(z) = |z|^{\gamma+2}\, z^\top \Pi(z) = |z|^{\gamma+2}\, z^\top\!\left(I_{d_v} - \frac{zz^\top}{|z|^2}\right) = 0^\top.
\end{equation}

\textbf{(ii)} Taking $\phi(x,v)=s(x,v)$ in \eqref{eq:D_symmetrized}, we obtain
\[
\nu\sum_p w_p  \phi (x_p,v_p)\cdot U^{\eta}(x_p,v_p)\geq 0,
\]
since $A(z)$ is positive semidefinite. 
\end{proof}

\begin{remark}
\label{remark-cons}
The above theorem guarantees that the contribution from the collision term to the particle method is physically consistent. In particular, provided that Maxwell's equations are solved appropriately (e.g., using Yee's finite difference scheme \cite{Yee66}), the semi-discrete particle method \eqref{eq:particle_ode}-\eqref{eq:collision_particle} preserves the total energy of the system. In addition, an energy-conserving second-order explicit time discretization scheme has recently been developed \cite{yoo2025explicit}, which can be readily combined with the present particle method to achieve fully discrete energy conservation. However, since this approach requires two or three evaluations of the collision operator per time step, we opt for the simpler forward Euler for efficiency.
\end{remark}

\subsection{Score Estimation}
\label{sec:score}

To estimate the score function $s(x,v)=\nabla_v \log f(x,v)$ using particles, we employ score-based transport modeling (SBTM) \citep{boffi2023probability, ilin2025transport} in a spatially inhomogeneous setting. 

SBTM trains a neural network $s_\theta \colon \R^{1+d_v} \to \R^{d_v}$ to approximate the score via implicit score matching (ISM) \citep{hyvarinen2005estimation}. The ideal objective is to minimize the expected squared error $\E_f[\|s_\theta - \nabla_v \log f\|^2]$. Expanding the square yields $\E_f[\|s_\theta\|^2] - 2\E_f[s_\theta \cdot \nabla_v \log f] + C$, where $C$ is independent of $\theta$. For the cross term, integration by parts gives
\begin{equation}
\label{eq:ism_ibp}
\iint_{\Omega_x\times\R^{d_v}} s_\theta \cdot (\nabla_v \log f)\, f \, dv \,dx= \iint_{\Omega_x\times\R^{d_v}} s_\theta \cdot \nabla_v f \, dv \, dx= -\iint_{\Omega_x\times\R^{d_v}} (\nabla_v \cdot s_\theta)\, f \, dv\, dx,
\end{equation}
so the unknown true score drops out, leaving the equivalent ISM loss. Replacing the expectation with a particle average gives
\begin{equation}
\label{eq:ism}
\mathcal{L}(\theta) = \sum_{p=1}^{n} w_p \Big[ |s_\theta(x_p, v_p)|^2 + 2\, \nabla_v \cdot s_\theta(x_p, v_p) \Big].
\end{equation}
The divergence $\nabla_v \cdot s_\theta$ is estimated via Hutchinson's trace estimator \citep{hutchinson1989stochastic} with Rademacher random vectors:
\begin{equation}
\label{eq:hutchinson}
\nabla_v \cdot s_\theta(x, v) \approx z^\top \big(\nabla_v s_\theta(x, v)\big) z, \quad z \sim \mathrm{Rademacher}(\pm 1)^{d_v},
\end{equation}
where the Jacobian-vector product $(\nabla_v s_\theta) z$ is computed via a single forward-mode automatic differentiation pass. We use a two-layer MLP with softsign activation (hidden dimension 256 for Landau damping and two-stream instability, and 512 for Weibel instability).  At $t = 0$, the network is pretrained on the known analytic score of the initial condition.  At each subsequent time step, $K$ gradient steps on $\mathcal{L}(\theta)$ are performed using AdamW with learning rate $2 \times 10^{-4}$. Importantly, each gradient descent step is only $O(n)$ in runtime and memory, delivering significant speedups at the scale of $n=O(10^6)$ particles.

The full SBTM-PIC method is summarized in Algorithm~\ref{alg:sbtm_pic}. We assume that $E = (E_1(t,x), E_2(t,x), 0)$ and $B = (0, 0, B_3(t,x))$, in which case Maxwell's equations reduce to:
\begin{equation}
\label{eq:maxwell_1d}
\partial_t E_1 = -J_1, \qquad \partial_t E_2 = -\partial_x B_3 - J_2, \qquad \partial_t B_3 = -\partial_x E_2.
\end{equation}
Both $E$ and $B$ are discretized at the cell centers $x_j$ and updated using centered second-order finite differences.

\begin{algorithm}[t]
\caption{SBTM-PIC for the Vlasov--Maxwell--Landau system}
\label{alg:sbtm_pic}
\begin{algorithmic}[1]
\REQUIRE Particles $\{x_p, v_p\}_{p=1}^n$ with weights $w_p$, grid fields $E_j$, $B_j$, time step $\Delta t$, final time $t_{\mathrm{final}}$, collision frequency $\nu$, ISM steps $K$, spatial kernel $\psi_\eta$, collision kernel $A$
\STATE Pretrain $s_\theta$ on the analytic score of the initial condition
\STATE $t \leftarrow 0$
\WHILE{$t < t_{\mathrm{final}}$}
  \STATE $E_p \leftarrow \eta\sum_{j}\psi_\eta(x_p{-}x_j)\,E_j$, $\; B_p \leftarrow \eta\sum_{j}\psi_\eta(x_p{-}x_j)\,B_j$ \hfill \textit{(grid $\to$ particle)}
  \STATE $v_p \leftarrow v_p + \Delta t \, (E_p + v_p \times B_p)$ \hfill \textit{(Lorentz push)}
  \STATE $x_p \leftarrow x_p + \Delta t \, v_{1,p} \pmod{L}$ \hfill \textit{(position advance)}
  \STATE $\rho_j \leftarrow \sum_{p} w_p\,\psi_\eta(x_j-x_p)$, \; $J_{i,j} \leftarrow \sum_{p} w_p\,v_{i,p}\,\psi_\eta(x_j{-}x_p)$ \hfill \textit{(particle $\to$ grid)}
  \STATE $E_{1,j} \leftarrow E_{1,j} - \Delta t\, J_{1,j}$;  \; $E_{2,j} \leftarrow E_{2,j} - \Delta t\,\left((B_{3,j+1} - B_{3,j-1})/(2\Delta x) + J_{2,j}\right)$; \\ 
  $B_{3,j} \leftarrow B_{3,j} - \Delta t\,(E_{2,j+1} - E_{2,j-1})/(2\Delta x)$; \hfill \textit{(field update)}
  \FOR{$k = 1, \ldots, K$}
    \STATE Draw $z \sim \mathrm{Rademacher}(\pm 1)^{d_v}$
    \STATE $\theta \leftarrow \theta - \alpha\,\nabla_\theta \!\left[\sum_{p} w_p\big(|s_\theta(x_p,v_p)|^2 + 2\,z^\top (\nabla_v s_\theta(x_p,v_p))\,z\big)\right]$ \hfill \textit{(ISM)}
  \ENDFOR
  \STATE $v_p \leftarrow v_p - \Delta t\,\nu \!\sum_{q} w_q\,\psi_\eta(x_p{-}x_q)\,A(v_p{-}v_q)\big[s_\theta(x_p,v_p) - s_\theta(x_q,v_q)\big]$ \hfill \textit{(collision)}
  \STATE $t \leftarrow t + \Delta t$
\ENDWHILE
\end{algorithmic}
\end{algorithm}

{\bf Comparison with the blob method}. 
The blob method 
\citep{carrillo2019blob, carrillo2020particle, bailo2024collisional} is an alternative approach for score estimation based on a regularized kernel approximation. While several variants exist, the formulation in \citep{bailo2024collisional} employs a regularization   consisting of a leading kernel density estimation (KDE) term together with a correction that ensures regularized entropy dissipation; see \citep[Eq. (2.25)]{bailo2024collisional}. For a thorough comparison, we implement this method. Since the correction is roughly four orders of magnitude smaller than the KDE term in our experiments and has a negligible effect on particle trajectories, we use only the KDE-based score:
\begin{equation}
\label{eq:blob}
s_{\text{KDE}}(x, v) = \nabla_v \log \bigg( \sum_{q} w_q K_h(v - v_q) \, \psi_\eta(x - x_q) \bigg),
\end{equation}
where $K_h$ is a Gaussian kernel with bandwidth $h$. This formula is approximately twice as fast as computing the full blob score as in \citep{bailo2024collisional}, while yielding empirically identical trajectories. Selecting an appropriate bandwidth $h$ can be challenging, especially when the velocity dimension $d_v$ is high; in this work, we use Silverman's rule \citep{silverman2018density}. The blob method uses exactly the same Algorithm~\ref{alg:sbtm_pic}, except that the ISM training (lines~9 -- 12) is replaced by the KDE formula~\eqref{eq:blob}. Evaluating \eqref{eq:blob} requires $O(n^2)$ pairwise computations within each spatial cell and $O(n^2)$ memory.

\section{Numerical Experiments}
\label{sec:experiments}


\subsection{Setup}

All experiments use the 1D-$d_v$V setting ($d_v \in \{2, 3\}$).  The spatial domain $[0, L]$ is discretized into $M$ grid cells with spacing $\eta = L / M$ and  periodic boundary conditions. In each experiment, $n$ particles are initialized by sampling positions from the spatial marginal of $f_0$ and velocities from the velocity marginal; each particle carries equal weight $w_p = L/n$. The first two examples (Landau damping and two-stream instability) are electrostatic cases, where we solve the VPL system \eqref{eq:vpl_1}-\eqref{eq:vpl_2}. The last example (Weibel instability) is an electromagnetic case, where we solve the VML system \eqref{eq:vlasov}-\eqref{eq:maxwell_full}. We compare two score estimation methods -- blob (KDE) and SBTM -- within the same PIC framework.  The \emph{only} difference between the two methods is how the score $\nabla_v \log f$ is estimated. 

The code is implemented in JAX \citep{bradbury2018jax} with double precision and is available at \url{https://github.com/Vilin97/Vlasov-Landau-SBTM}.  All experiments were run on NVIDIA L40S and H200 GPUs on the University of Washington Hyak and Tillicum clusters.

\textbf{Diagnostic quantities.}  We monitor the following discrete analogues of energy conservation and entropy dissipation \eqref{eq:H-theorem}-\eqref{eq:energy-consv}. The total energy is defined as
\begin{equation}
\label{eq:diagnostics}
\mathcal{E}_K = \frac{1}{2}\sum_{p=1}^{n} w_p |v_p|^2, \quad
\mathcal{E}_E = \frac{\eta}{2}\sum_{j=1}^{M} |E_j|^2, \quad
\mathcal{E}_B = \frac{\eta}{2}\sum_{j=1}^{M} |B_j|^2, \quad
\mathcal{E} = \mathcal{E}_K + \mathcal{E}_E + \mathcal{E}_B,
\end{equation}
where $\mathcal{E}_K$ is the kinetic energy, $\mathcal{E}_E$ is the electric field energy, and $\mathcal{E}_B$ is the magnetic field energy (present only in the electromagnetic case). The total energy $\mathcal{E}$ is conserved at the semi-discrete level by Theorem~\ref{thm:collision} (i) and Remark~\ref{remark-cons}. The estimated entropy production is defined as
\begin{equation}
\label{eq:entropy_production}
\dot{\mathcal{H}} = \sum_{p=1}^{n}w_p s(x_p, v_p) \cdot U^{\eta}(x_p, v_p),
\end{equation}
which is non-negative by Theorem~\ref{thm:collision} (ii). This quantity is expected to vanish at equilibrium. 

\subsection{Landau Damping}
\label{sec:landau}

\textbf{Setup.}  A small spatial perturbation on a Maxwellian equilibrium:
\begin{equation}
f_0(x, v) = \frac{1 + \alpha \cos(kx)}{(2\pi)^{d_v/2}} \exp\!\left(-\frac{|v|^2}{2}\right),
\end{equation}
with wavenumber $k = 0.5$ and domain $L = 2\pi/k$.  We use $\alpha = 0.1$, $d_v = 3$, $M = 100$, $\Delta t = 0.02$, $\nu = 0.4$, $t_{\mathrm{final}} = 15$, and $K = 100$, sweeping $n \in \{5 \times 10^5, 10^6, 3 \times 10^6\}$. 



\textbf{Estimated entropy production and total energy.}  Figure~\ref{fig:app_landau_entropy_energy} shows the estimated entropy production and total energy evolution.  SBTM produces consistent estimated entropy production curves across all particle counts, while the blob method overestimates especially at lower $n$ (plot (a)).  Since the system relaxes to a Maxwellian equilibrium (Theorem~\ref{thm:equilibrium_vpl}), the true entropy production must vanish at long times; SBTM's decay toward zero is the physically correct behavior, while the blob method's estimated entropy production remains elevated due to persistent score estimation errors in low-density regions.  At $n = 3 \times 10^6$ the blob prediction moves closer to SBTM but remains far from the SBTM curve, indicating that blob convergence requires substantially more particles.  Although neither method conserves energy exactly at the fully discrete level -- the forward Euler time stepping and spatial discretization both introduce energy errors -- the blob method exhibits visibly larger energy drift than SBTM (plot (b)).

\begin{figure}[h]
\centering
\begin{subfigure}{0.48\linewidth}
\includegraphics[width=\linewidth]{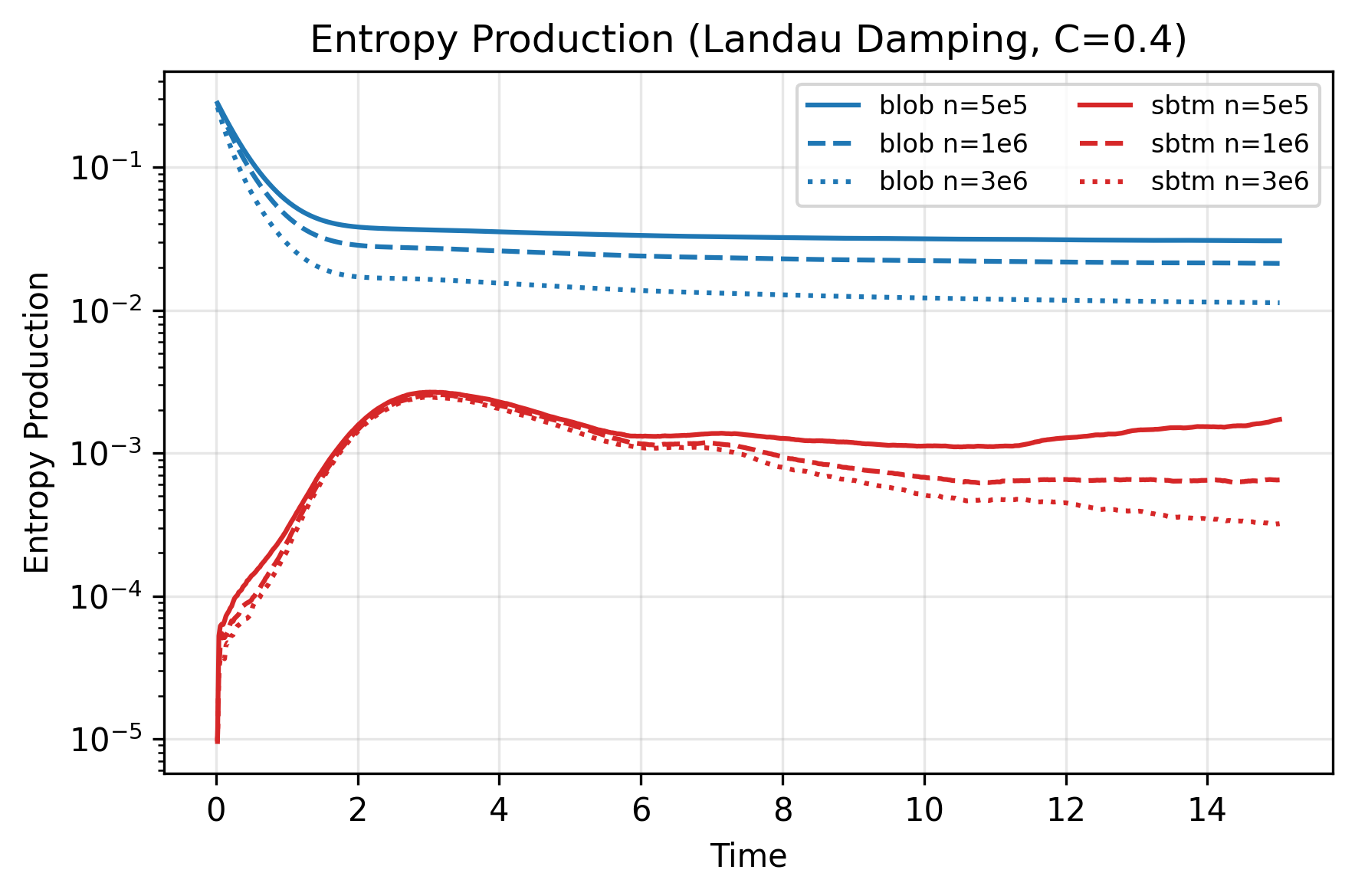}
\caption{Estimated entropy production (log scale)}
\end{subfigure}
\hfill
\begin{subfigure}{0.48\linewidth}
\includegraphics[width=\linewidth]{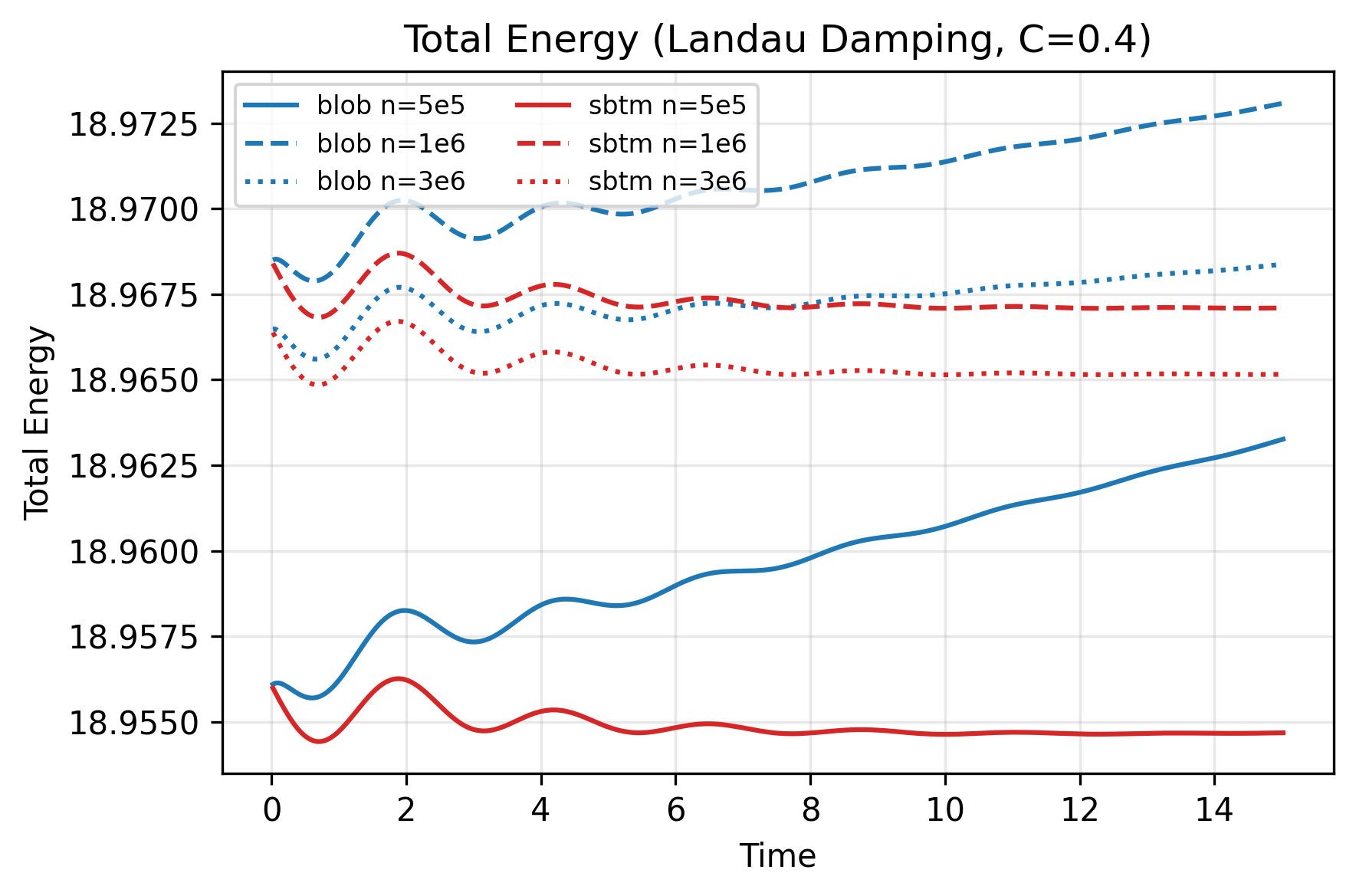}
\caption{Total energy (kinetic $+$ electromagnetic)}
\end{subfigure}
\caption{Landau damping at $\nu = 0.4$: estimated entropy production and total energy across particle counts.  SBTM (red) dissipates estimated entropy consistently and maintains near-constant total energy.  The blob method (blue) overestimates estimated entropy production and exhibits energy drift.}
\label{fig:app_landau_entropy_energy}
\end{figure}

\textbf{Electric field decay.}  Figure~\ref{fig:app_landau_efield} shows the $L^2$ norm of the electric field over time for the blob method and SBTM, along with the computed decay rate and the linear theory damping rate (plotted for reference). At this collision frequency ($\nu=0.4$), the theoretical decay rate is not known. However, we observe that the SBTM rate remains nearly constant as $n$ varies from $5 \times 10^5$ to $3 \times 10^6$, whereas the blob method's rate changes with $n$, approaching the SBTM rate at higher $n$. This demonstrates that SBTM converges faster with a relatively small number of particles, while the blob method requires significantly more particles to achieve convergence.

\begin{figure}[h]
\centering
\begin{subfigure}{0.48\linewidth}
\includegraphics[width=\linewidth]{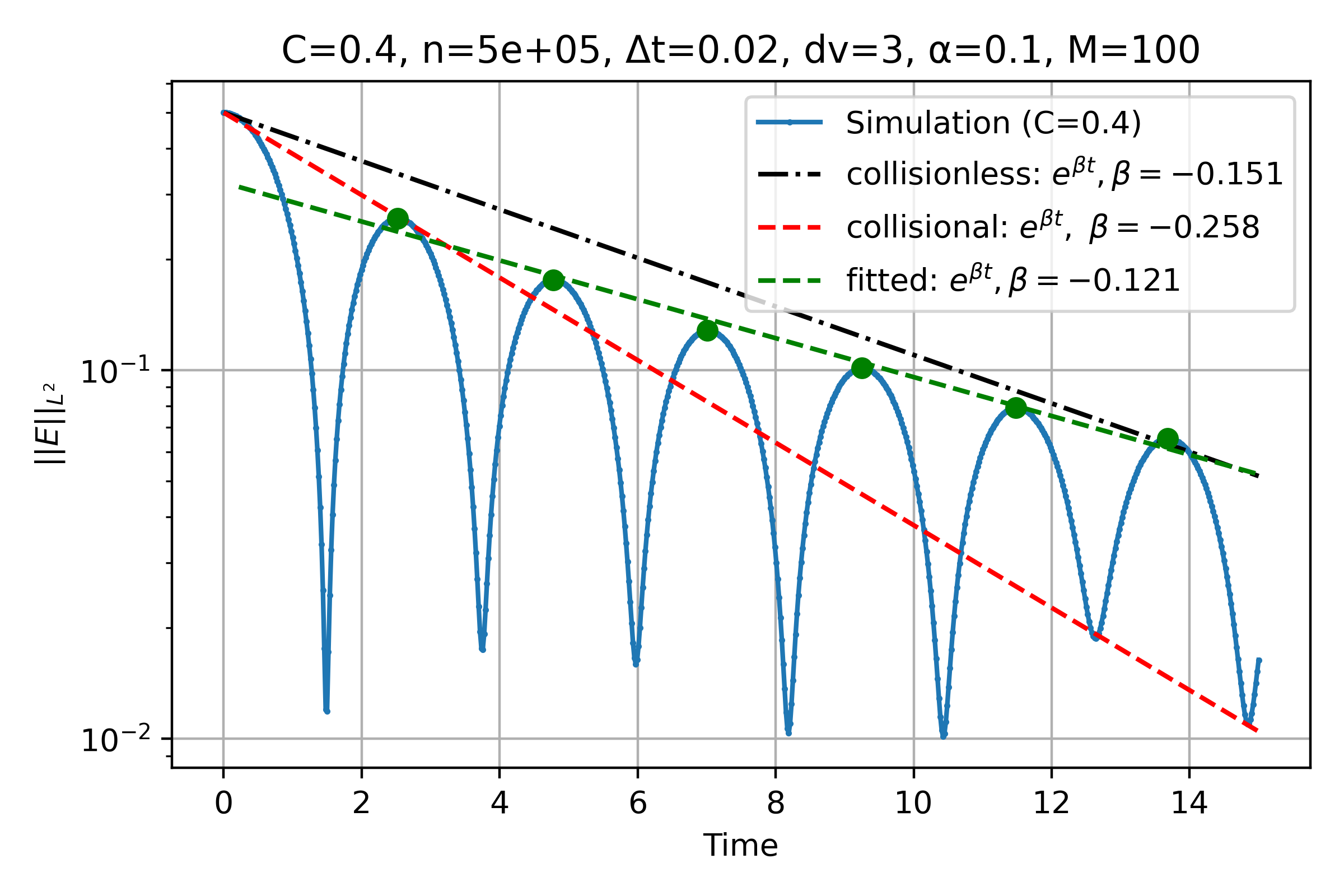}
\caption{Blob (KDE), $n = 5 \times 10^5$}
\end{subfigure}
\hfill
\begin{subfigure}{0.48\linewidth}
\includegraphics[width=\linewidth]{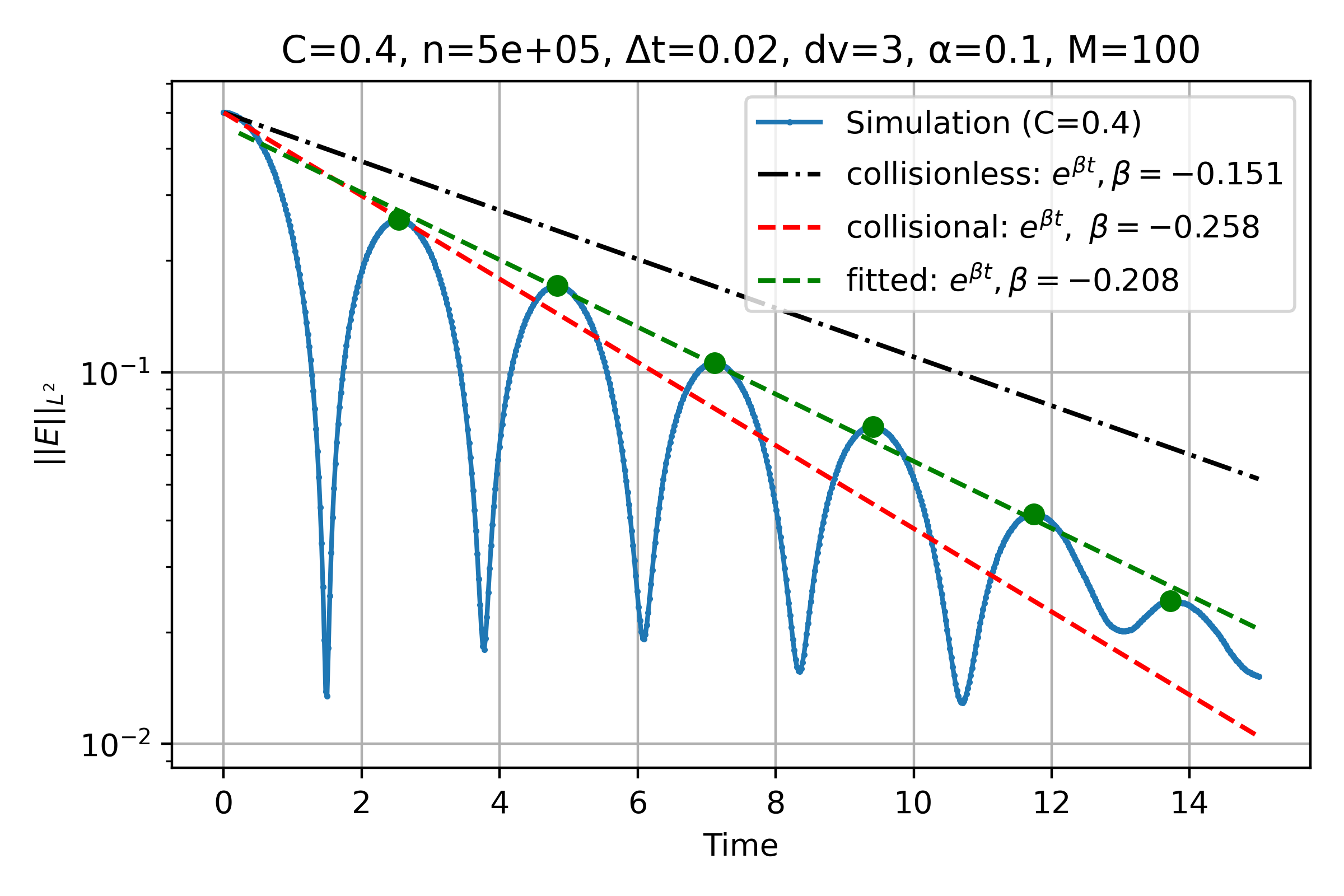}
\caption{SBTM, $n = 5 \times 10^5$}
\end{subfigure}
\begin{subfigure}{0.48\linewidth}
\includegraphics[width=\linewidth]{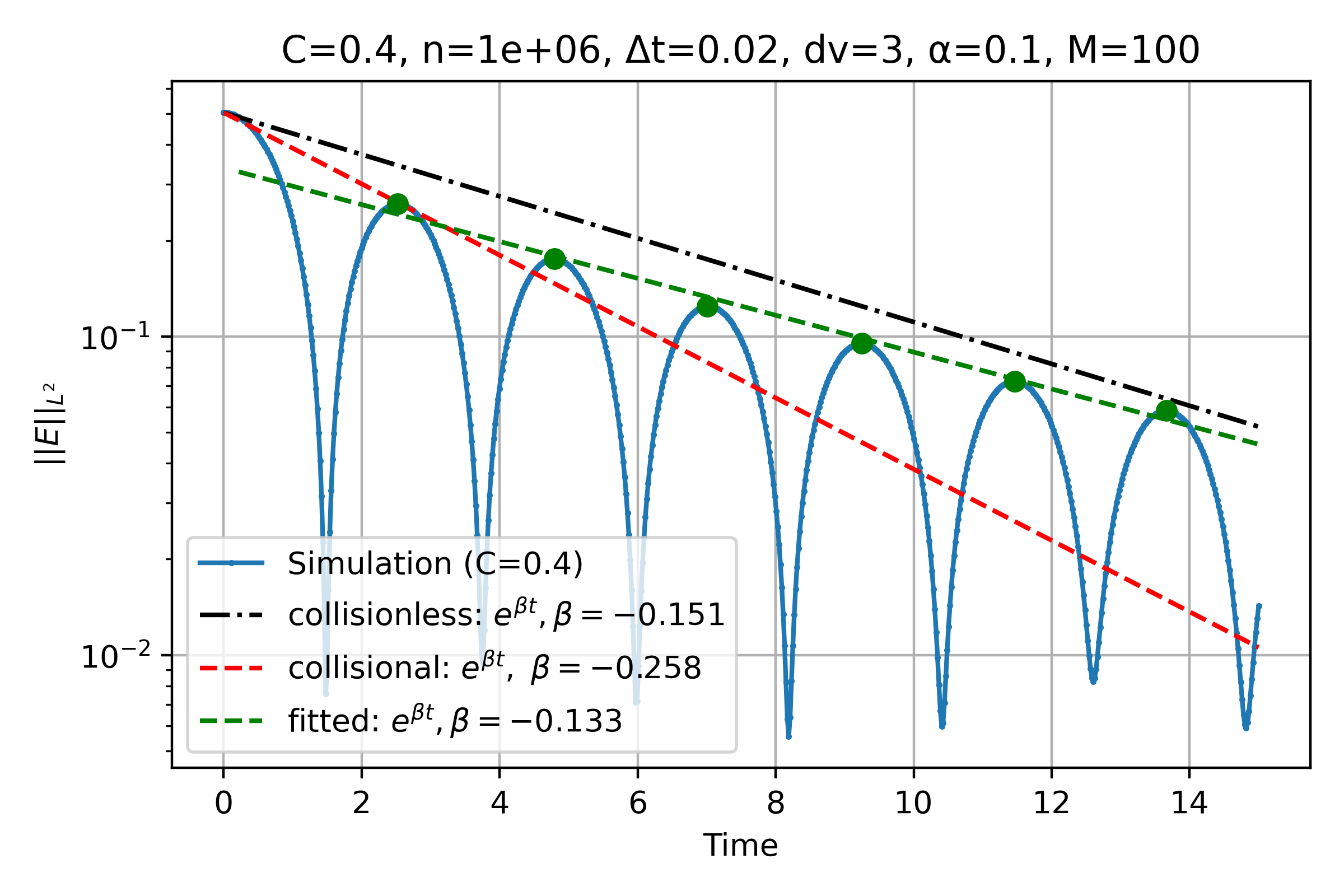}
\caption{Blob (KDE), $n = 10^6$}
\end{subfigure}
\hfill
\begin{subfigure}{0.48\linewidth}
\includegraphics[width=\linewidth]{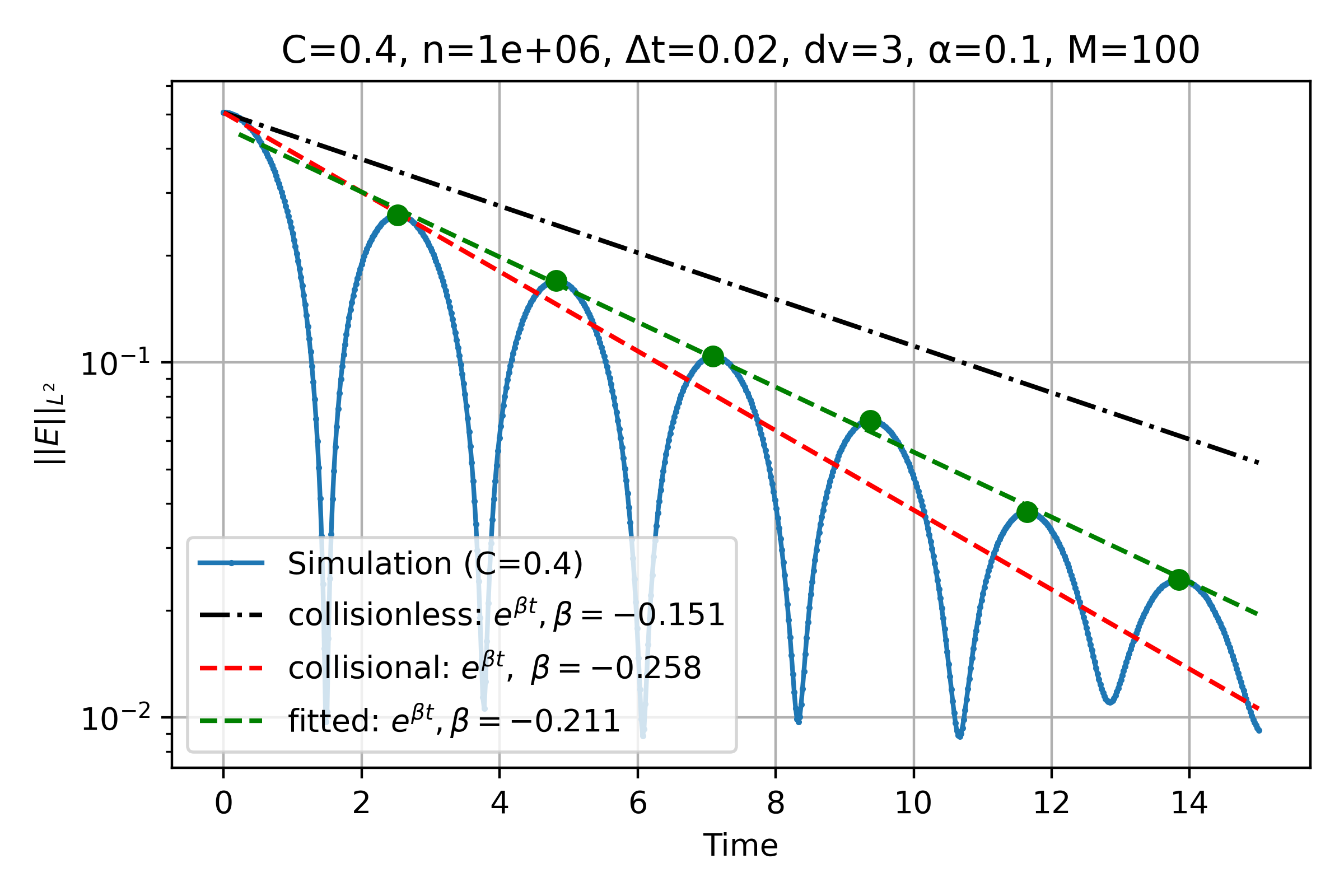}
\caption{SBTM, $n = 10^6$}
\end{subfigure}
\begin{subfigure}{0.48\linewidth}
\includegraphics[width=\linewidth]{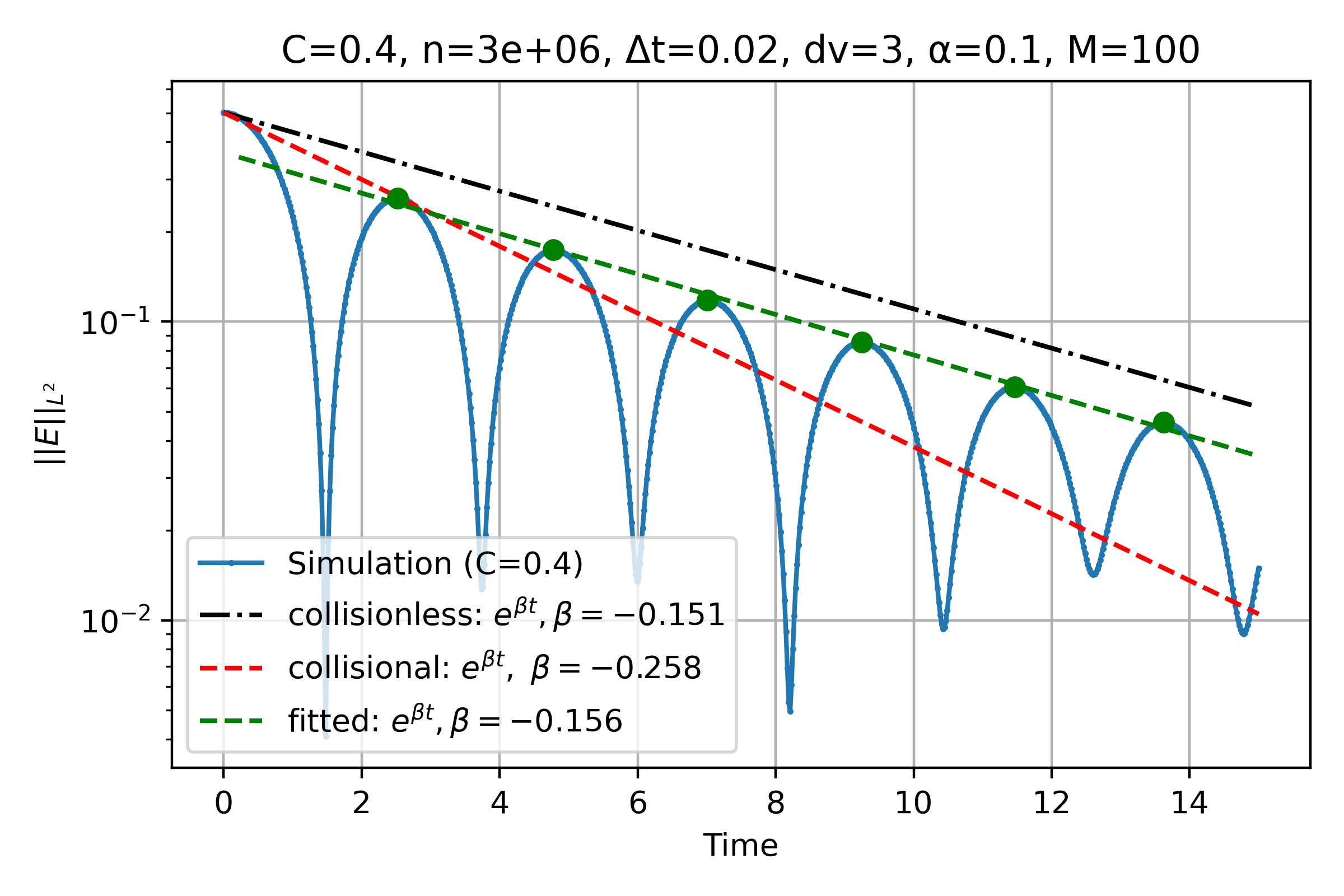}
\caption{Blob (KDE), $n = 3 \times 10^6$}
\end{subfigure}
\hfill
\begin{subfigure}{0.48\linewidth}
\includegraphics[width=\linewidth]{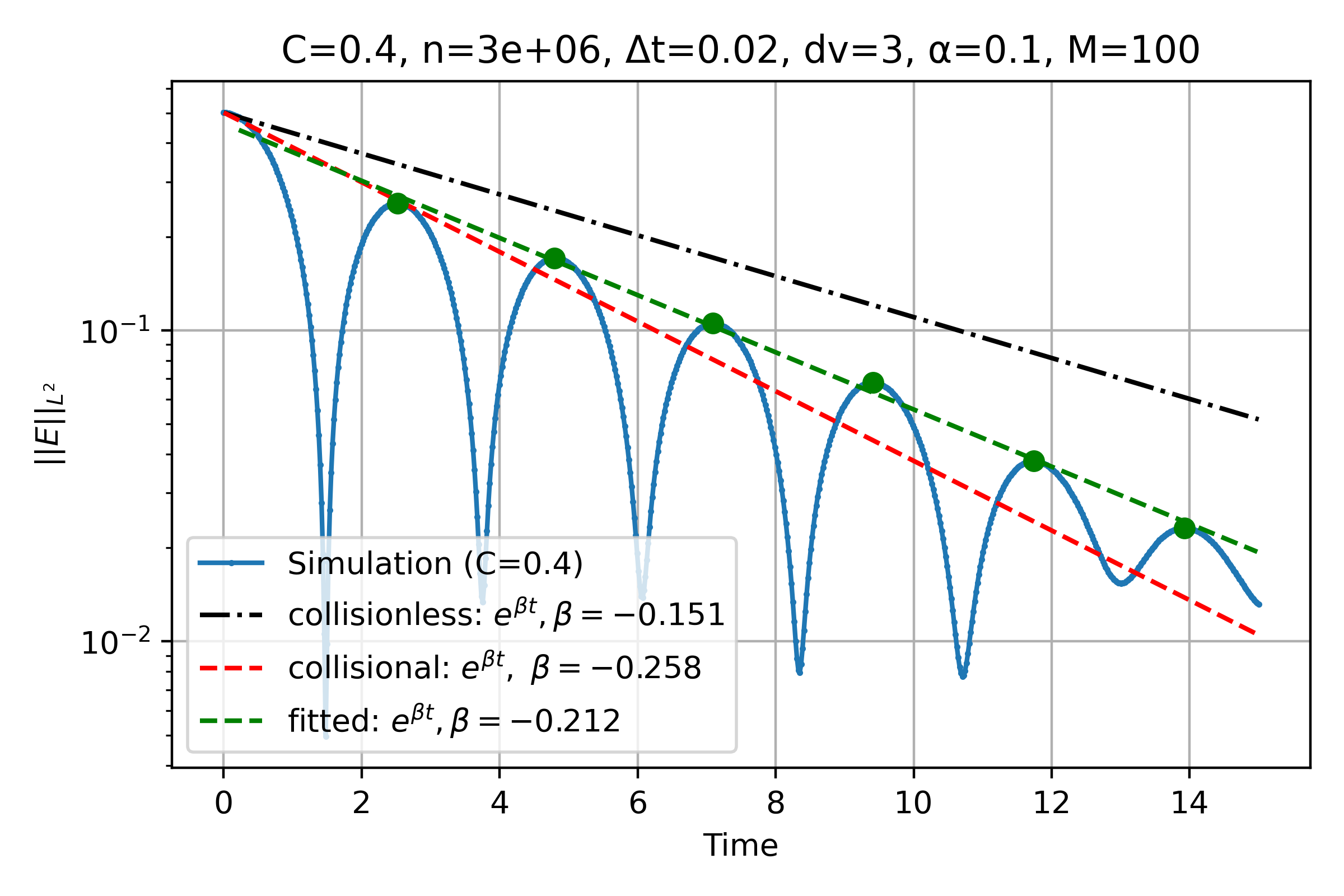}
\caption{SBTM, $n = 3 \times 10^6$}
\end{subfigure}
\caption{Landau damping: $L^2$ norm of the electric field over time at $\nu = 0.4$.  The linear theory damping rate is shown for reference (at $\nu = 0.4$ it is not expected to be accurate).  The blob method's damping rate varies with $n$ (compare (a), (c), and (e)), converging toward the SBTM rate.  The SBTM rate is consistent across all particle counts (b), (d), (f).}
\label{fig:app_landau_efield}
\end{figure}

\textbf{Velocity space distribution and equilibration.}  Figure~\ref{fig:landau_v1v2} shows $(v_1, v_2)$ phase space heatmaps and $v_2$-marginal densities at $n = 10^6$ for $\nu \in \{0.4, 1.0\}$.  SBTM produces smooth Gaussian tails for both collision frequencies, while the blob method exhibits sharp cutoffs in low-density regions. By Theorem~\ref{thm:equilibrium_vpl}, the system relaxes to a Maxwellian at temperature $T_\infty$ determined by total energy conservation.  The initial electric field energy is $\mathcal{E}_E = \frac{\alpha^2 L}{4k^2} \approx 0.13$, while the initial kinetic energy is $\mathcal{E}_K = \frac{d_v}{2} L \approx 18.85$.  Since $\mathcal{E}_E /\mathcal{E}_K= \alpha^2 / (2 d_v k^2) \approx 0.67\%$, the equilibrium temperature $T_\infty = 1 + \alpha^2/(2 d_v k^2) \approx 1.007$ is nearly indistinguishable from the initial temperature $T = 1$.  This is consistent with Figure~\ref{fig:landau_v1v2_C10} (d) and (h), where the $v_2$-marginal matches a standard Gaussian at all times.


\begin{figure}[h]
\centering
\begin{subfigure}{0.42\linewidth}
\includegraphics[width=\linewidth]{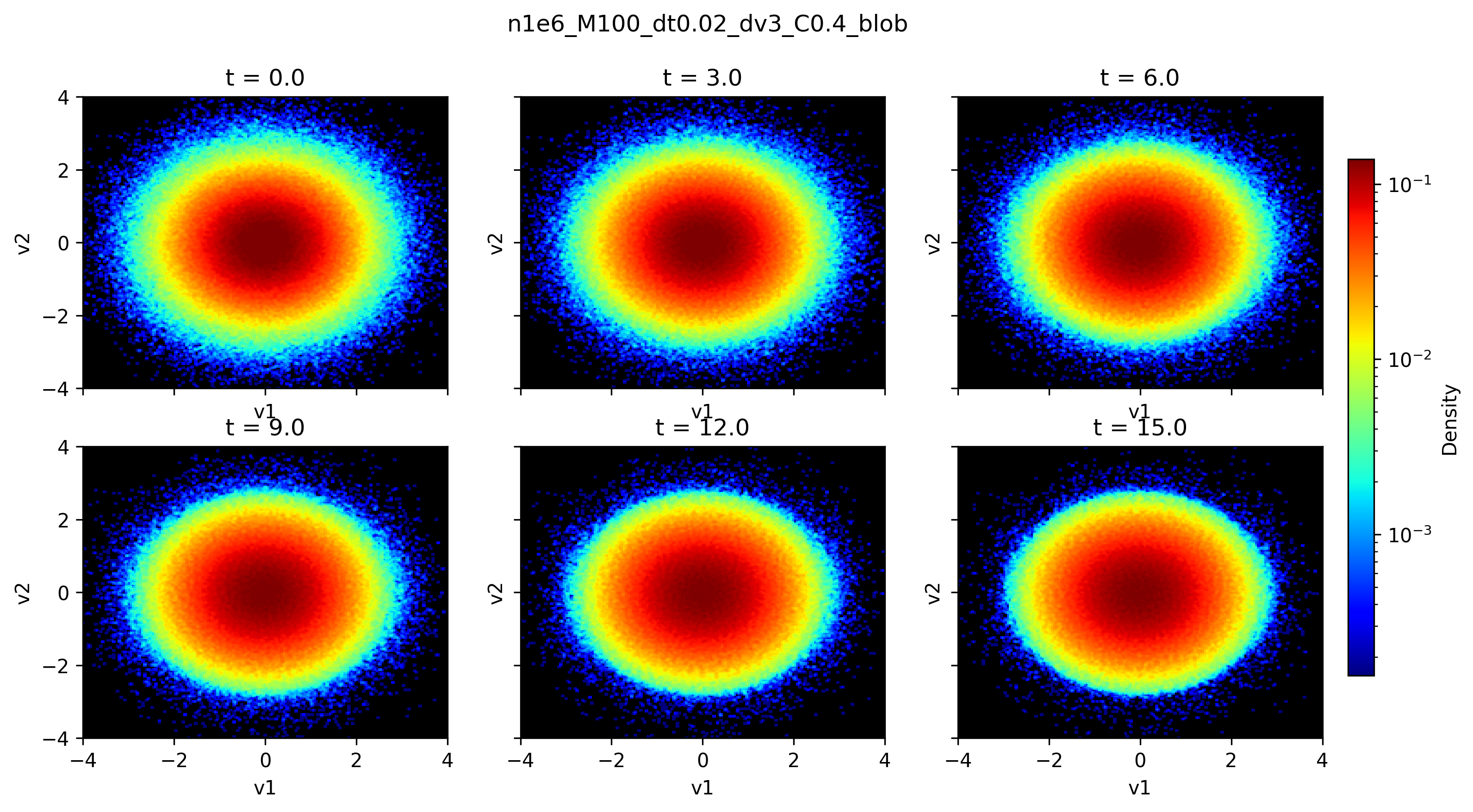}
\caption{Blob, $\nu = 0.4$}
\end{subfigure}
\hfill
\begin{subfigure}{0.42\linewidth}
\includegraphics[width=\linewidth]{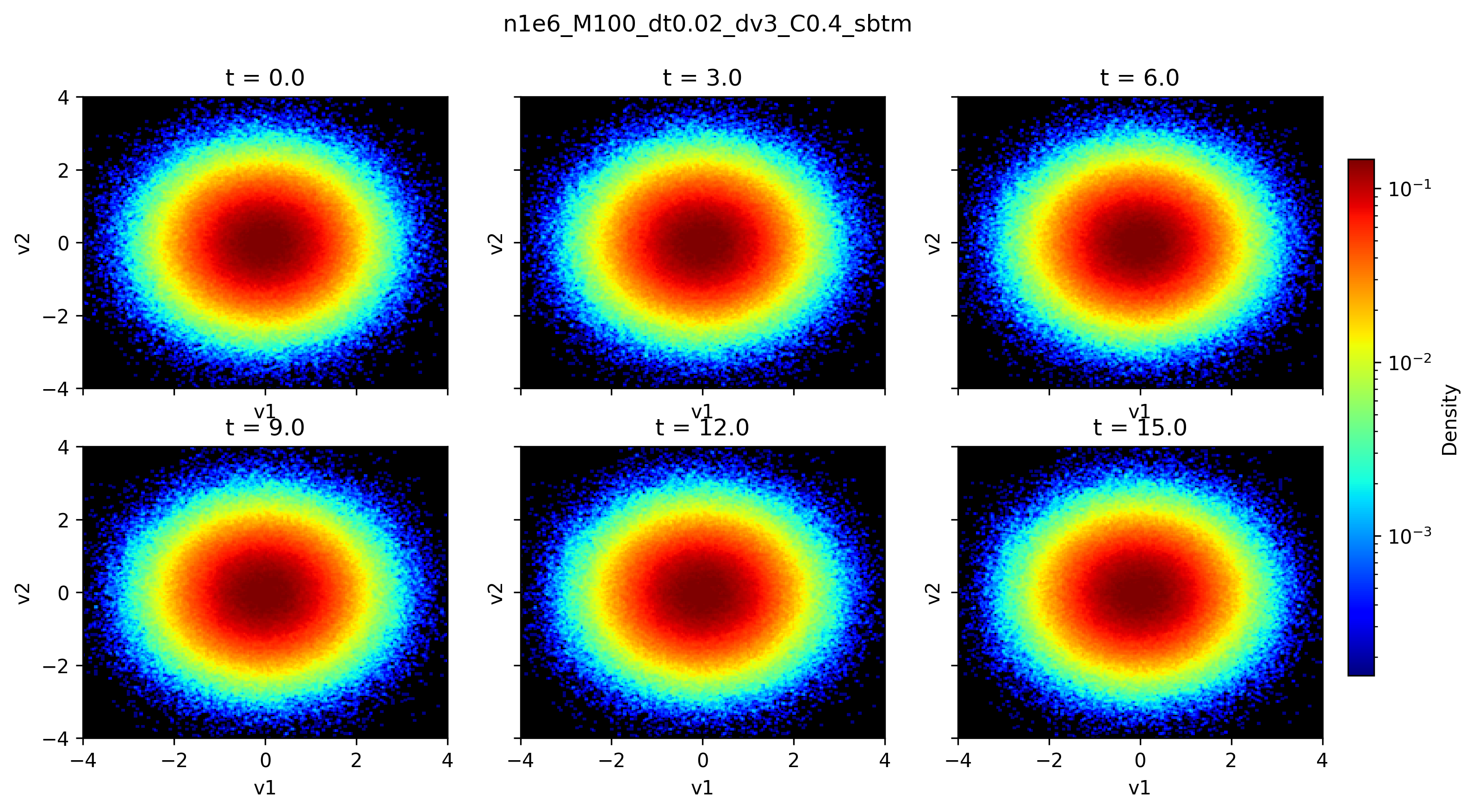}
\caption{SBTM, $\nu = 0.4$}
\end{subfigure}
\begin{subfigure}{0.42\linewidth}
\includegraphics[width=\linewidth]{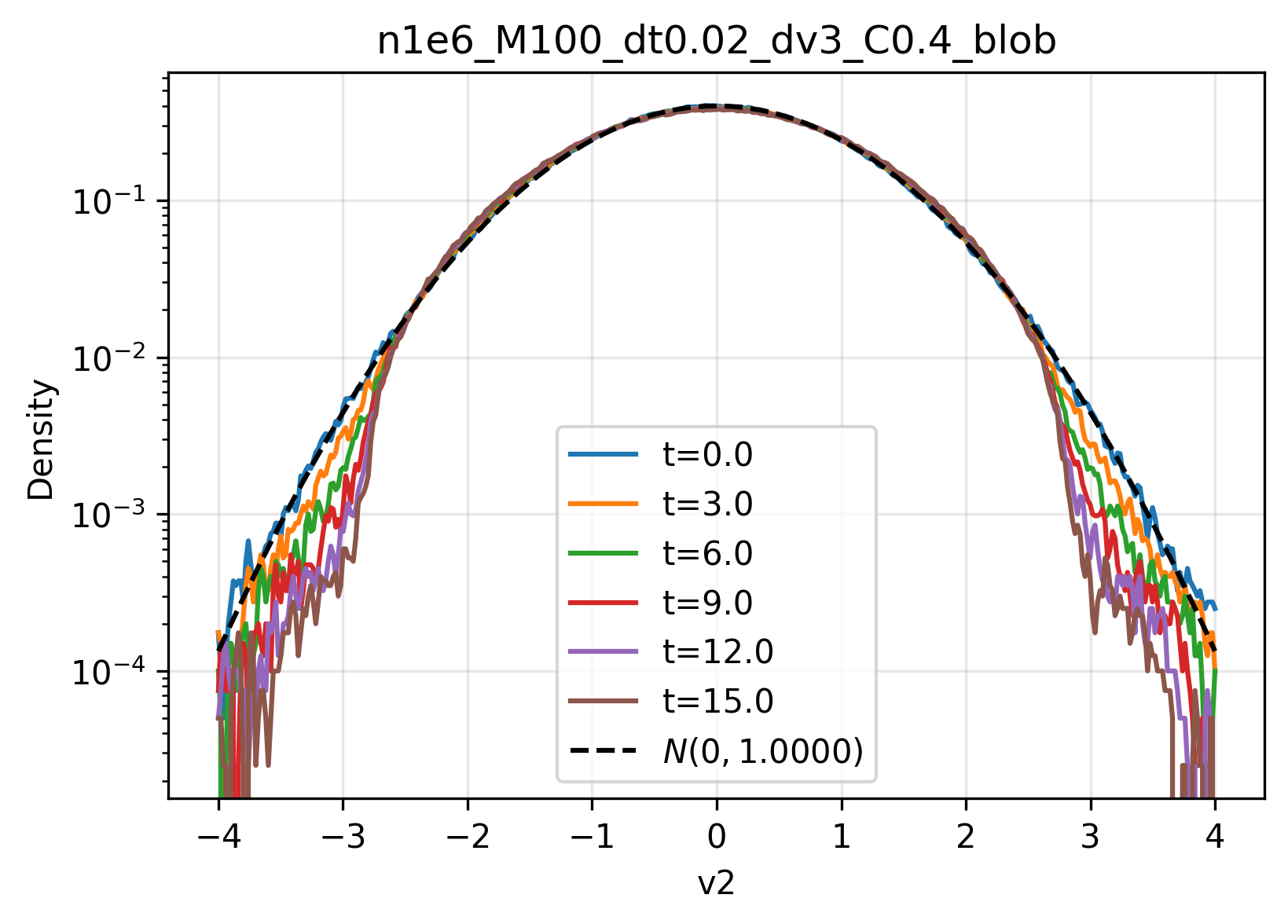}
\caption{Blob, $v_2$-marginal (log), $\nu = 0.4$}
\end{subfigure}
\hfill
\begin{subfigure}{0.42\linewidth}
\includegraphics[width=\linewidth]{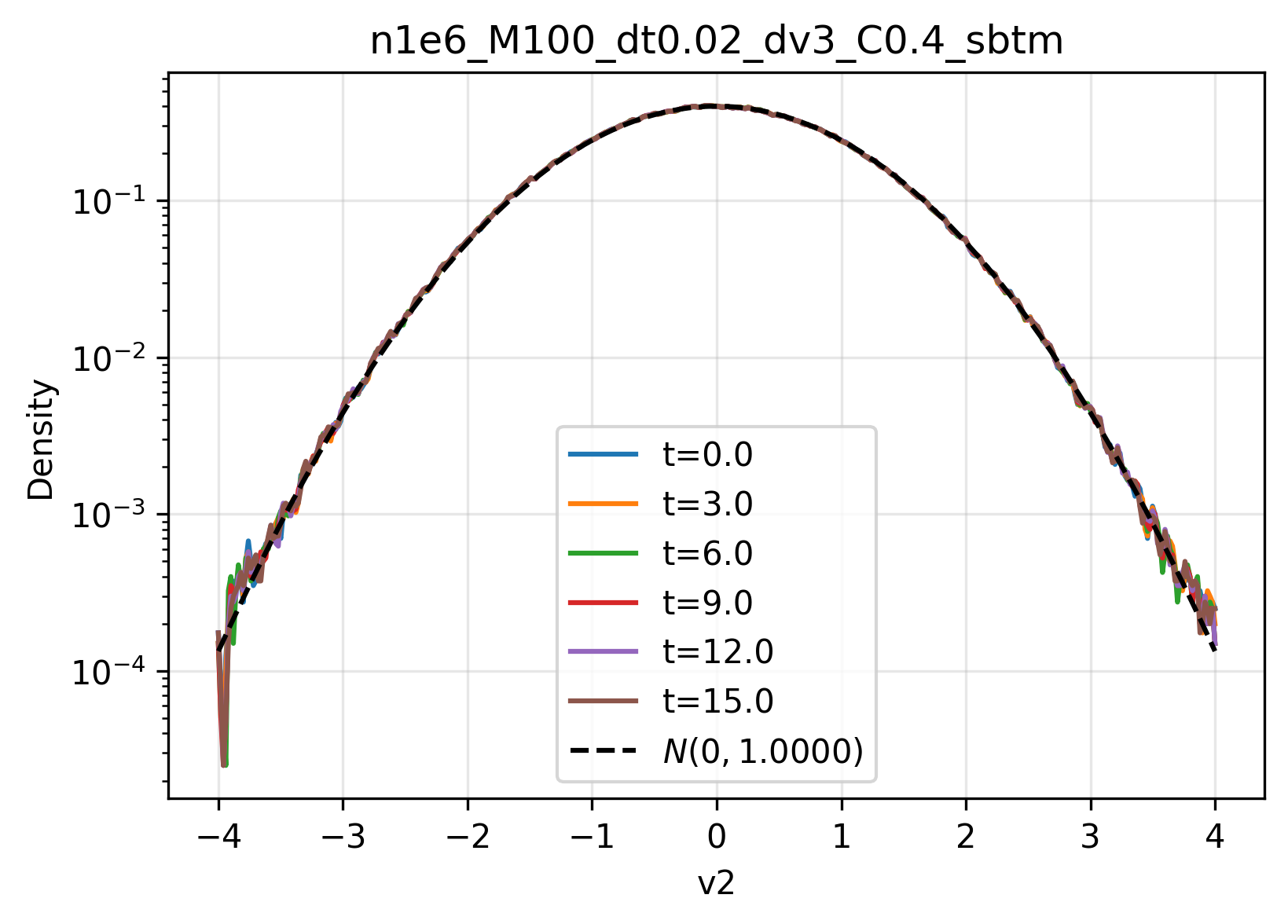}
\caption{SBTM, $v_2$-marginal (log), $\nu = 0.4$}
\end{subfigure}
\begin{subfigure}{0.42\linewidth}
\includegraphics[width=\linewidth]{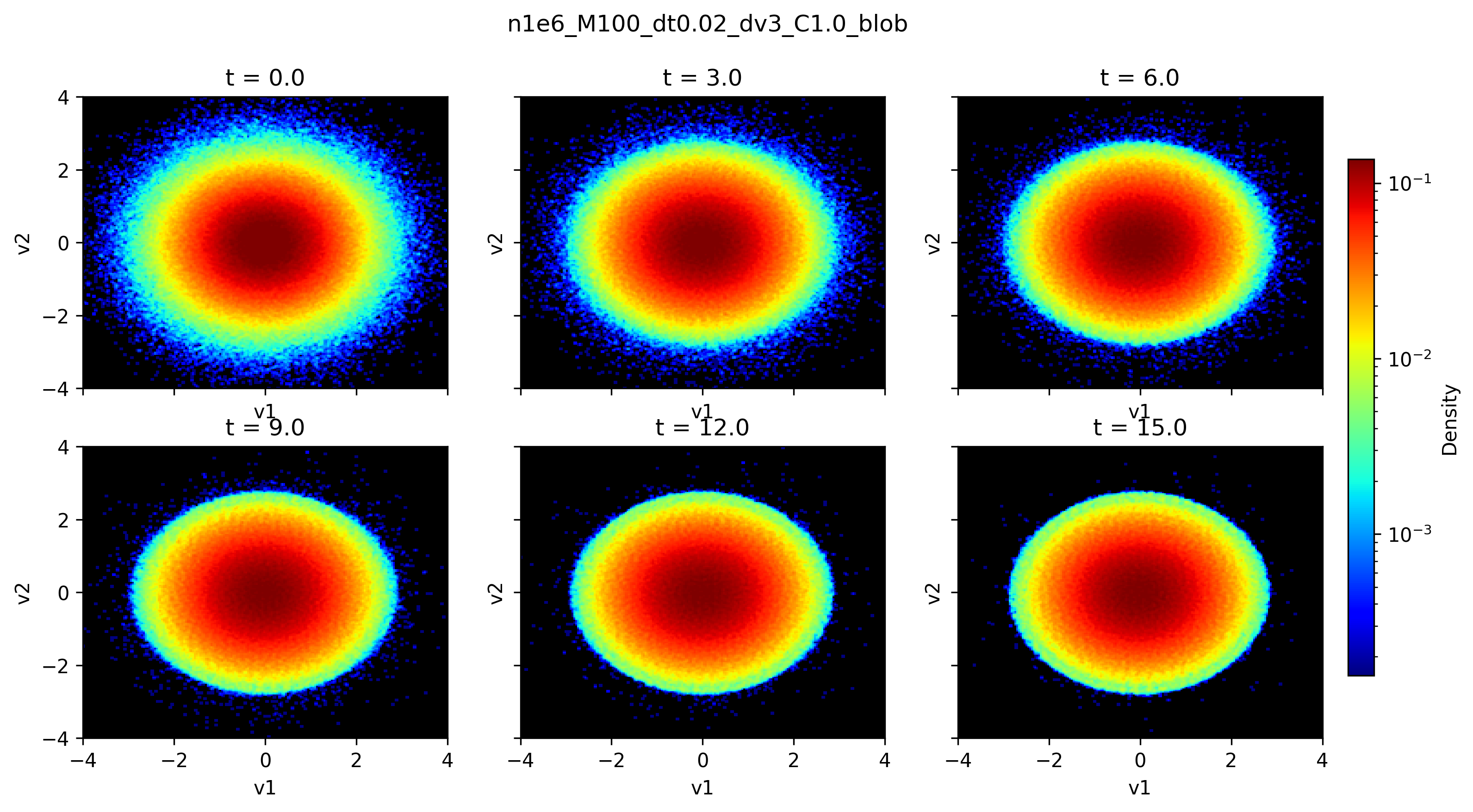}
\caption{Blob, $\nu = 1.0$}
\end{subfigure}
\hfill
\begin{subfigure}{0.42\linewidth}
\includegraphics[width=\linewidth]{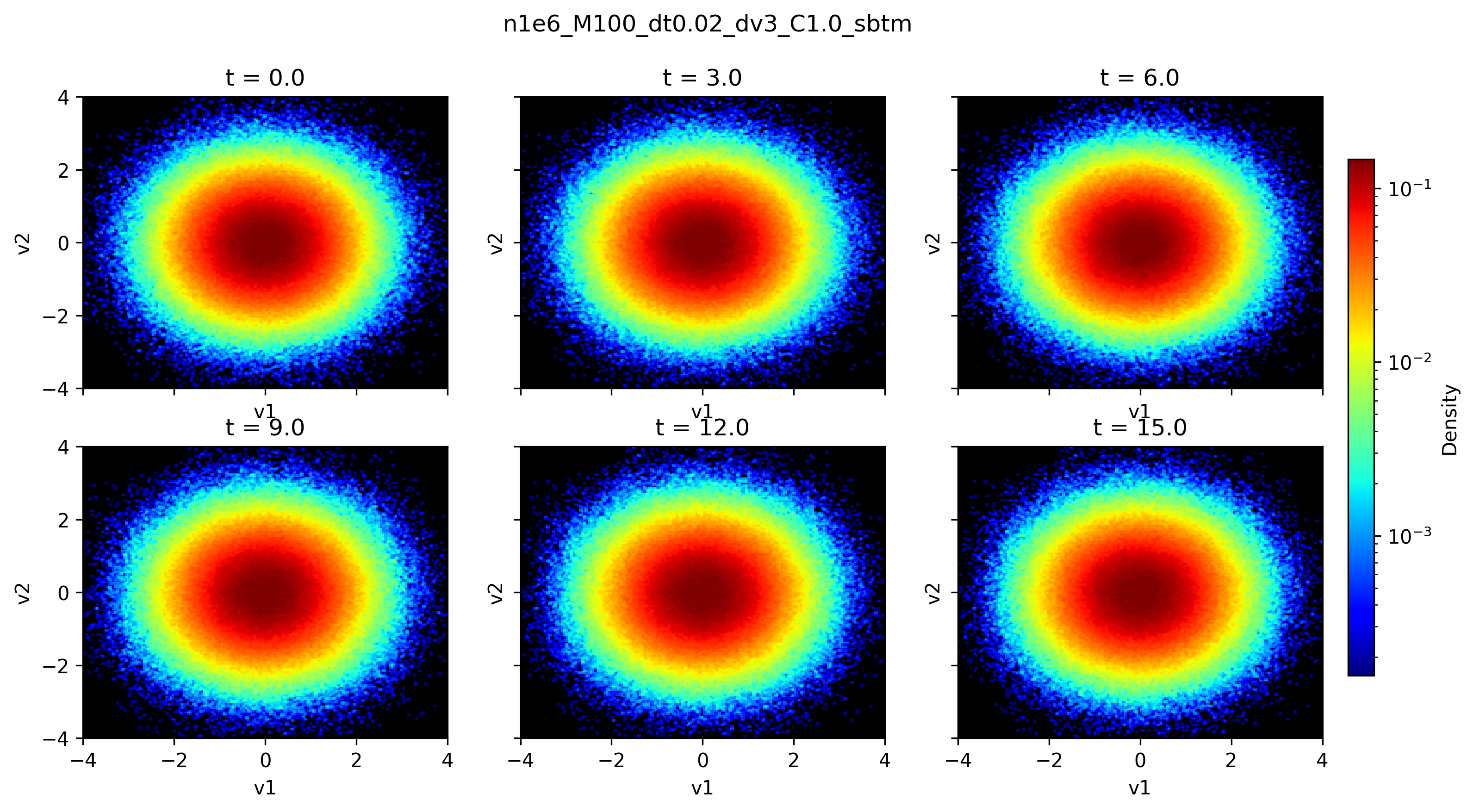}
\caption{SBTM, $\nu = 1.0$}
\end{subfigure}
\begin{subfigure}{0.42\linewidth}
\includegraphics[width=\linewidth]{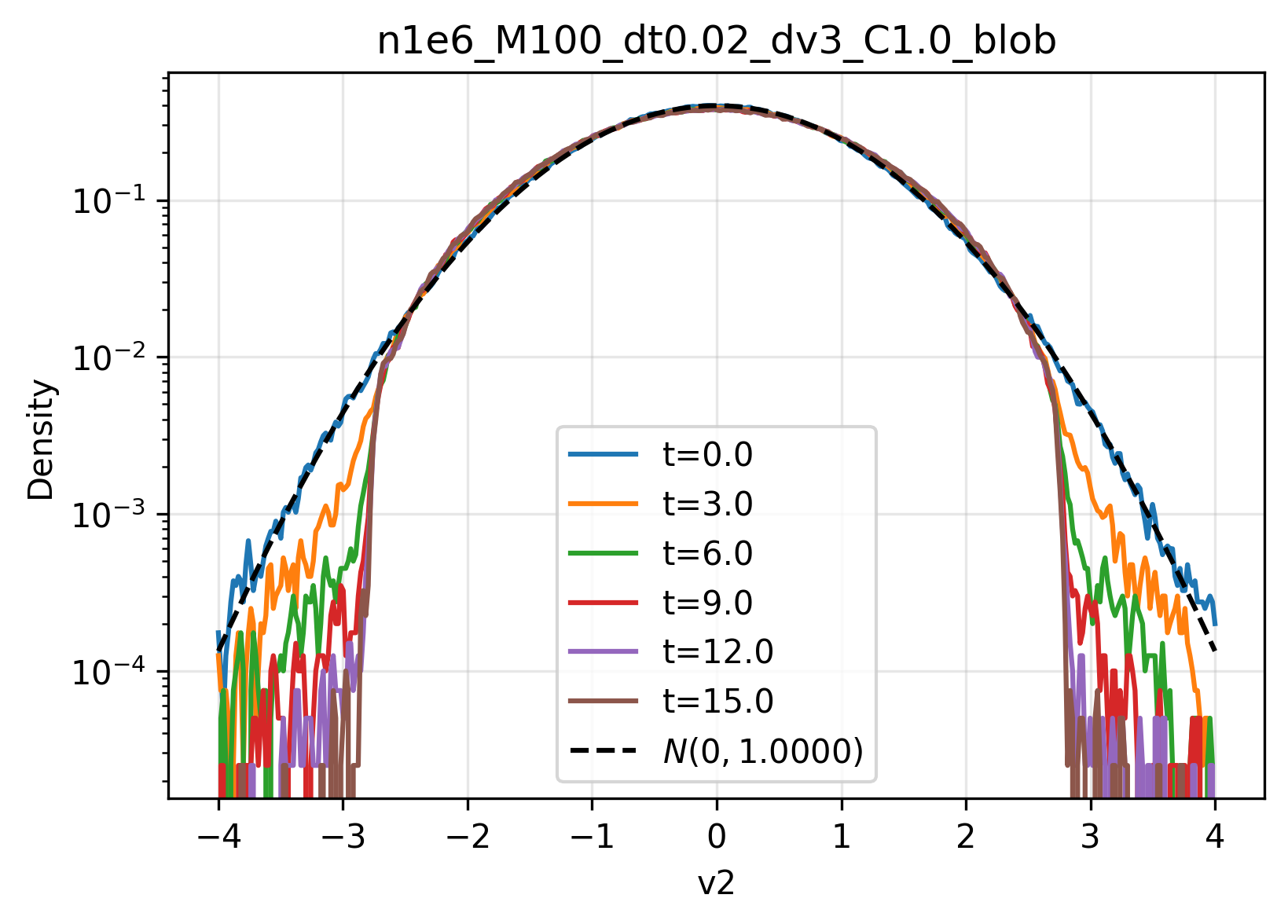}
\caption{Blob, $v_2$-marginal (log), $\nu = 1.0$}
\end{subfigure}
\hfill
\begin{subfigure}{0.42\linewidth}
\includegraphics[width=\linewidth]{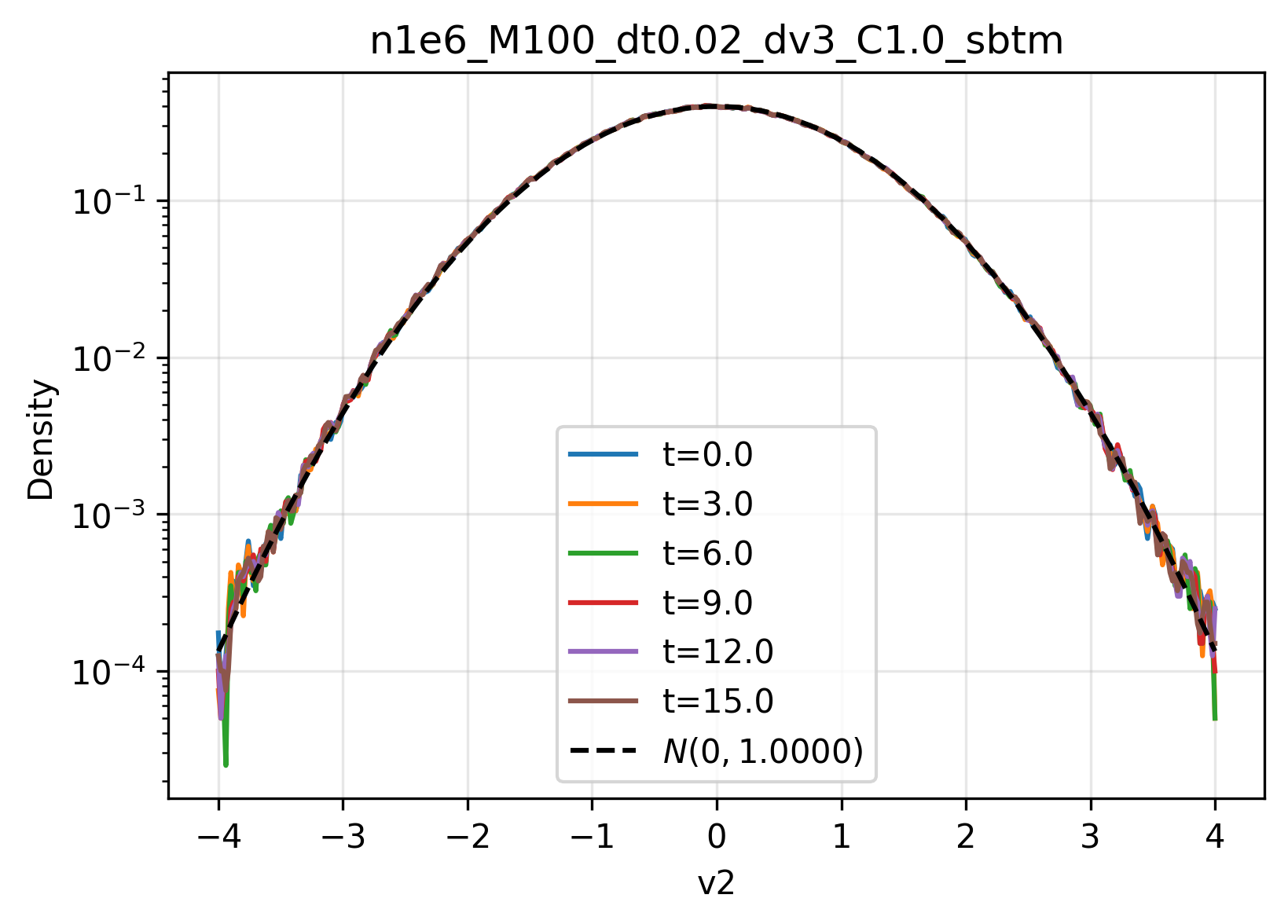}
\caption{SBTM, $v_2$-marginal (log), $\nu = 1.0$}
\end{subfigure}
\caption{Landau damping: velocity-space distributions at $n = 10^6$.  Rows 1--2: $\nu = 0.4$ (heatmaps and $v_2$-marginal).  Rows 3--4: $\nu = 1.0$.  SBTM maintains smooth Gaussian tails; the blob method shows unphysical sharp cutoffs at $|v_2| \approx 3$.}
\label{fig:landau_v1v2}
\label{fig:landau_v1v2_C10}
\end{figure}

\FloatBarrier

\subsection{Two-Stream Instability}
\label{sec:twostream}

\textbf{Setup.}  Two counter-streaming Maxwellian beams with a small spatial perturbation:
\begin{equation}
\label{eq:twostream_ic}
f_0(x, v) = \frac{1 + \alpha \cos(kx)}{2} \left[\mathcal{N}(v_1; c, 1) + \mathcal{N}(v_1; -c, 1) \right] \prod_{j \geq 2} \mathcal{N}(v_j; 0, 1),
\end{equation}
with beam speed $c = 2.4$, wavenumber $k = 1/5$, perturbation $\alpha = 1/200$, domain $L = 2\pi/k$, $M = 100$, $\Delta t = 0.05$, $d_v = 3$, $\nu = 0.32$, $t_{\mathrm{final}} = 50$, and $K = 100$ ISM steps per time step.  We test $n \in \{5 \times 10^5, 10^6, 3 \times 10^6\}$.

\textbf{Estimated entropy production and total energy.
} Figure~\ref{fig:twostream_diagnostics} shows the estimated entropy production
and total energy evolution.  As in the Landau damping case, SBTM dissipates estimated entropy consistently at all particle counts -- the SBTM curves nearly coincide (plot (d)) -- while the blob method overestimates entropy production, especially at lower $n$.  The blob method also introduces visibly larger energy drift than SBTM (plot (c)). 

\begin{figure}[h]
\centering
\begin{subfigure}{0.48\linewidth}
\includegraphics[width=\linewidth]{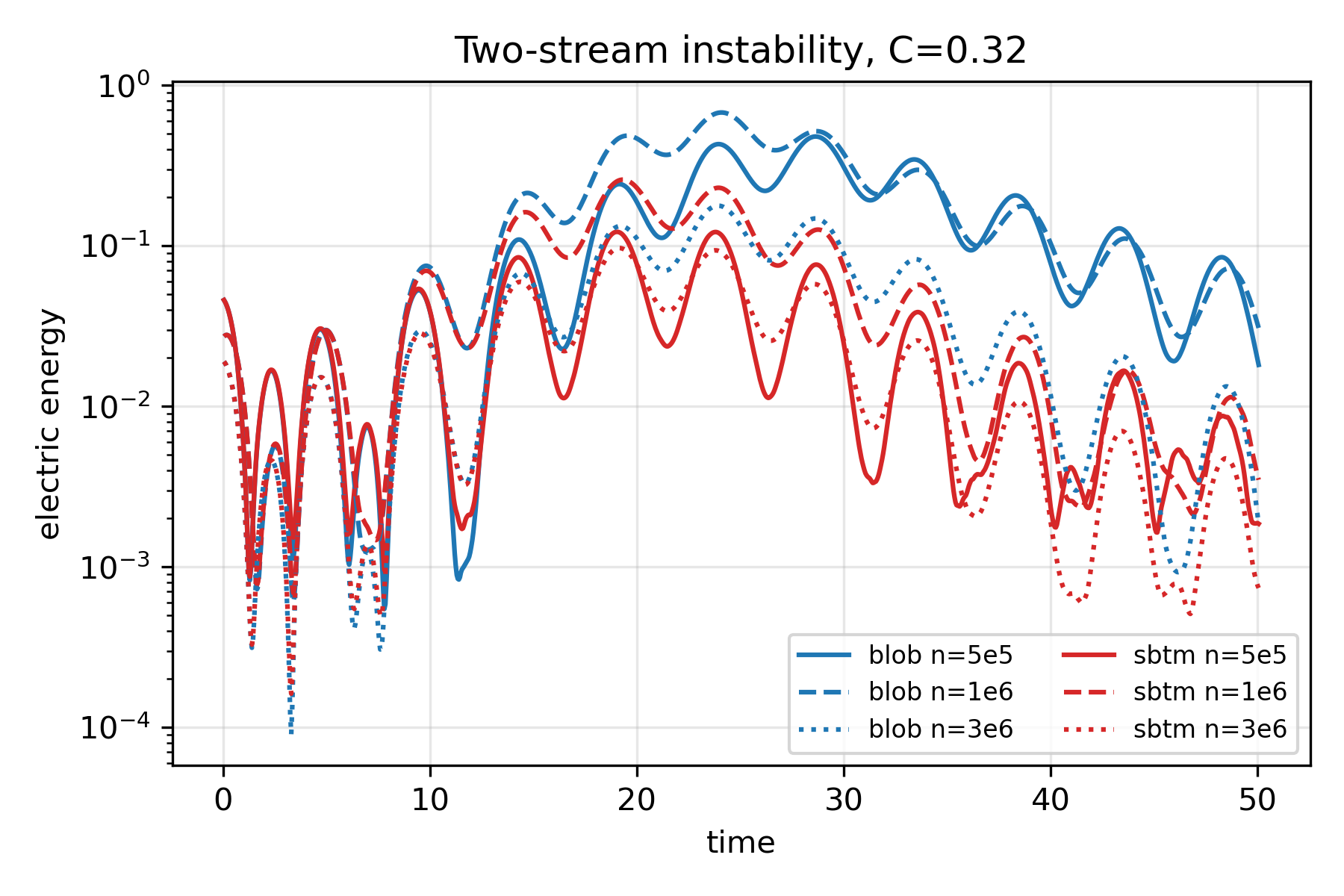}
\caption{Electric field energy}
\end{subfigure}
\hfill
\begin{subfigure}{0.48\linewidth}
\includegraphics[width=\linewidth]{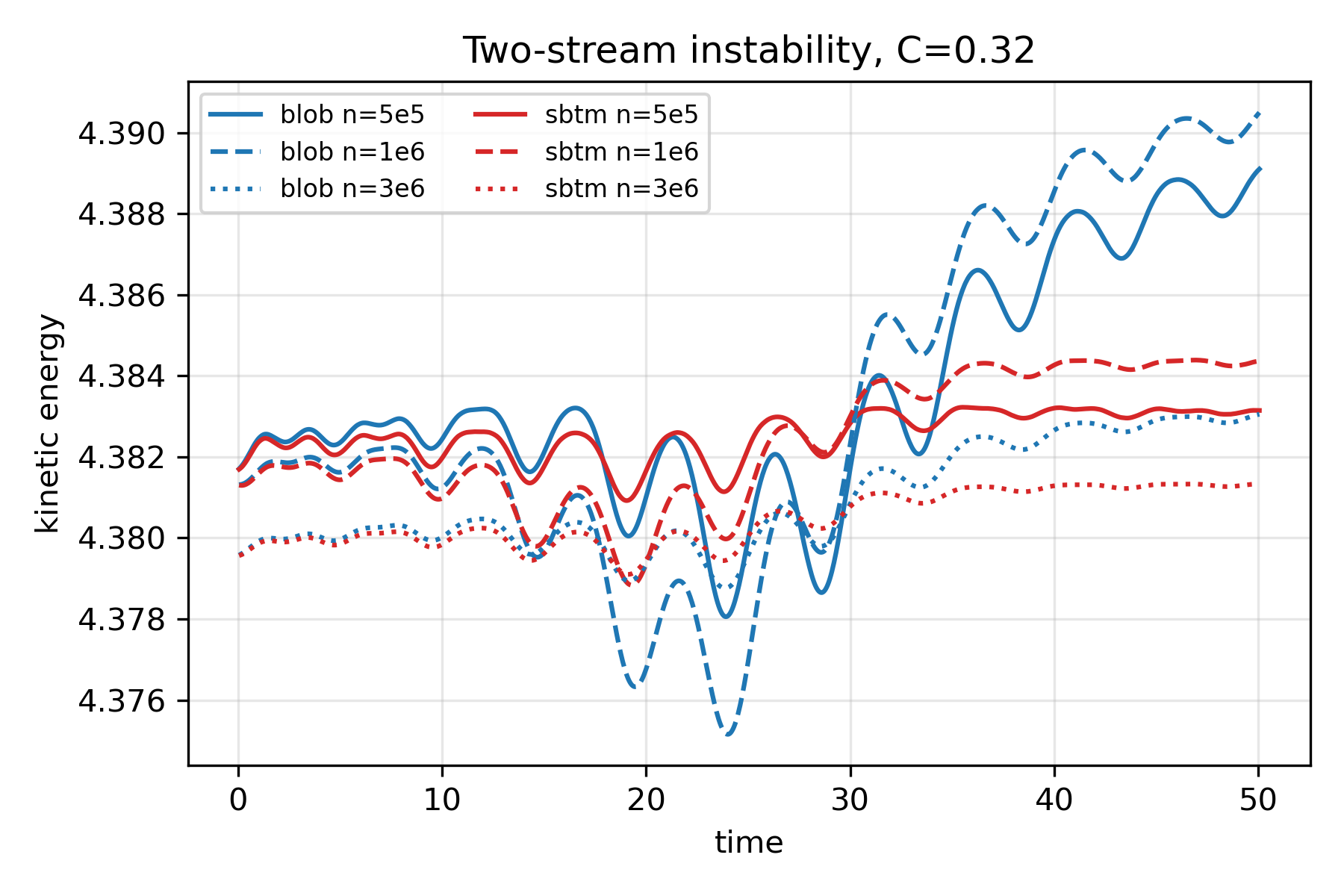}
\caption{Kinetic energy}
\end{subfigure}
\begin{subfigure}{0.48\linewidth}
\includegraphics[width=\linewidth]{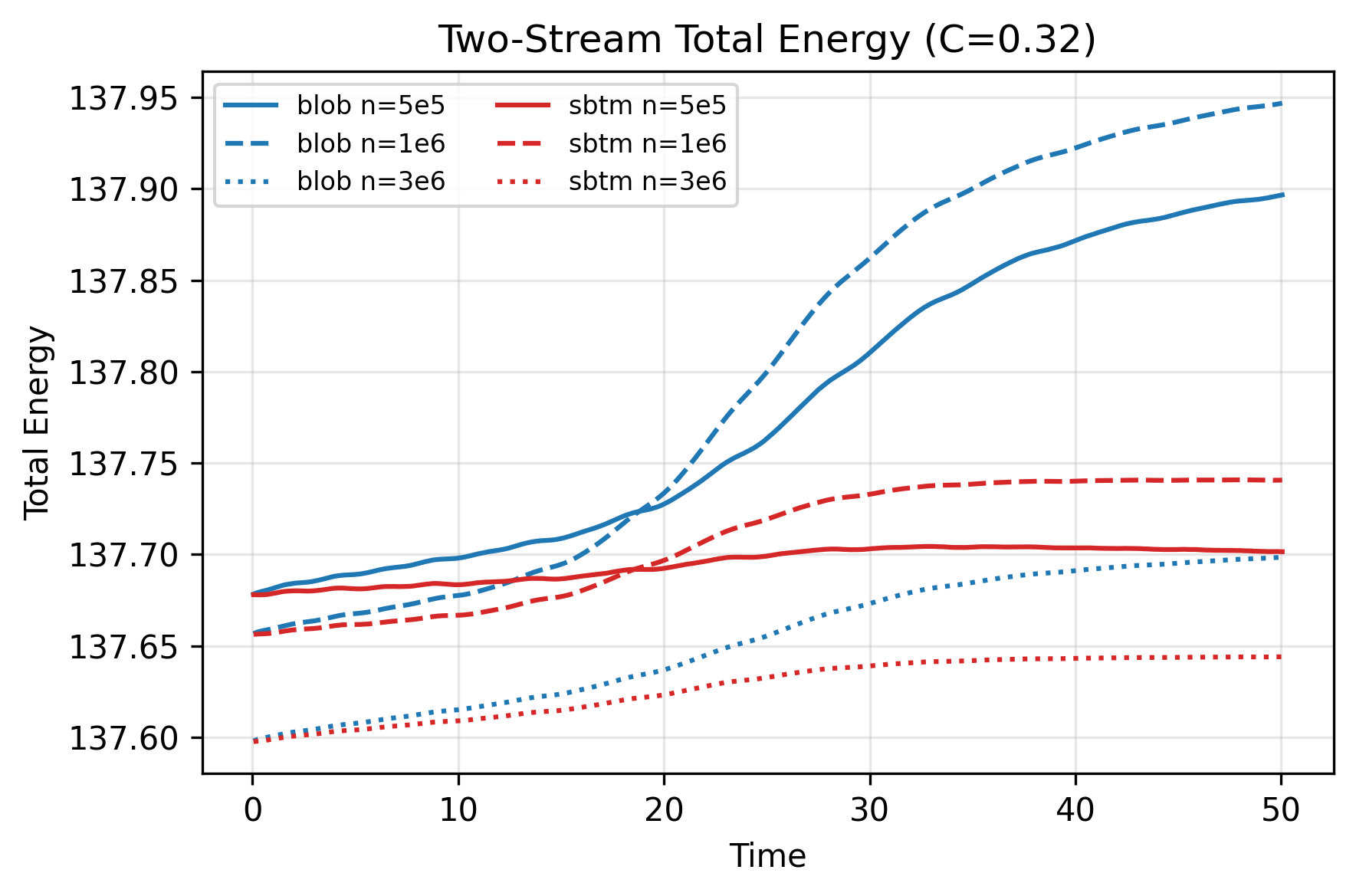}
\caption{Total energy}
\end{subfigure}
\hfill
\begin{subfigure}{0.48\linewidth}
\includegraphics[width=\linewidth]{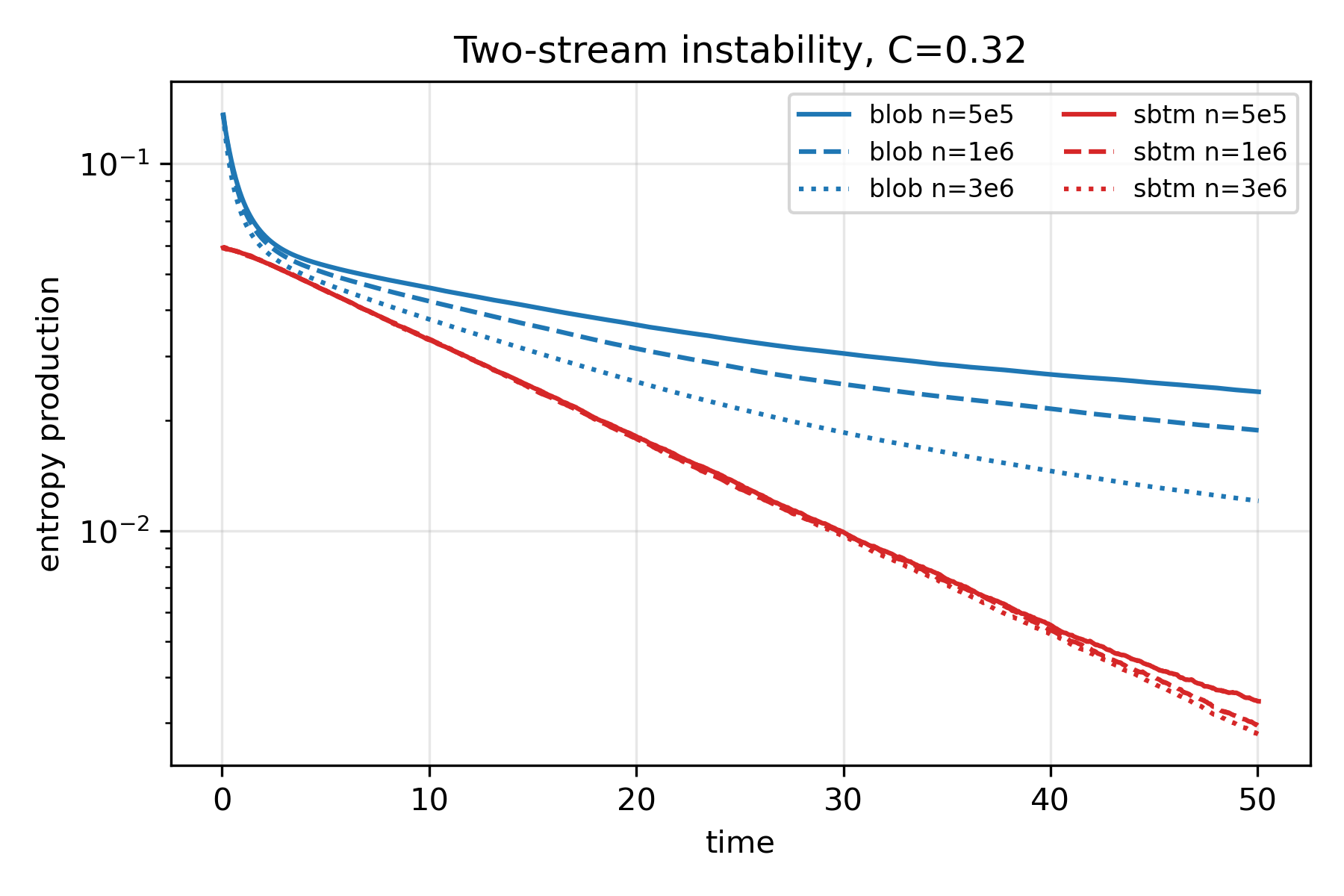}
\caption{Estimated entropy production}
\end{subfigure}
\caption{Two-stream instability: energy and estimated entropy diagnostics at $\nu = 0.32$.  Blob (blue) vs SBTM (red) at $n = 5 \times 10^5$ (solid), $10^6$ (dashed), $3 \times 10^6$ (dotted).}
\label{fig:twostream_diagnostics}
\end{figure}

\textbf{Phase space dynamics.}  Figure~\ref{fig:twostream_phase} compares $(x, v_1)$ phase space snapshots across methods and particle counts.  The phase space vortex formed by the two-stream instability should dissipate by $t = 50$.  SBTM achieves full vortex dissipation at all particle counts, while the blob method requires $n = 3 \times 10^6$ to approach the same level of dissipation.

\begin{figure}[t]
\centering
\begin{subfigure}{0.48\linewidth}
\includegraphics[width=\linewidth]{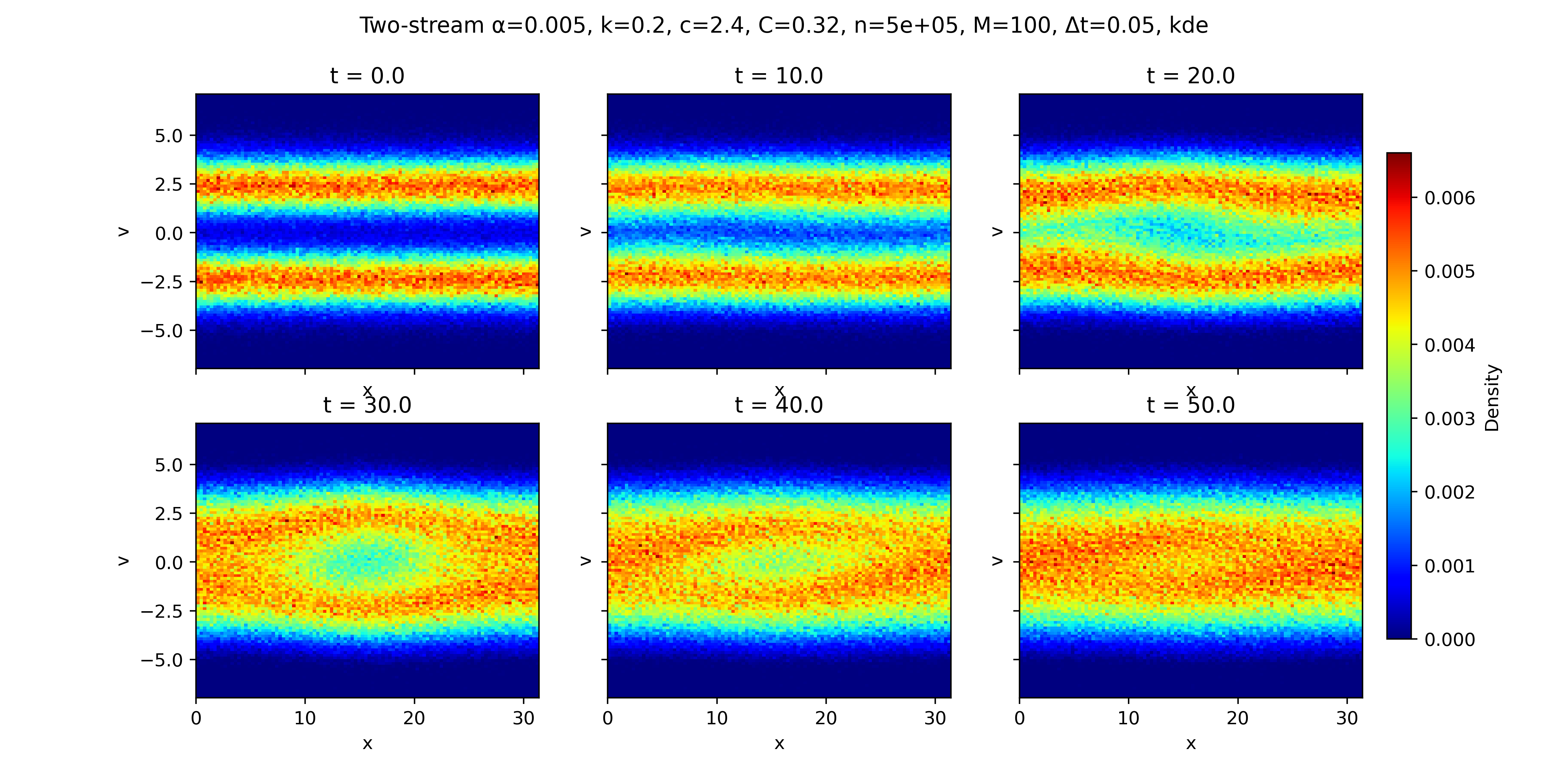}
\caption{Blob, $n = 5 \times 10^5$}
\end{subfigure}
\hfill
\begin{subfigure}{0.48\linewidth}
\includegraphics[width=\linewidth]{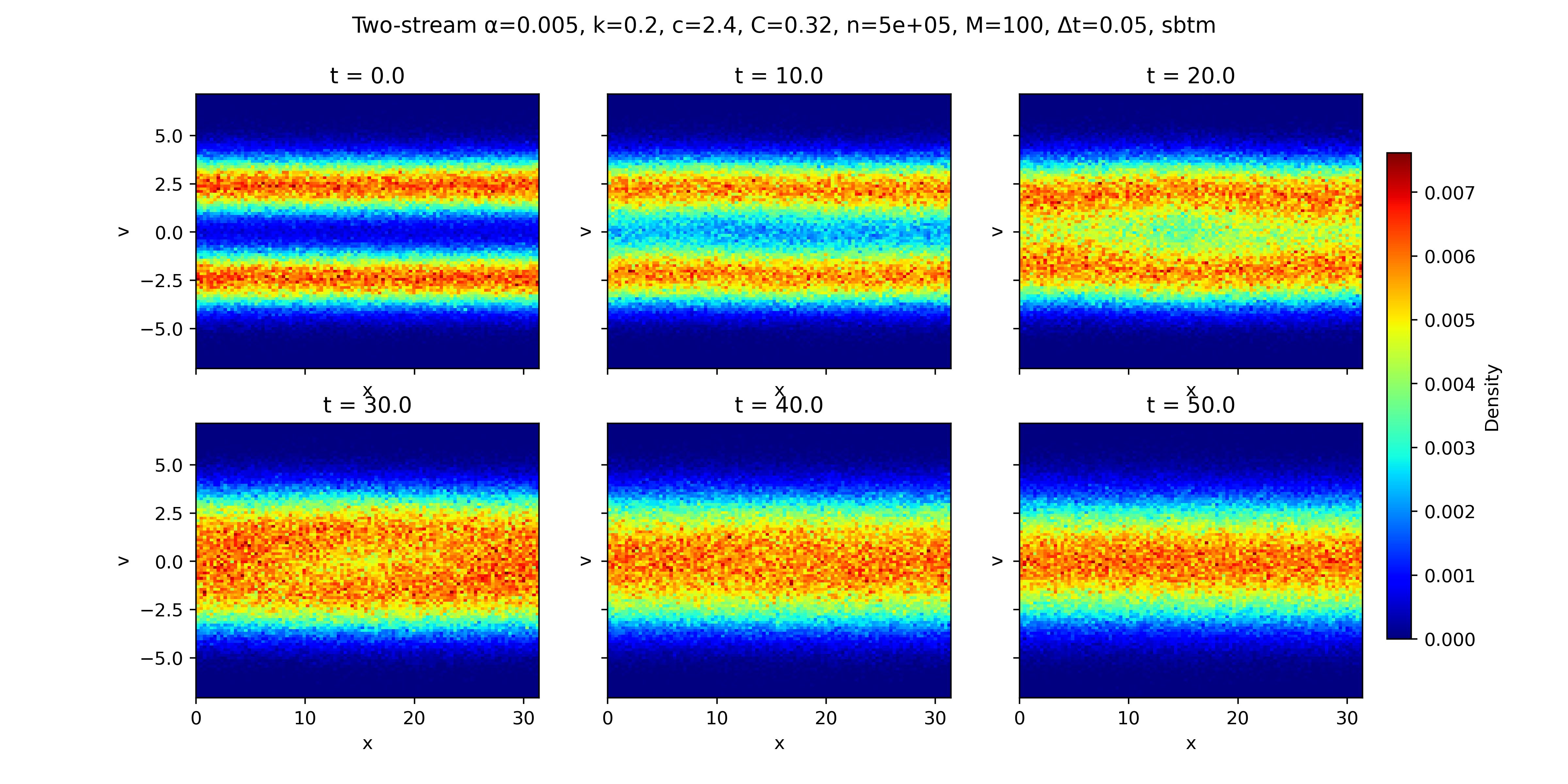}
\caption{SBTM, $n = 5 \times 10^5$}
\end{subfigure}
\begin{subfigure}{0.48\linewidth}
\includegraphics[width=\linewidth]{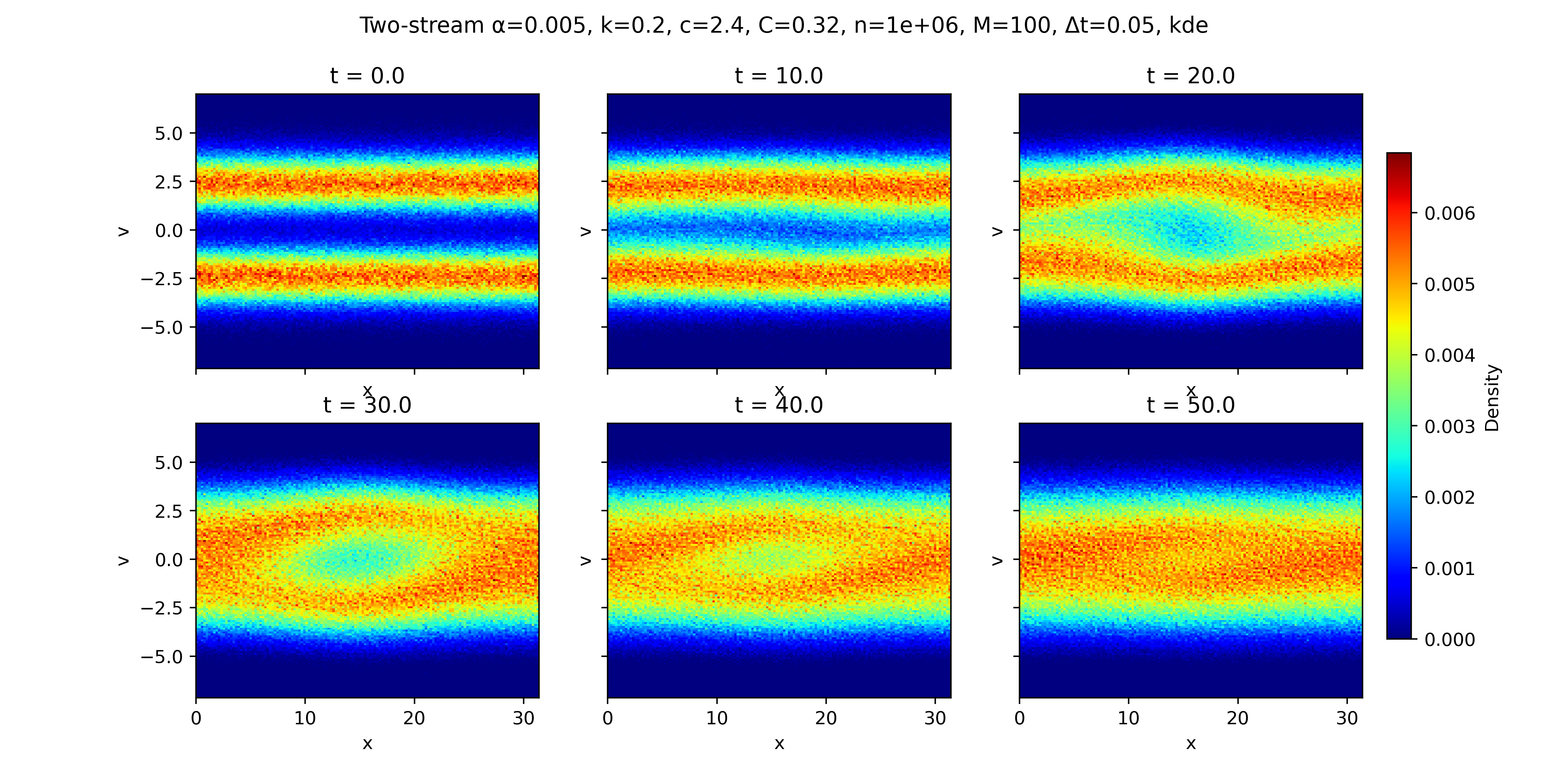}
\caption{Blob, $n = 10^6$}
\end{subfigure}
\hfill
\begin{subfigure}{0.48\linewidth}
\includegraphics[width=\linewidth]{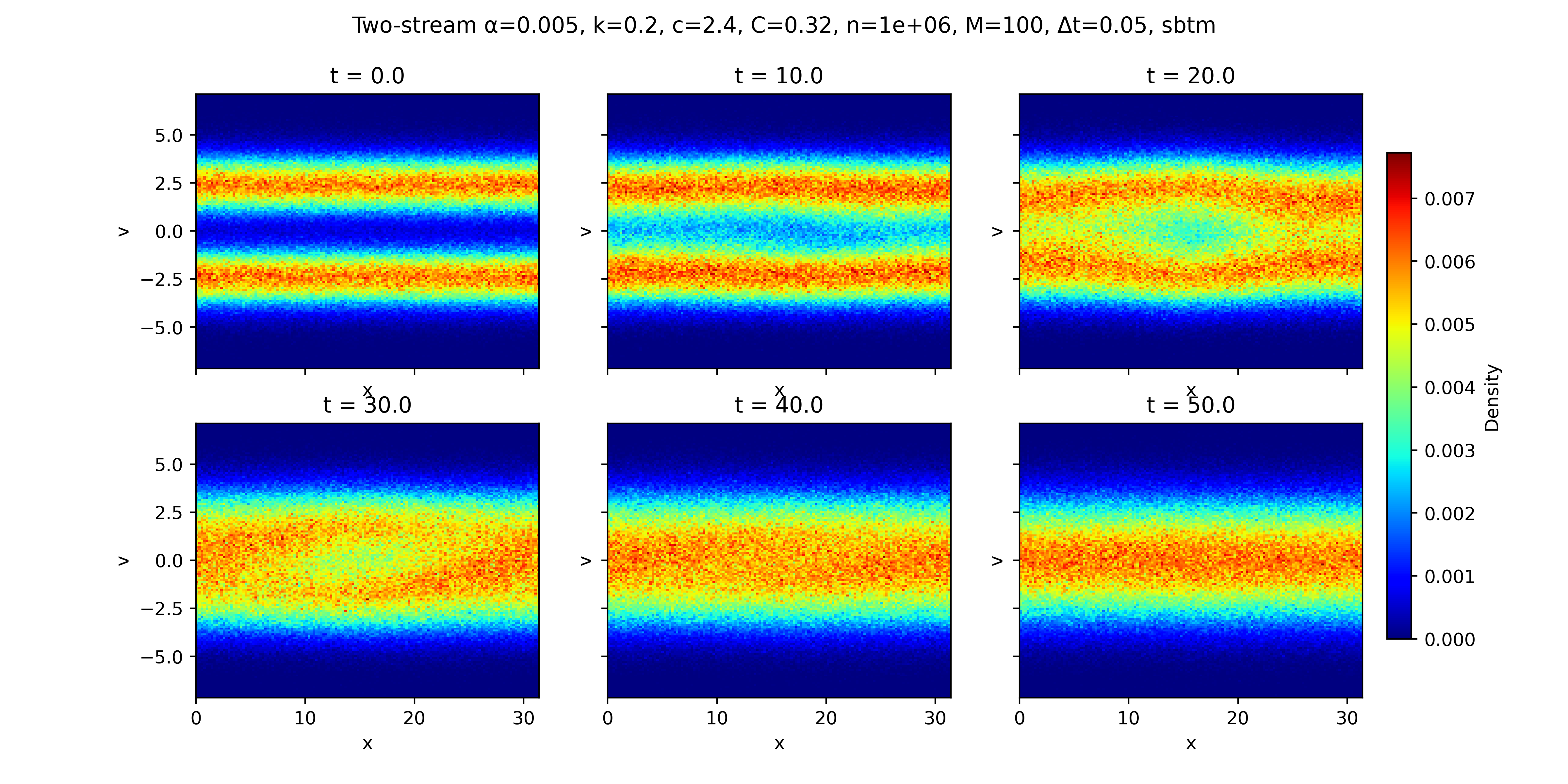}
\caption{SBTM, $n = 10^6$}
\end{subfigure}
\begin{subfigure}{0.48\linewidth}
\includegraphics[width=\linewidth]{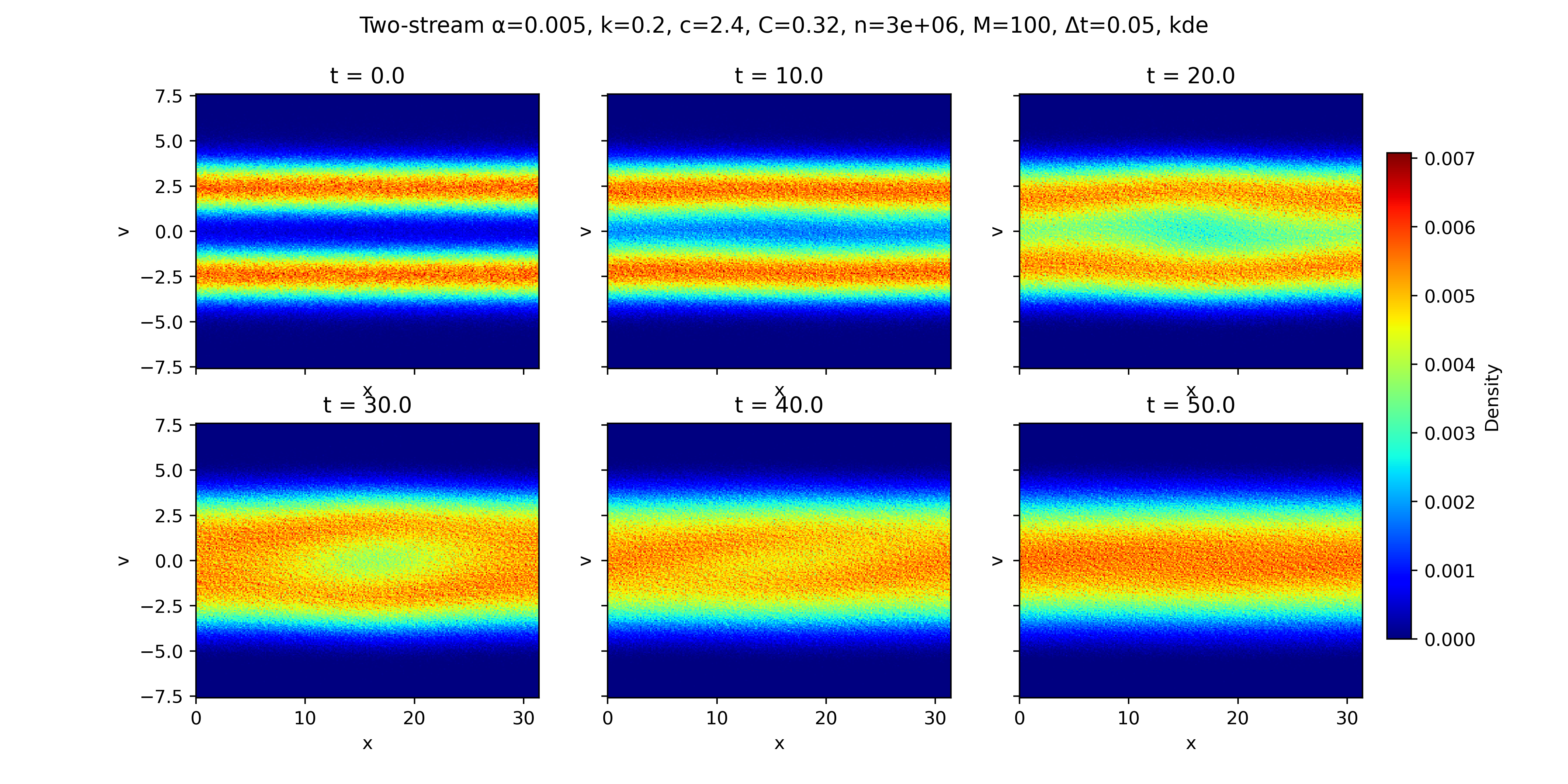}
\caption{Blob, $n = 3 \times 10^6$}
\end{subfigure}
\hfill
\begin{subfigure}{0.48\linewidth}
\includegraphics[width=\linewidth]{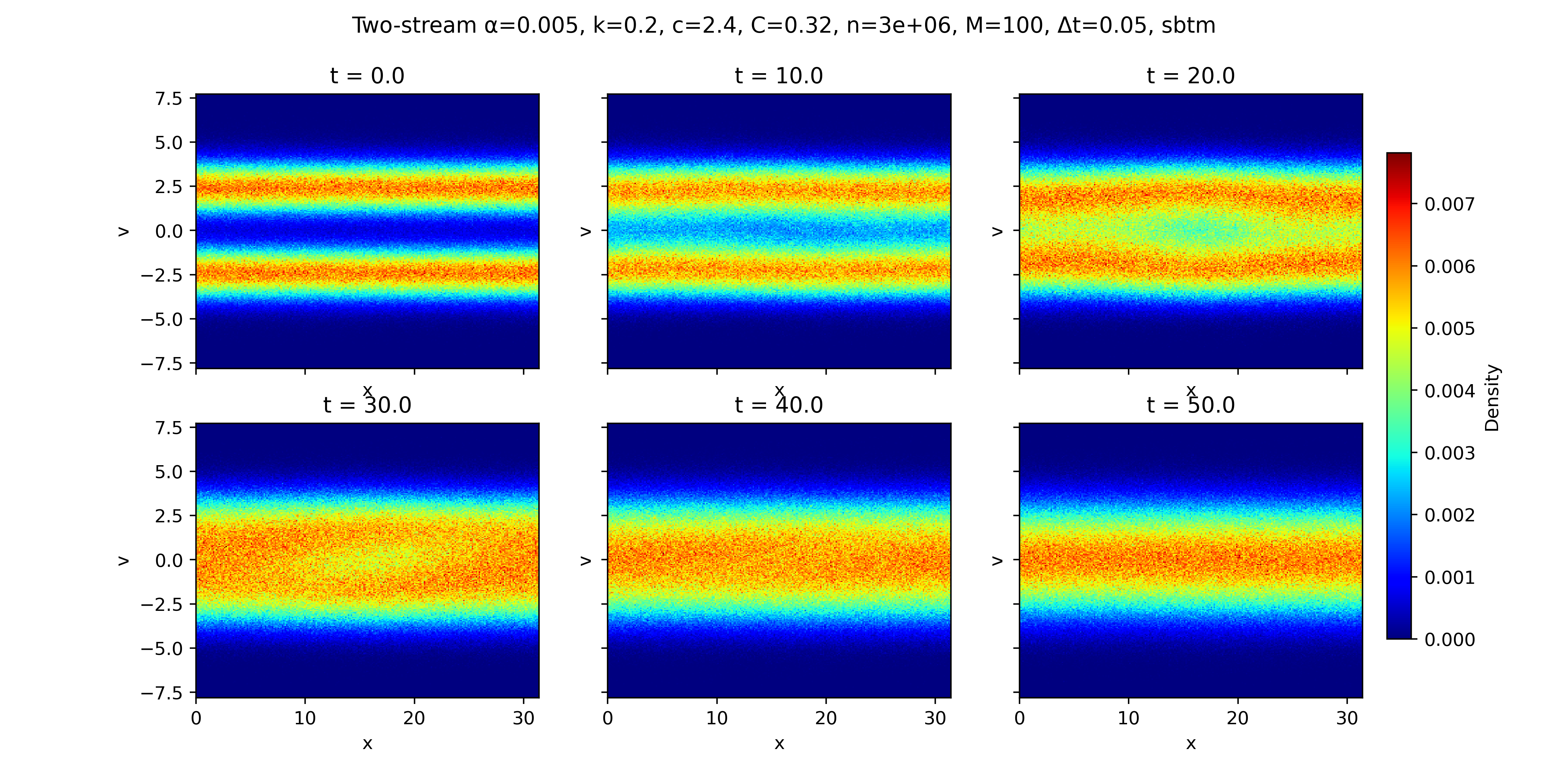}
\caption{SBTM, $n = 3 \times 10^6$}
\end{subfigure}
\caption{Two-stream instability: $(x, v_1)$ phase space at $\nu = 0.32$ across particle counts.  SBTM achieves full vortex dissipation at all particle counts, while the blob method requires $n = 3 \times 10^6$ to approach the same level of dissipation.}
\label{fig:twostream_phase}
\end{figure}

\textbf{Score estimation quality.}  Figure~\ref{fig:twostream_score} shows score quiver plots at $n = 10^6$.  At $t = 0$, where the true score is known analytically, SBTM achieves low mean squared error (MSE) while the blob method has substantially higher MSE.  At $t = t_{\mathrm{final}}$, SBTM scores are smooth and physically consistent: the score vectors point toward the origin with magnitude increasing away from it, as expected for a near-Maxwellian distribution ($\nabla_v \log f \approx -v/\sigma^2$).  The blob method produces noisy, incoherent scores in low-density tail regions. 

\begin{figure}[t]
\centering
\begin{subfigure}{0.48\linewidth}
\includegraphics[width=\linewidth]{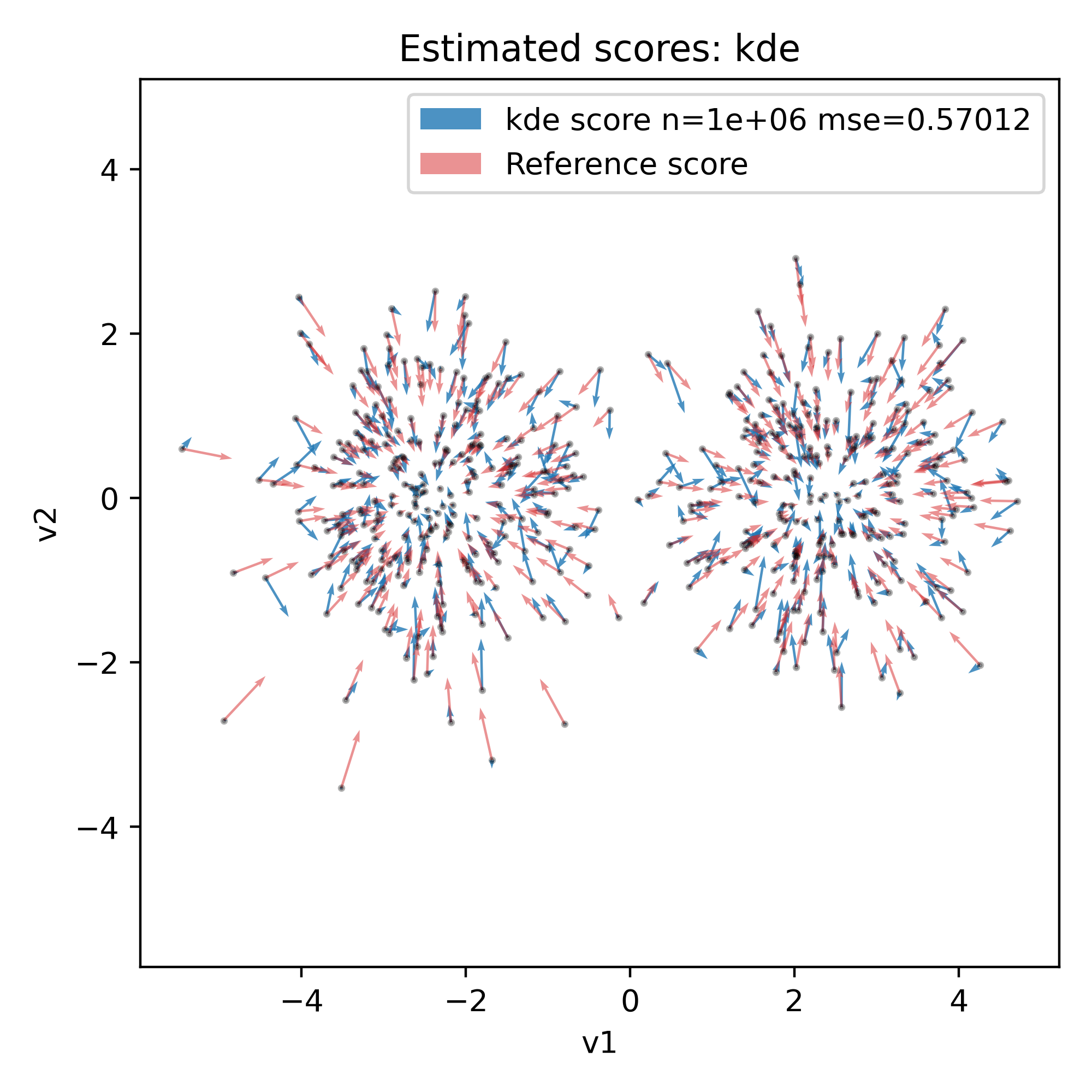}
\caption{Blob, $t = 0$}
\end{subfigure}
\hfill
\begin{subfigure}{0.48\linewidth}
\includegraphics[width=\linewidth]{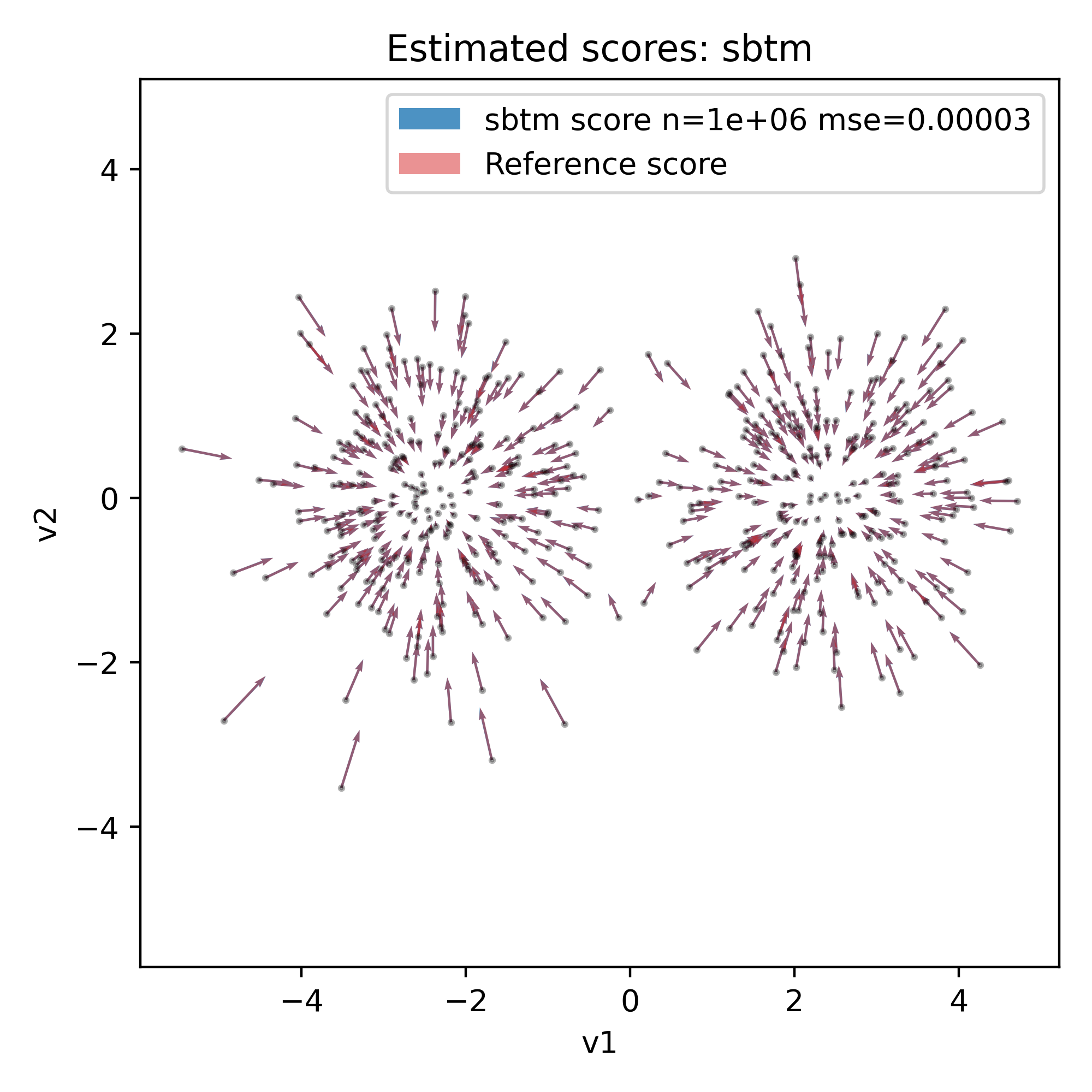}
\caption{SBTM, $t = 0$}
\end{subfigure}
\begin{subfigure}{0.48\linewidth}
\includegraphics[width=\linewidth]{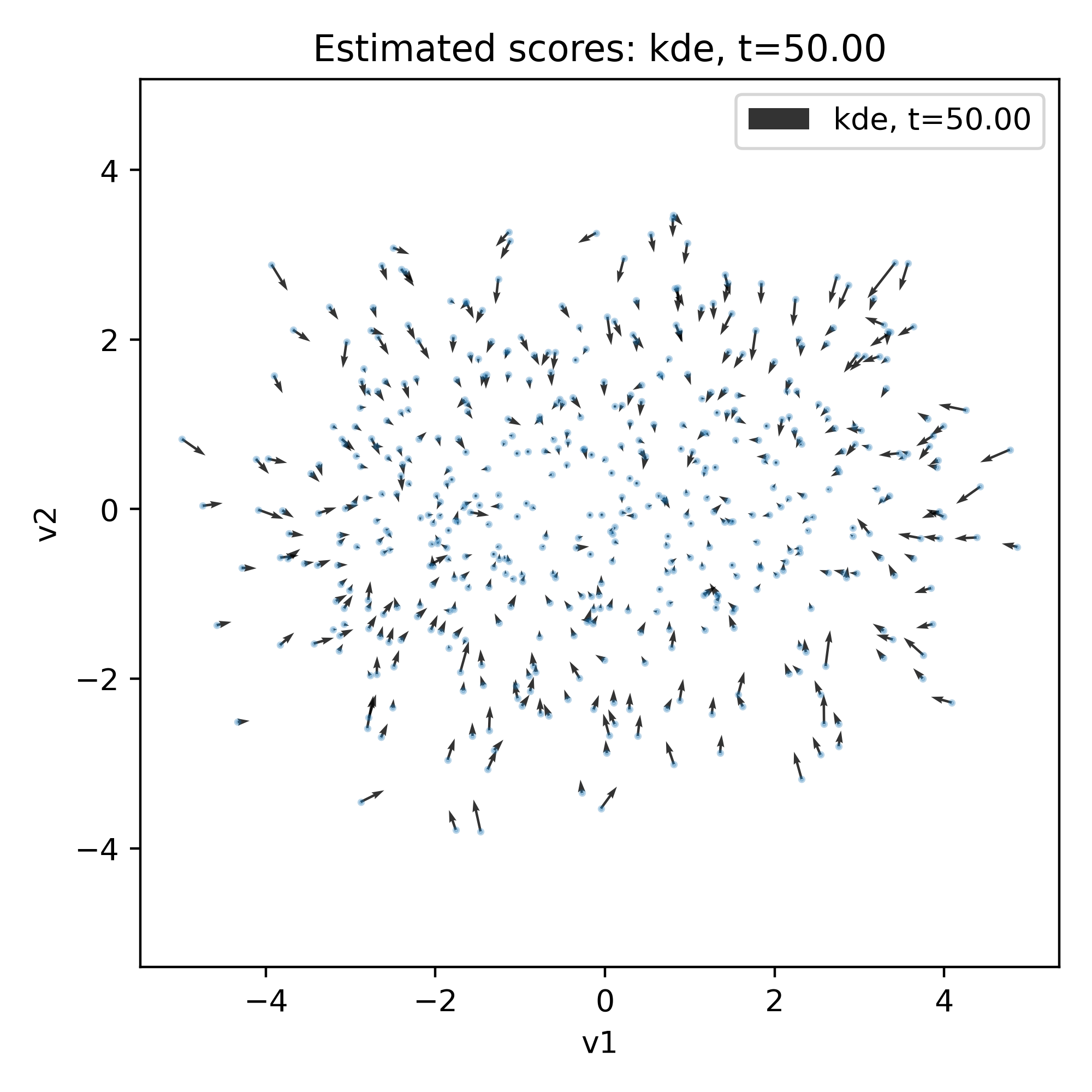}
\caption{Blob, $t = t_{\mathrm{final}}$}
\end{subfigure}
\hfill
\begin{subfigure}{0.48\linewidth}
\includegraphics[width=\linewidth]{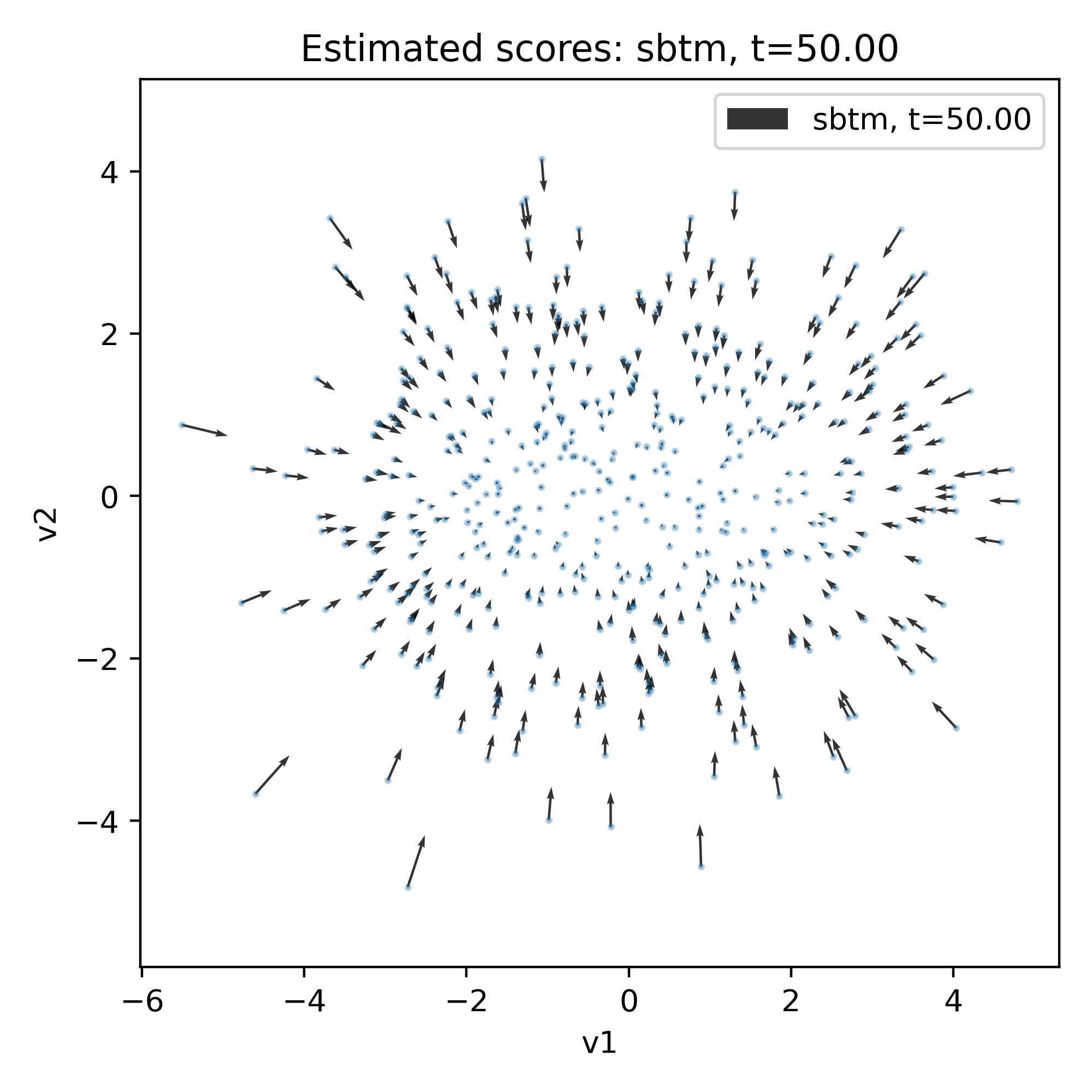}
\caption{SBTM, $t = t_{\mathrm{final}}$}
\end{subfigure}
\caption{Two-stream instability: score quiver plots at $\nu = 0.24$, $n = 10^6$.  Top: initial time (true score in red; MSE in legend).  Bottom: final time.  SBTM produces smooth, physically consistent scores; the blob method has high MSE, particularly in low-density tail regions.}
\label{fig:twostream_score}
\end{figure}

%

\textbf{Collision force quiver plots.}  Figure~\ref{fig:twostream_flow} shows the collision force $U^{\eta}(x_p,v_p)$ at six times for blob and SBTM at $n = 5 \times 10^5$.  SBTM produces a physically consistent collision field: it preserves the vertical ($v_1 \to -v_1$) and horizontal ($v_2 \to -v_2$) symmetry of the two-stream initial condition and vanishes by $t = 50$ as the system equilibrates.  The blob method's collision field is noisy and incoherent, breaks these symmetries, and persists at $t = 50$ even though the system should be near equilibrium.

\begin{figure}[htp!]
\centering
\begin{subfigure}{0.31\linewidth}
\includegraphics[width=\linewidth]{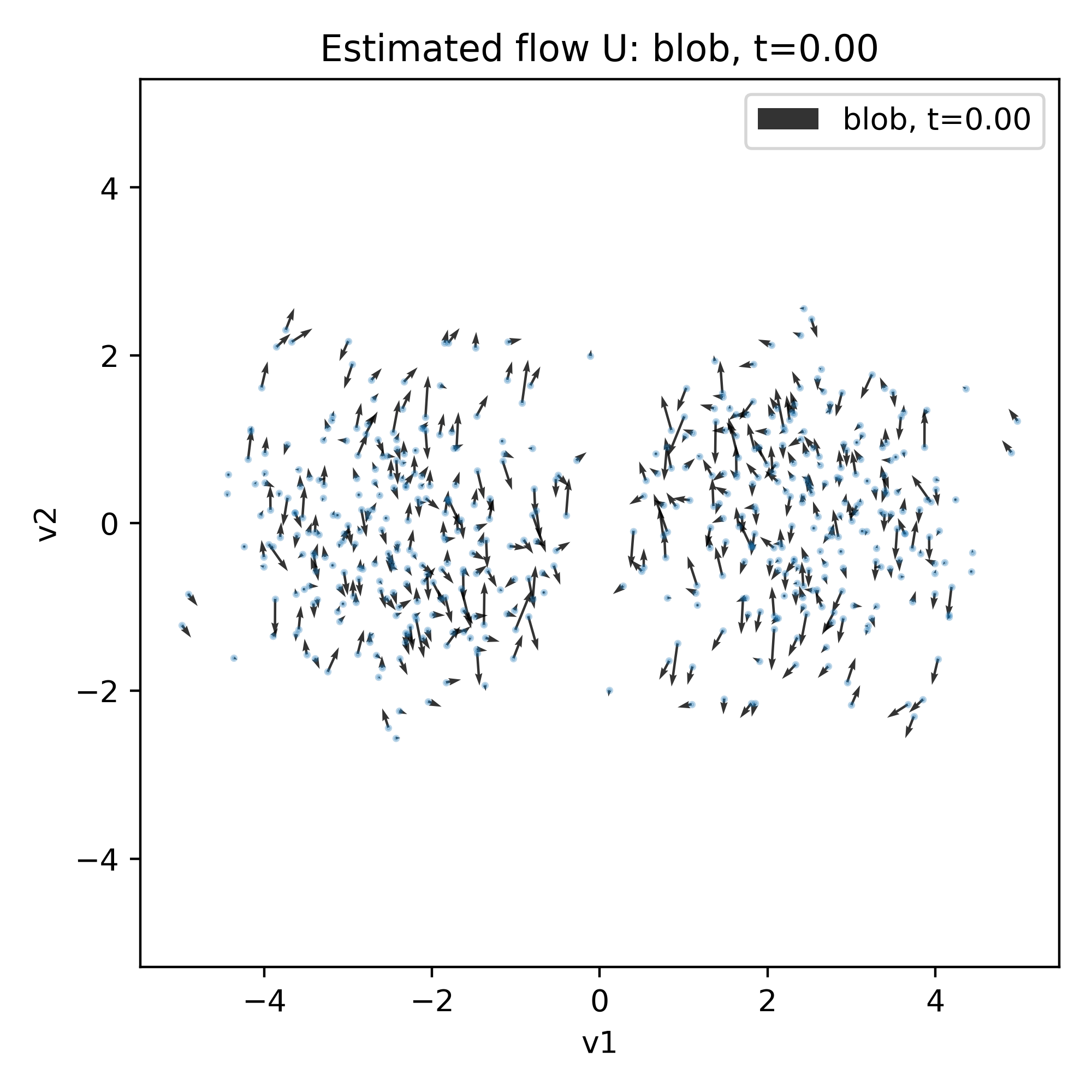}
\caption{Blob, $t = 0$}
\end{subfigure}
\hfill
\begin{subfigure}{0.31\linewidth}
\includegraphics[width=\linewidth]{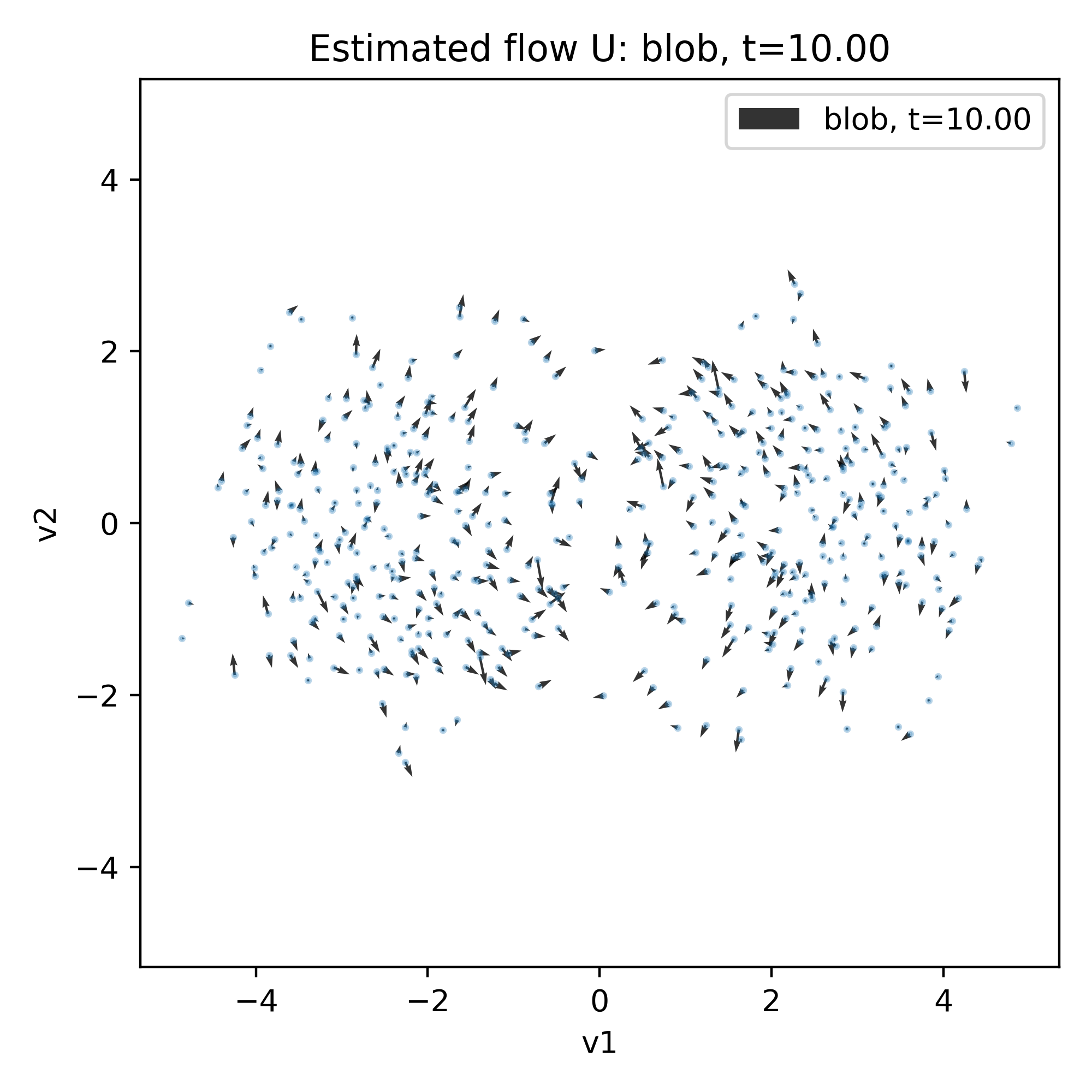}
\caption{Blob, $t = 10$}
\end{subfigure}
\hfill
\begin{subfigure}{0.31\linewidth}
\includegraphics[width=\linewidth]{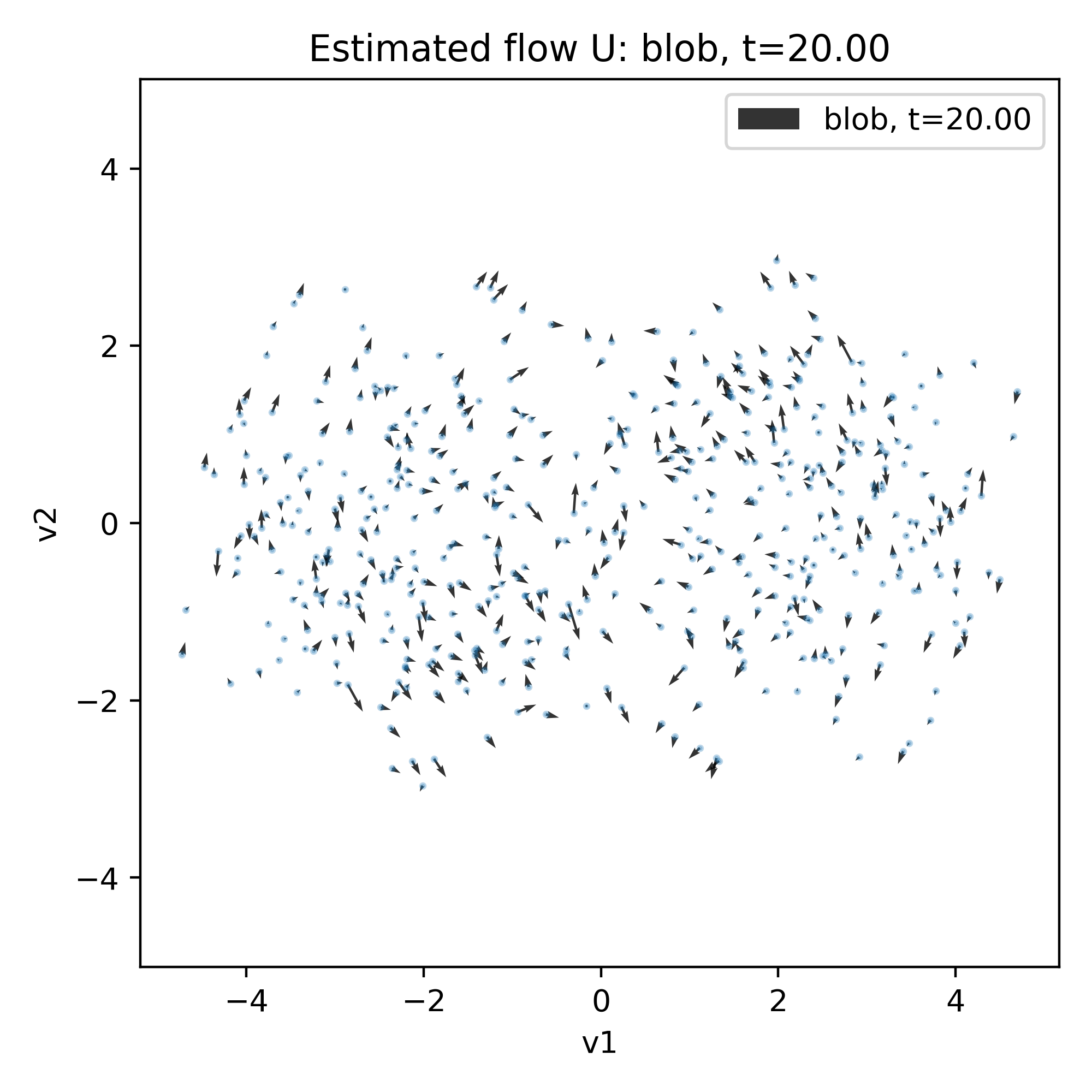}
\caption{Blob, $t = 20$}
\end{subfigure}
\begin{subfigure}{0.31\linewidth}
\includegraphics[width=\linewidth]{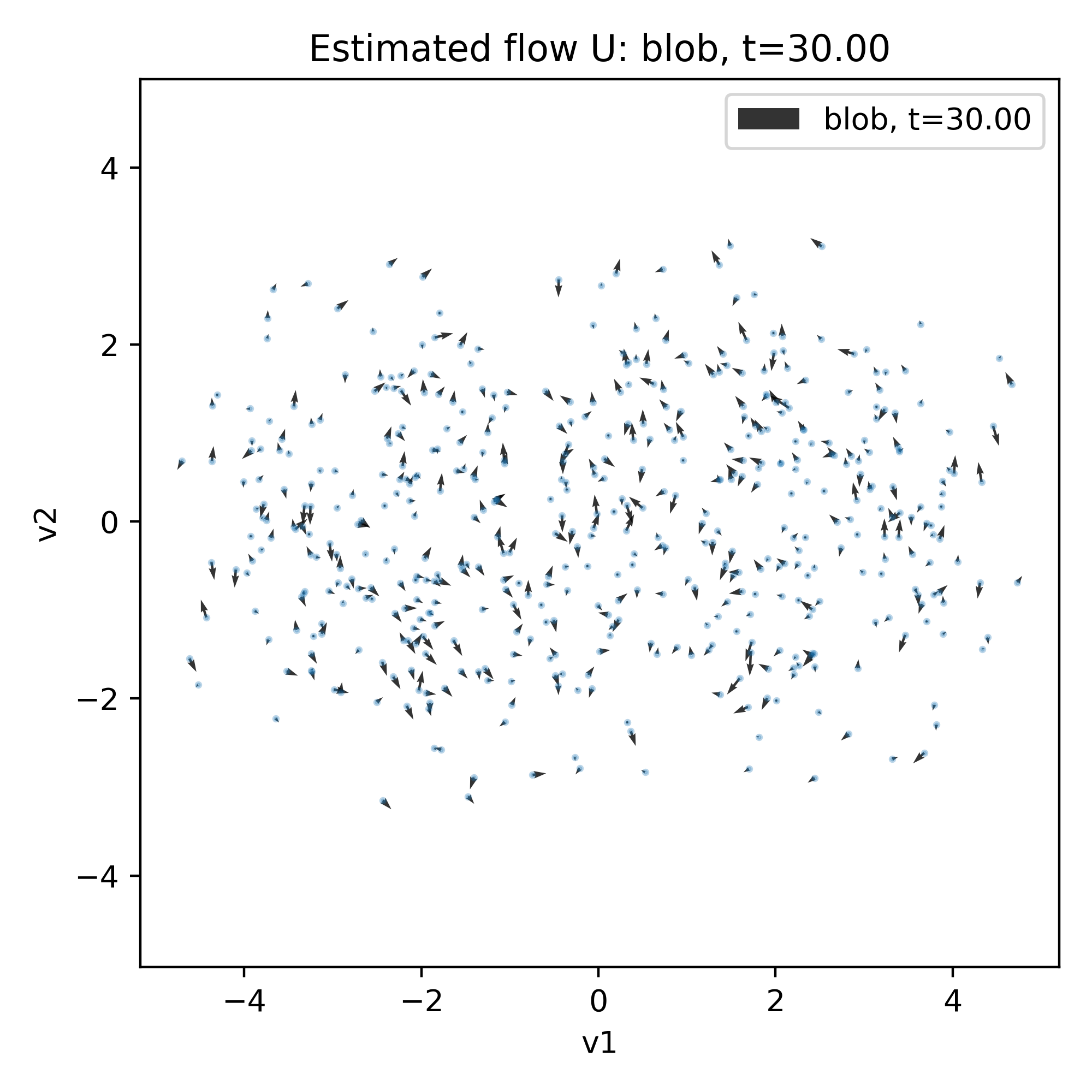}
\caption{Blob, $t = 30$}
\end{subfigure}
\hfill
\begin{subfigure}{0.31\linewidth}
\includegraphics[width=\linewidth]{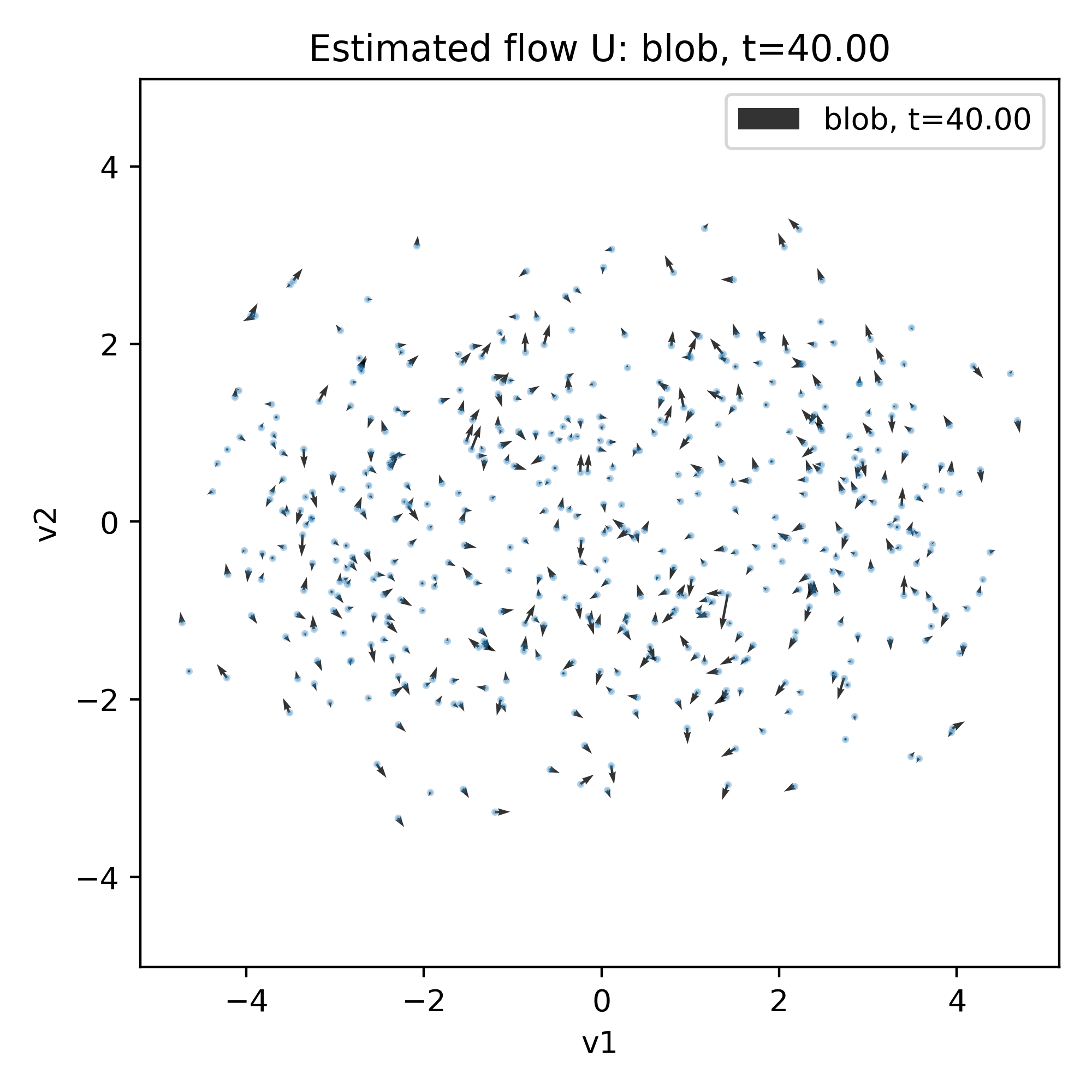}
\caption{Blob, $t = 40$}
\end{subfigure}
\hfill
\begin{subfigure}{0.31\linewidth}
\includegraphics[width=\linewidth]{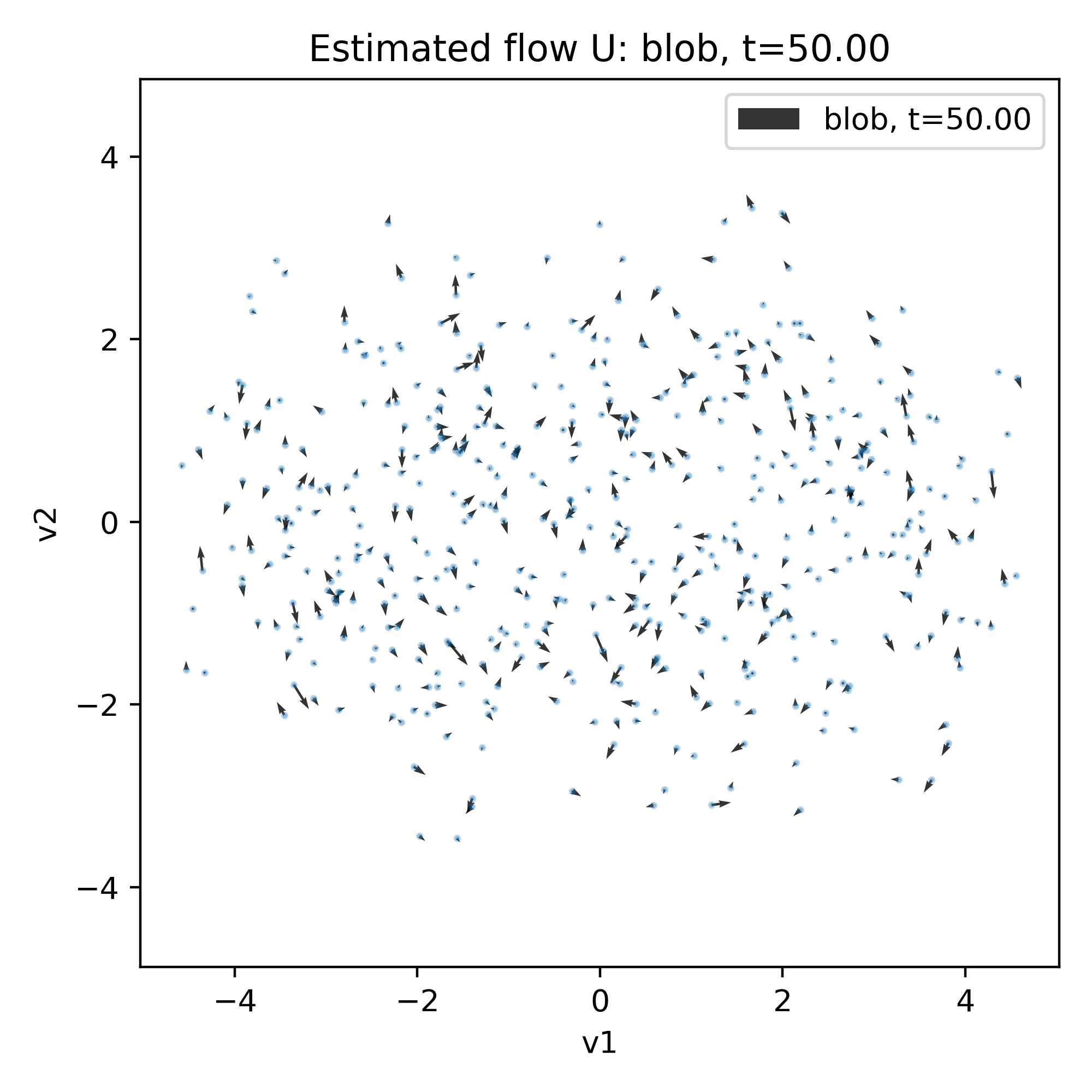}
\caption{Blob, $t = 50$}
\end{subfigure}
\begin{subfigure}{0.31\linewidth}
\includegraphics[width=\linewidth]{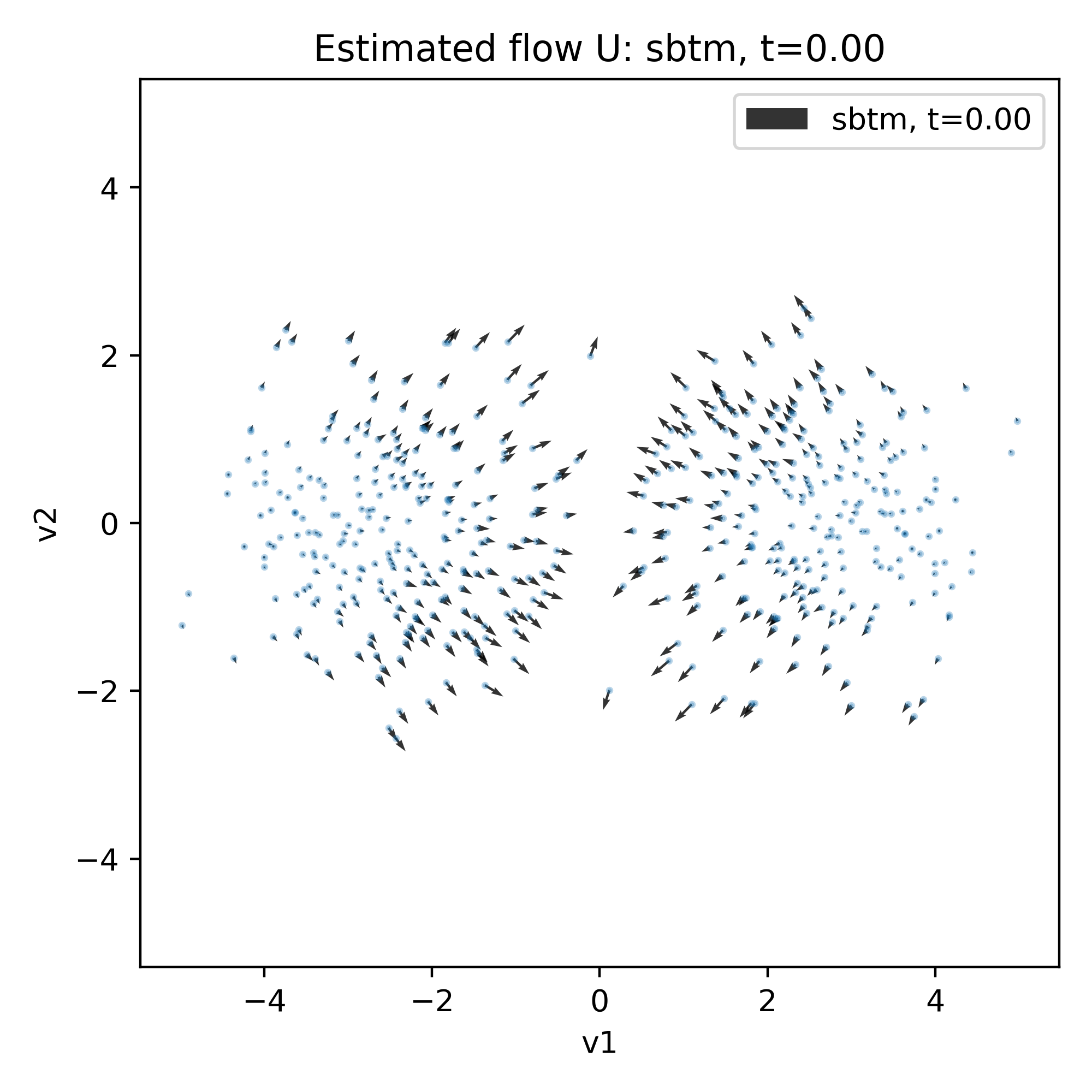}
\caption{SBTM, $t = 0$}
\end{subfigure}
\hfill
\begin{subfigure}{0.31\linewidth}
\includegraphics[width=\linewidth]{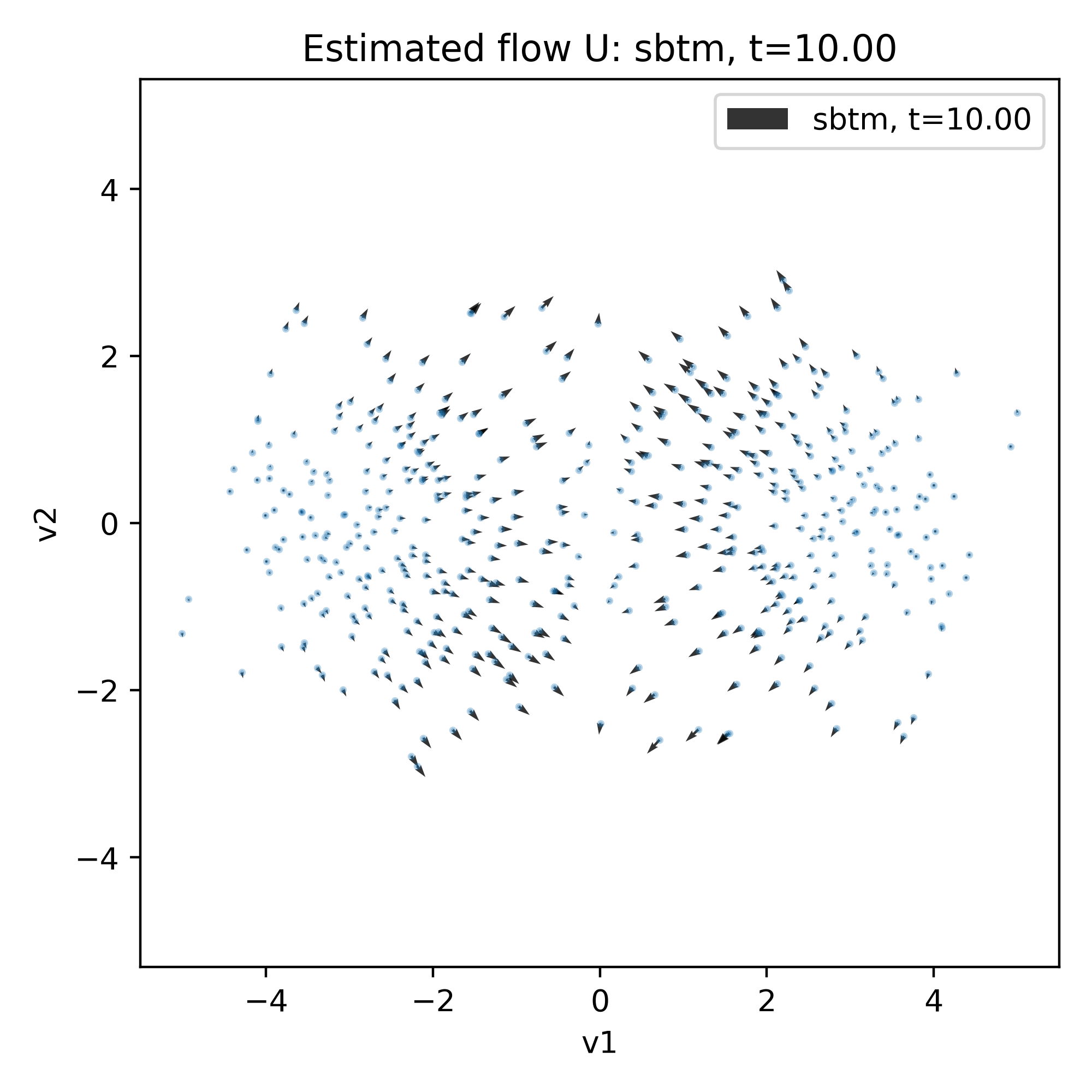}
\caption{SBTM, $t = 10$}
\end{subfigure}
\hfill
\begin{subfigure}{0.31\linewidth}
\includegraphics[width=\linewidth]{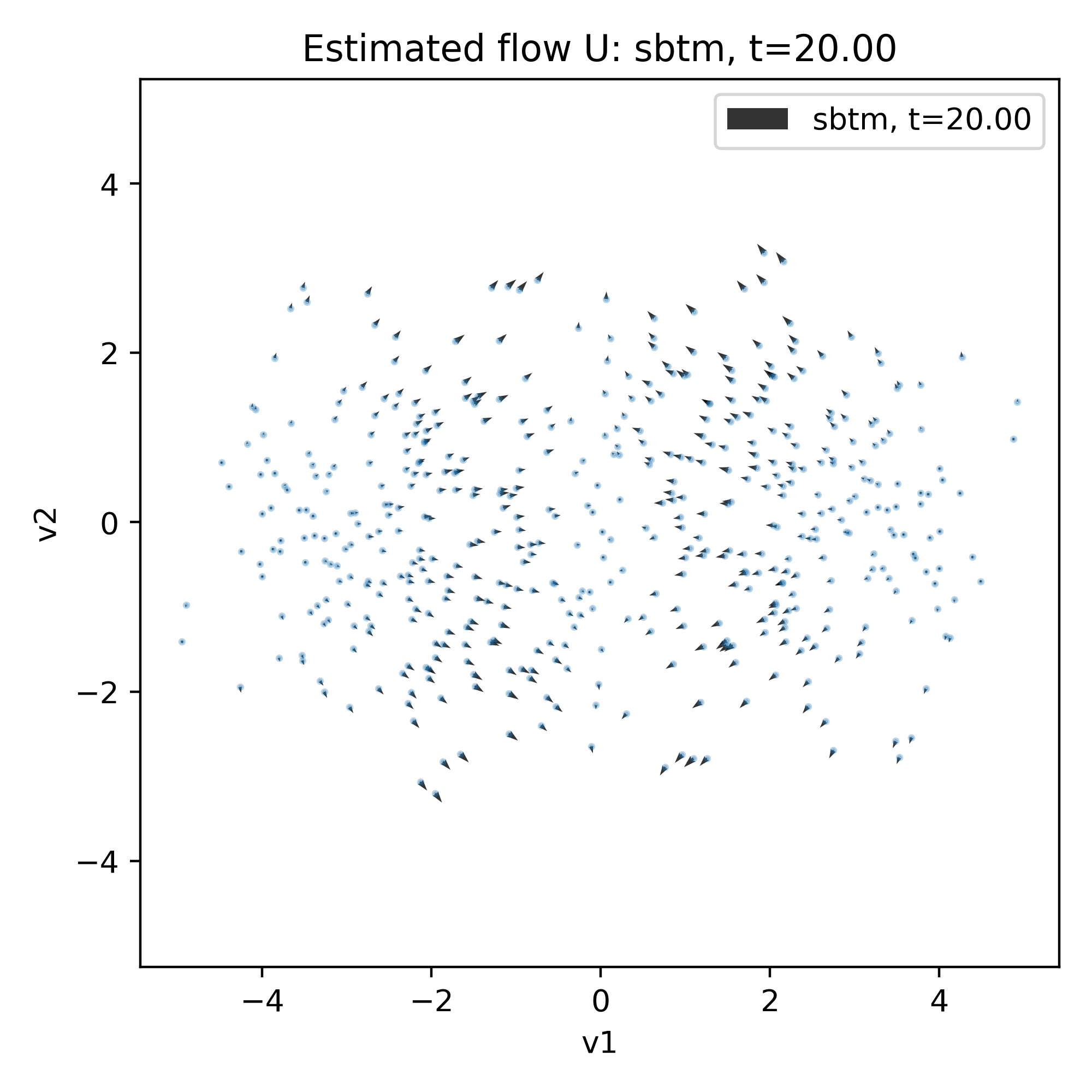}
\caption{SBTM, $t = 20$}
\end{subfigure}
\begin{subfigure}{0.31\linewidth}
\includegraphics[width=\linewidth]{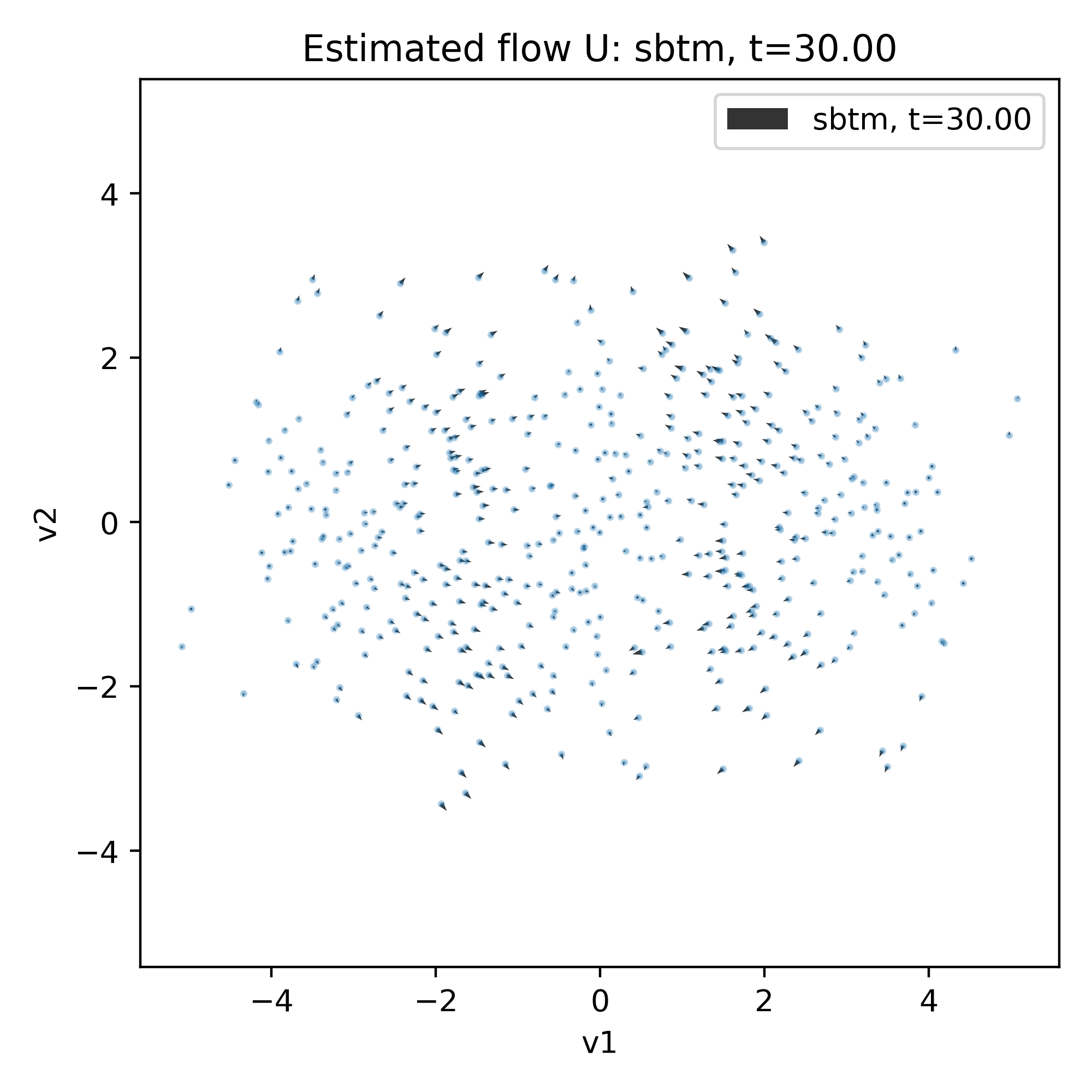}
\caption{SBTM, $t = 30$}
\end{subfigure}
\hfill
\begin{subfigure}{0.31\linewidth}
\includegraphics[width=\linewidth]{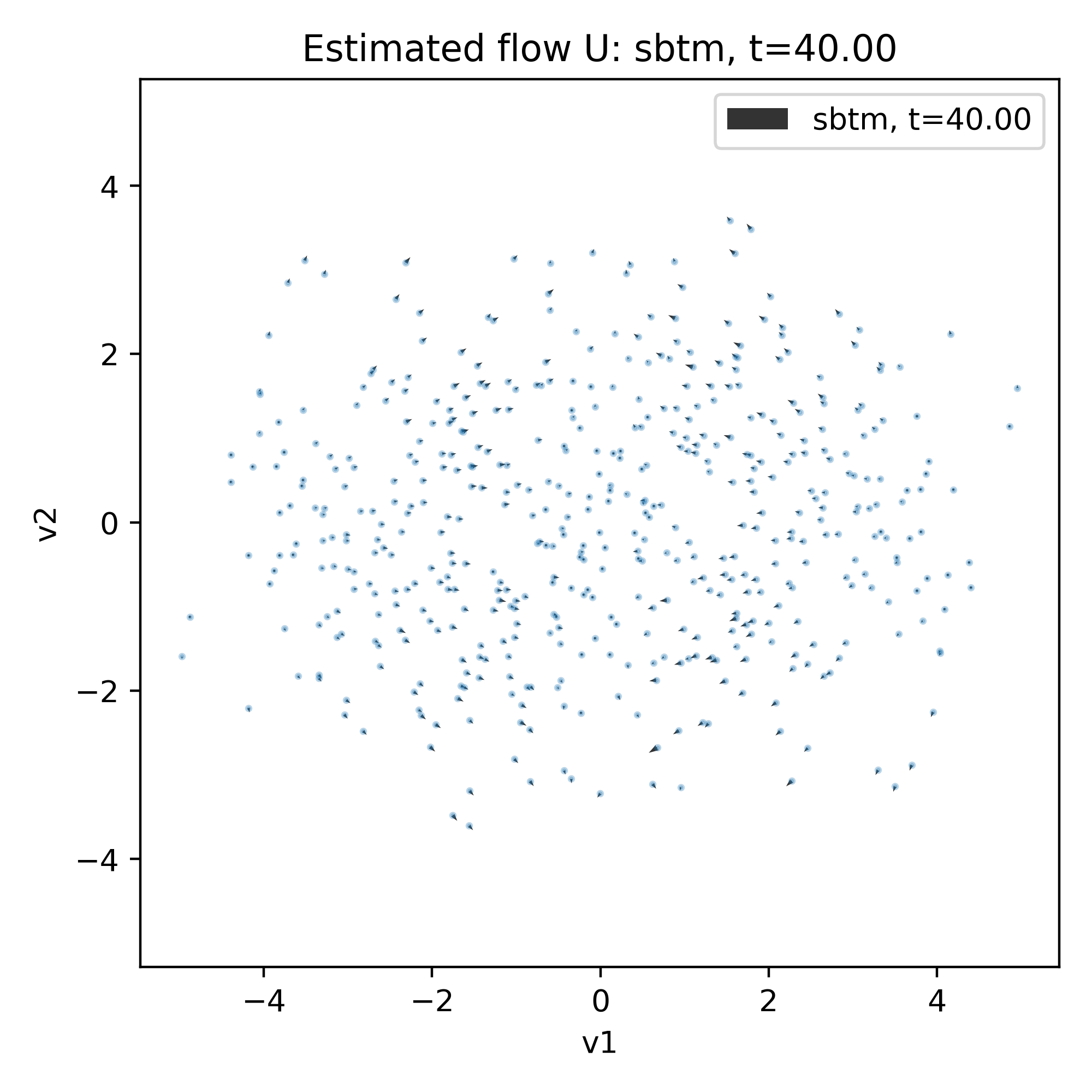}
\caption{SBTM, $t = 40$}
\end{subfigure}
\hfill
\begin{subfigure}{0.31\linewidth}
\includegraphics[width=\linewidth]{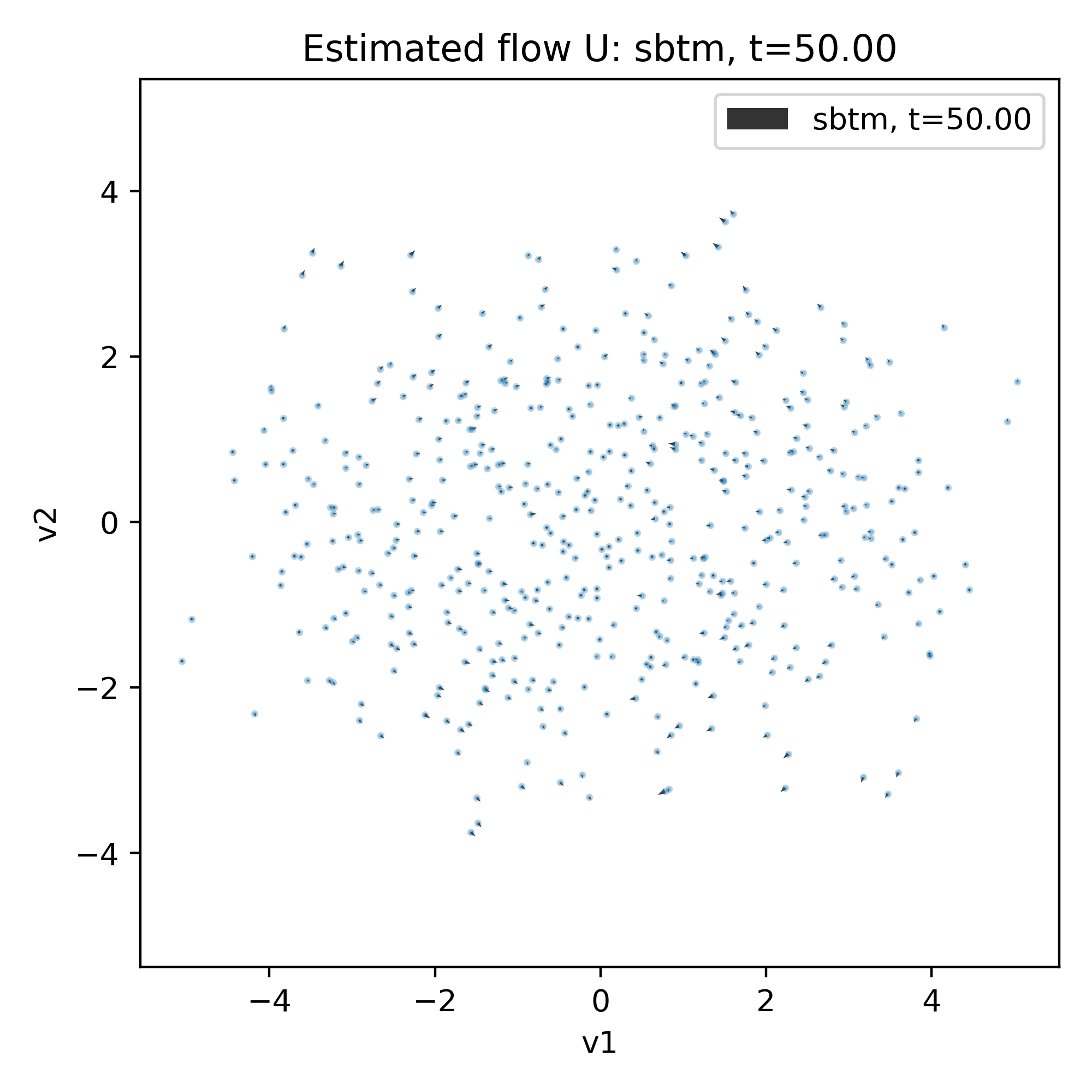}
\caption{SBTM, $t = 50$}
\end{subfigure}
\caption{Two-stream instability: collision force at $\nu = 0.24$, $n = 5 \times 10^5$.  Top two rows: blob (noisy, incoherent fields that worsen as fine structure develops).  Bottom two rows: SBTM (smooth, coherent fields at all times).}
\label{fig:twostream_flow}
\end{figure}

\FloatBarrier

\subsection{Weibel Instability}
\label{sec:weibel}

\textbf{Setup.}  An anisotropic bi-Maxwellian with a small magnetic perturbation:
\begin{equation}
\label{eq:weibel_ic}
f_0(x, v) \propto \exp\!\left(-\frac{v_1^2}{\beta}\right) \left[\exp\!\left(-\frac{(v_2 - c)^2}{\beta}\right) + \exp\!\left(-\frac{(v_2 + c)^2}{\beta}\right)\right] \prod_{j \geq 3} \exp\!\left(-\frac{v_j^2}{\beta}\right),
\end{equation}
with $c = 0.3$, $\beta = 0.01$, $B_3(0, x) = \alpha_B \sin(kx)$, $\alpha_B = 10^{-3}$, $k = 1/5$, domain $L = 2\pi/k = 10\pi$, $M = 100$, $\Delta t = 0.1$, $d_v = 3$, $n = 10^6$, $t_{\mathrm{final}} = 125$, and $K = 100$ ISM steps per time step.  We sweep $\nu \in \{10^{-4}, 2 \times 10^{-4}, 4 \times 10^{-4}, 8 \times 10^{-4}\}$. 

\textbf{Velocity distribution and equilibration.}  By Theorem~\ref{thm:equilibrium}, the collision operator drives the distribution toward a Maxwellian $f_\infty \propto \exp(-|v|^2/(2T_\infty))$, where $T_\infty=0.035$ is determined by total energy conservation. Figure~\ref{fig:weibel_v2} shows the $v_2$-marginal density at $\nu = 8 \times 10^{-4}$ on both logarithmic and linear scales.  SBTM converges to the expected Maxwellian with smooth Gaussian tails.  The blob method produces sharp, unphysical cutoffs in the tails at $v_2 \approx \pm 0.55$: the KDE score estimate truncates the distribution rather than allowing the smooth tails predicted by kinetic theory.  On a linear scale both methods appear similar, but the blob method produces a slightly broader distribution at late times; the log-scale version reveals the sharp tail cutoff. Table~\ref{tab:weibel_thermalization} quantifies equilibration by the $L^2$ distance of the $v_2$-marginal to $\mathcal{N}(0, T_\infty)$ at $t = t_{\mathrm{final}}$.  SBTM is $1.2$--$3.5\times$ closer to this Maxwellian across all collision frequencies.

\begin{table}[h]
\centering
\begin{tabular}{lcc}
\toprule
$\nu$ & Blob & SBTM \\
\midrule
$2 \times 10^{-4}$ & 0.325 & \textbf{0.265} \\
$4 \times 10^{-4}$ & 0.134 & \textbf{0.053} \\
$8 \times 10^{-4}$ & 0.070 & \textbf{0.020} \\
\bottomrule
\end{tabular}
\caption{Weibel instability: $L^2$ distance of the $v_2$-marginal to $\mathcal{N}(0, T_\infty)$ at $ t_{\mathrm{final}}=125$ ($n = 10^6$, $d_v = 3$).  Lower is better.  SBTM relaxes closer to the Maxwellian steady state across all $\nu$. }
\label{tab:weibel_thermalization}
\end{table}

\begin{figure}[t]
\centering
\begin{subfigure}{0.48\linewidth}
\includegraphics[width=\linewidth]{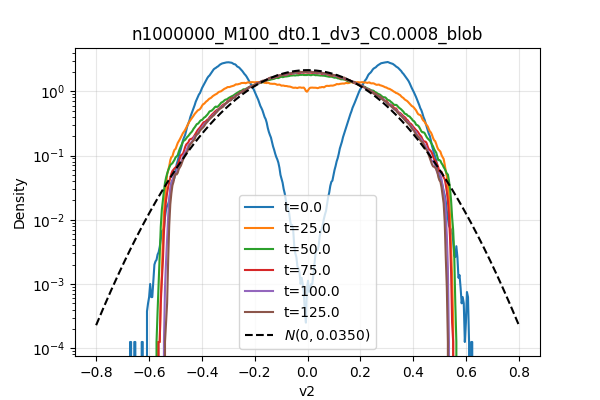}
\caption{Blob, log scale}
\end{subfigure}
\hfill
\begin{subfigure}{0.48\linewidth}
\includegraphics[width=\linewidth]{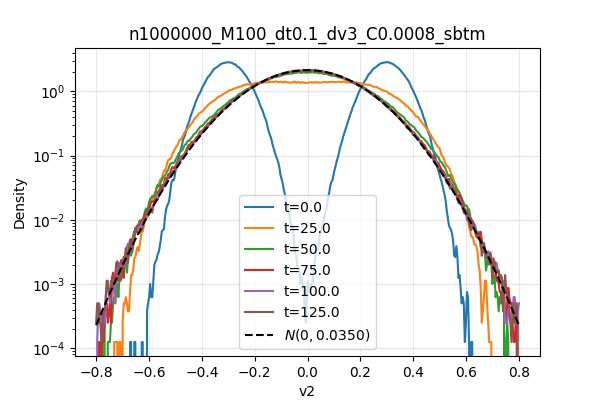}
\caption{SBTM, log scale}
\end{subfigure}
\begin{subfigure}{0.48\linewidth}
\includegraphics[width=\linewidth]{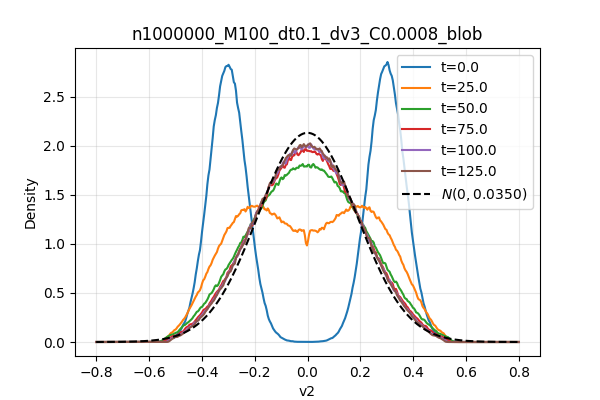}
\caption{Blob, linear scale}
\end{subfigure}
\hfill
\begin{subfigure}{0.48\linewidth}
\includegraphics[width=\linewidth]{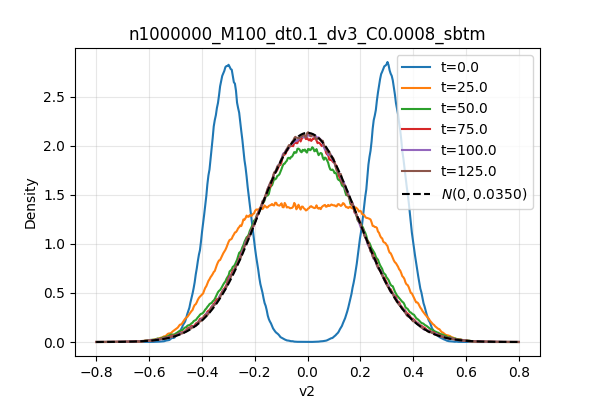}
\caption{SBTM, linear scale}
\end{subfigure}
\caption{Weibel instability: $v_2$-marginal density evolution at $\nu = 8 \times 10^{-4}$.  Top: log scale; bottom: linear scale.  The dashed black curve is a Gaussian/Maxwellian.  SBTM converges to it with smooth Gaussian tails, while the blob method exhibits sharp, unphysical cutoffs at $v_2 \approx \pm 0.55$.}
\label{fig:weibel_v2}
\end{figure}

\textbf{Phase space and velocity-space slices.}  Figure~\ref{fig:weibel_phase} shows the $(v_1, v_2)$-marginal density at several times for two collision frequencies, along with $(v_1, v_2)$ density slices at $x = 0$.  At low $\nu$, both methods are similar; at higher $\nu$, the blob method produces a non-Gaussian density while SBTM equilibrates correctly.  At $\nu = 0$ (no collisions), both methods produce identical density slices, confirming that the PIC framework is the same.  At $\nu = 10^{-4}$, the blob method already over-smooths the distribution, while SBTM retains the fine structures observed in the collisionless case.

\begin{figure}[h]
\centering
\begin{subfigure}{0.48\linewidth}
\includegraphics[width=\linewidth]{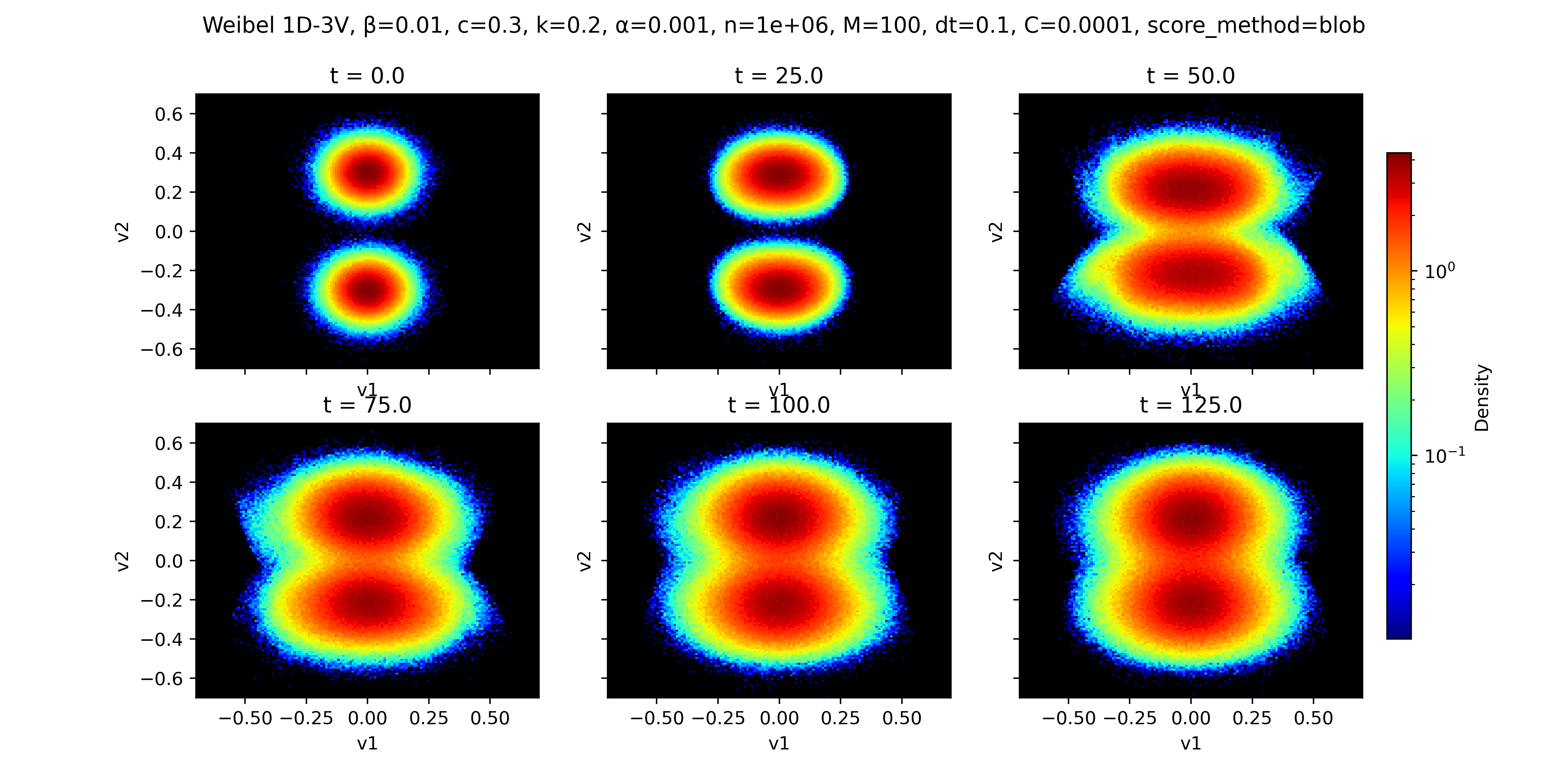}
\caption{Blob, $\nu = 10^{-4}$}
\end{subfigure}
\hfill
\begin{subfigure}{0.48\linewidth}
\includegraphics[width=\linewidth]{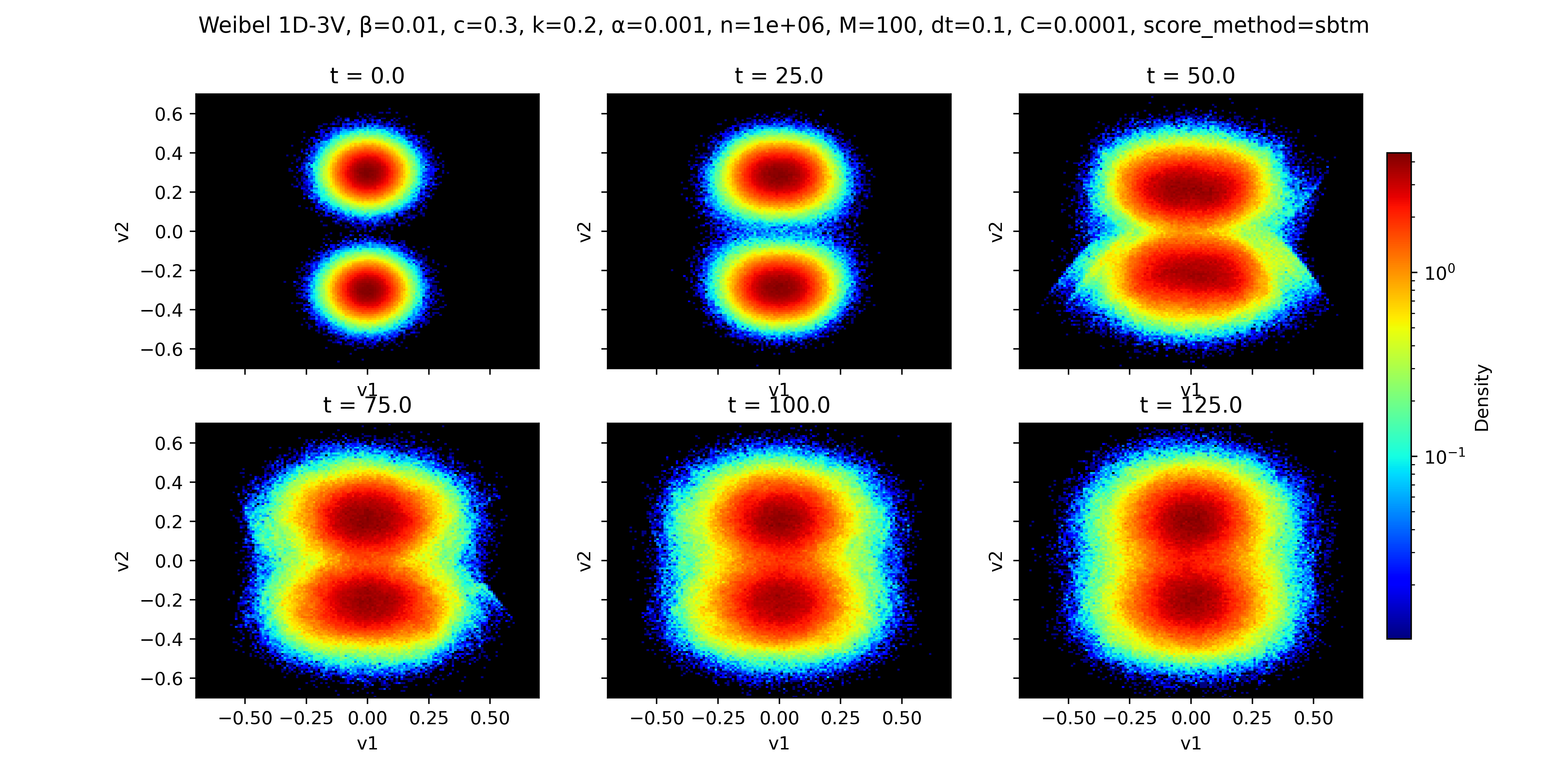}
\caption{SBTM, $\nu = 10^{-4}$}
\end{subfigure}
\begin{subfigure}{0.48\linewidth}
\includegraphics[width=\linewidth]{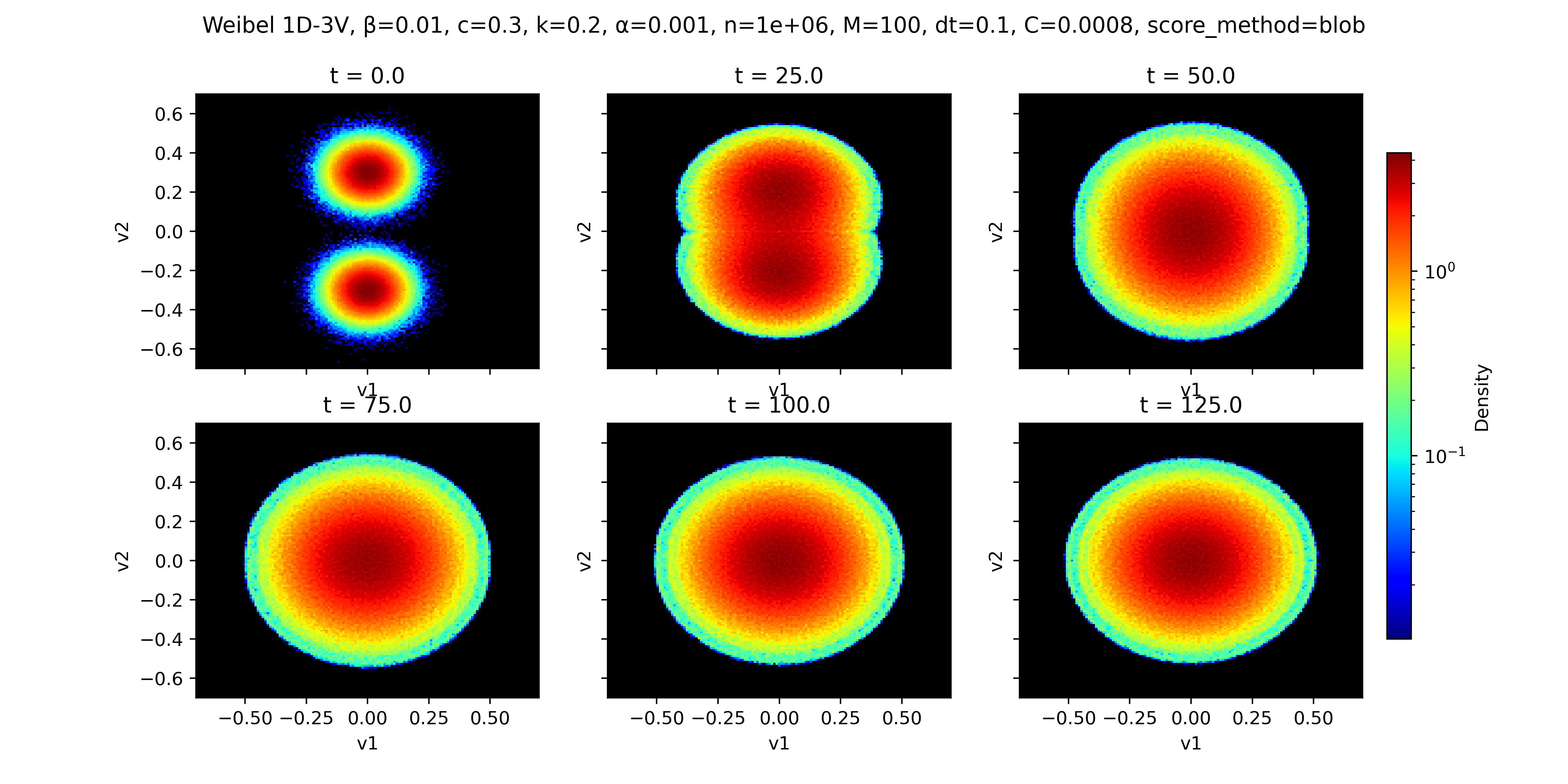}
\caption{Blob, $\nu = 8 \times 10^{-4}$}
\end{subfigure}
\hfill
\begin{subfigure}{0.48\linewidth}
\includegraphics[width=\linewidth]{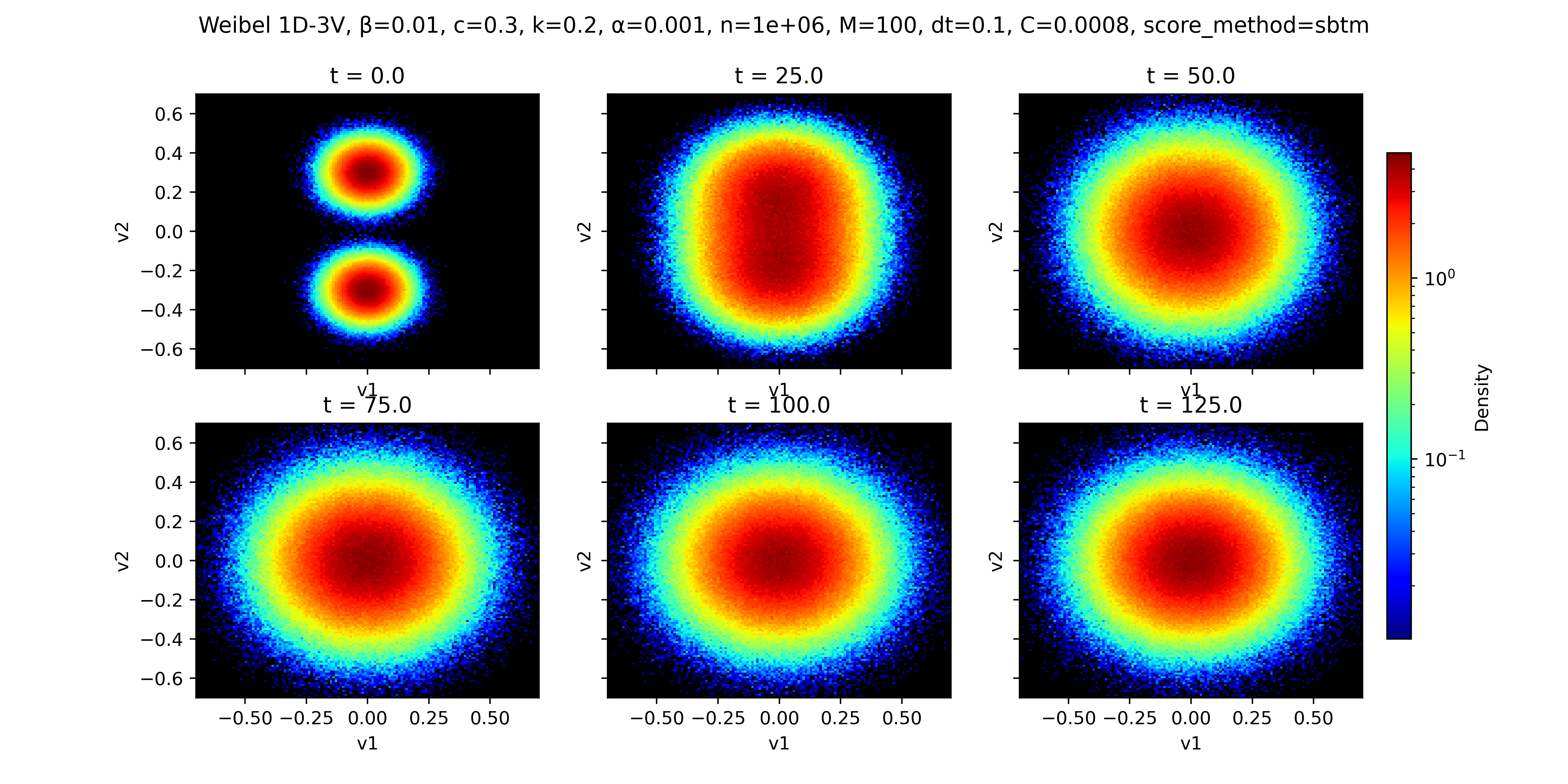}
\caption{SBTM, $\nu = 8 \times 10^{-4}$}
\end{subfigure}
\begin{subfigure}{0.48\linewidth}
\includegraphics[width=\linewidth]{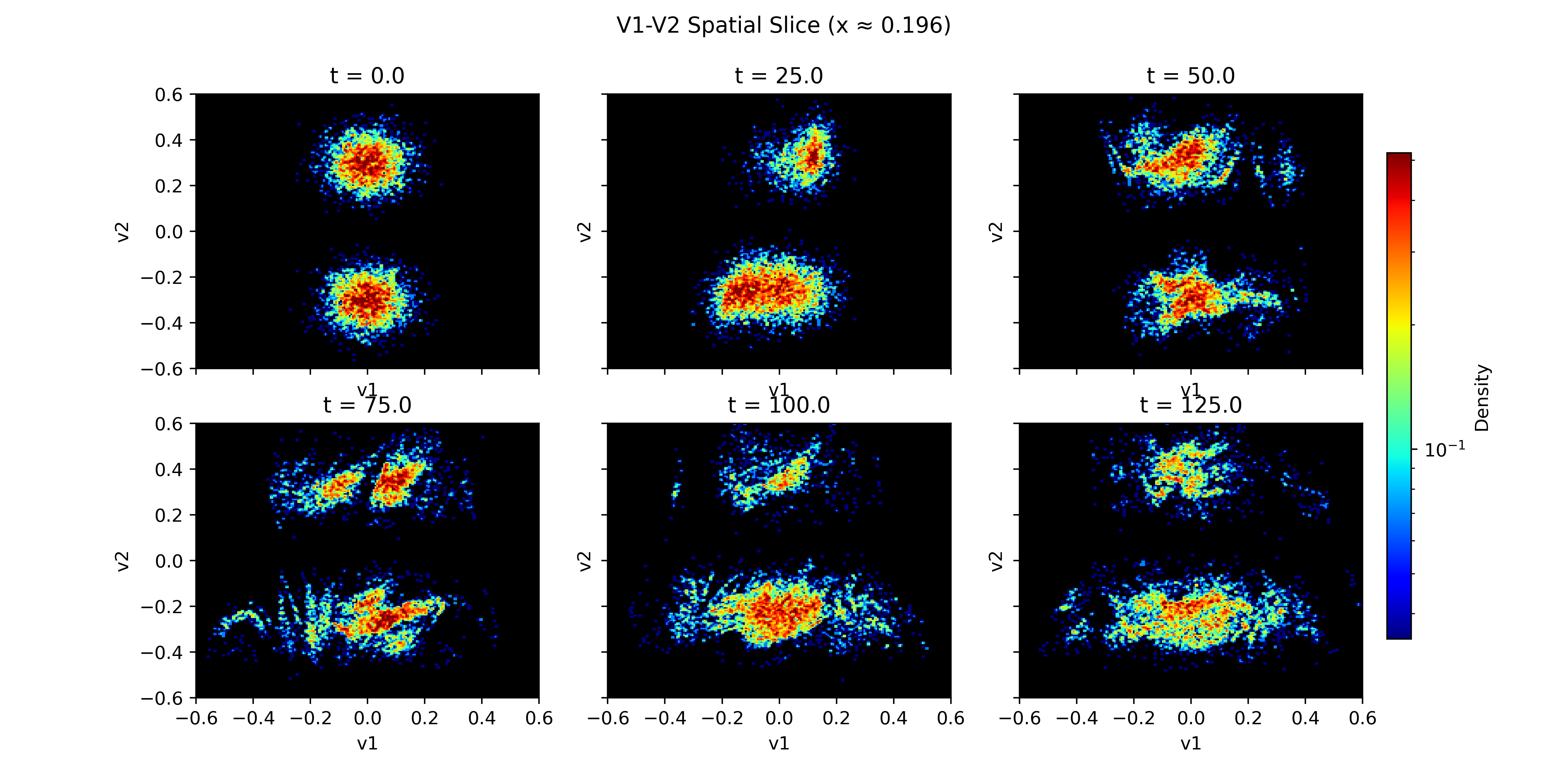}
\caption{Blob, $\nu = 0$}
\end{subfigure}
\hfill
\begin{subfigure}{0.48\linewidth}
\includegraphics[width=\linewidth]{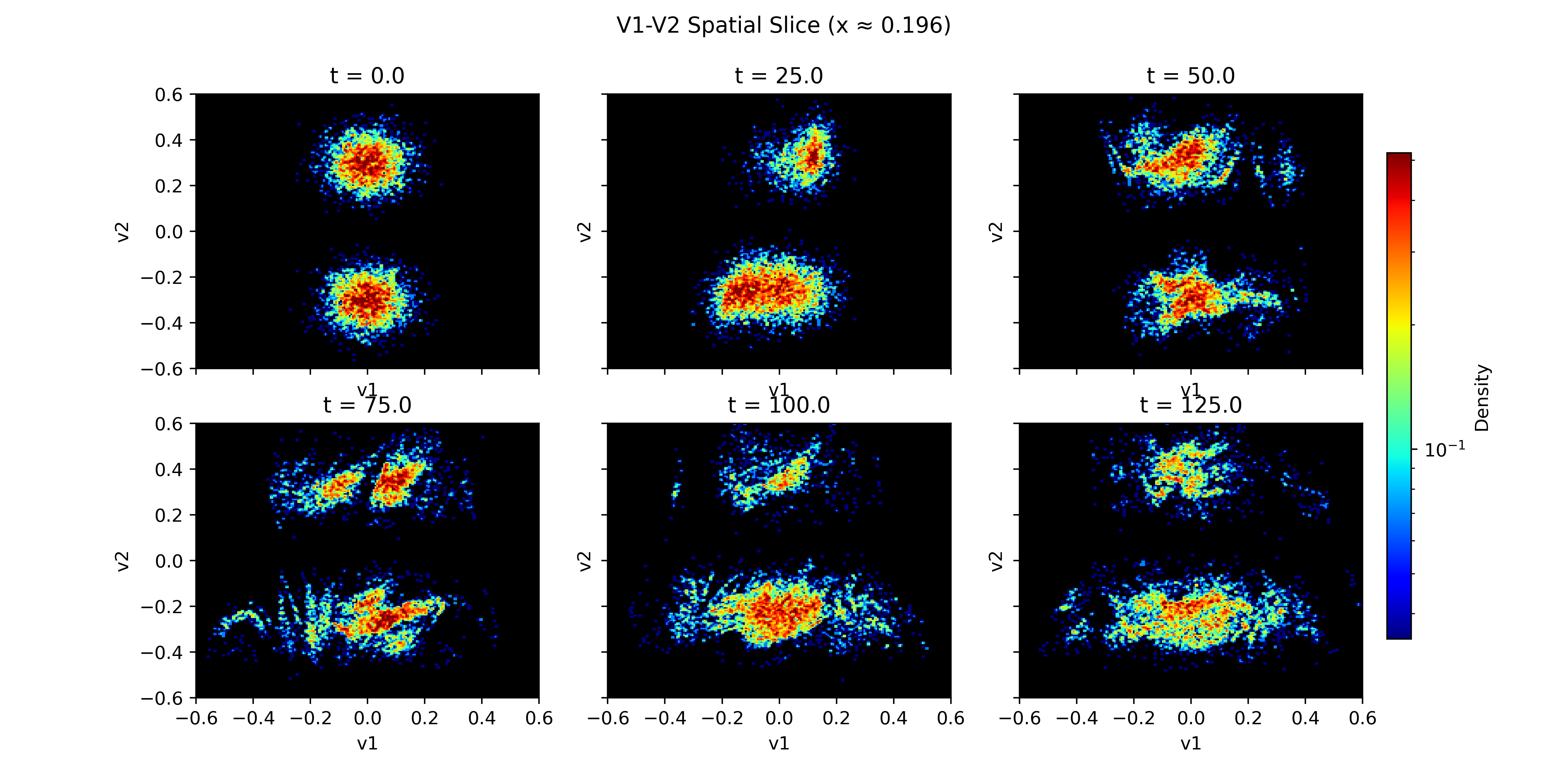}
\caption{SBTM, $\nu = 0$}
\end{subfigure}
\begin{subfigure}{0.48\linewidth}
\includegraphics[width=\linewidth]{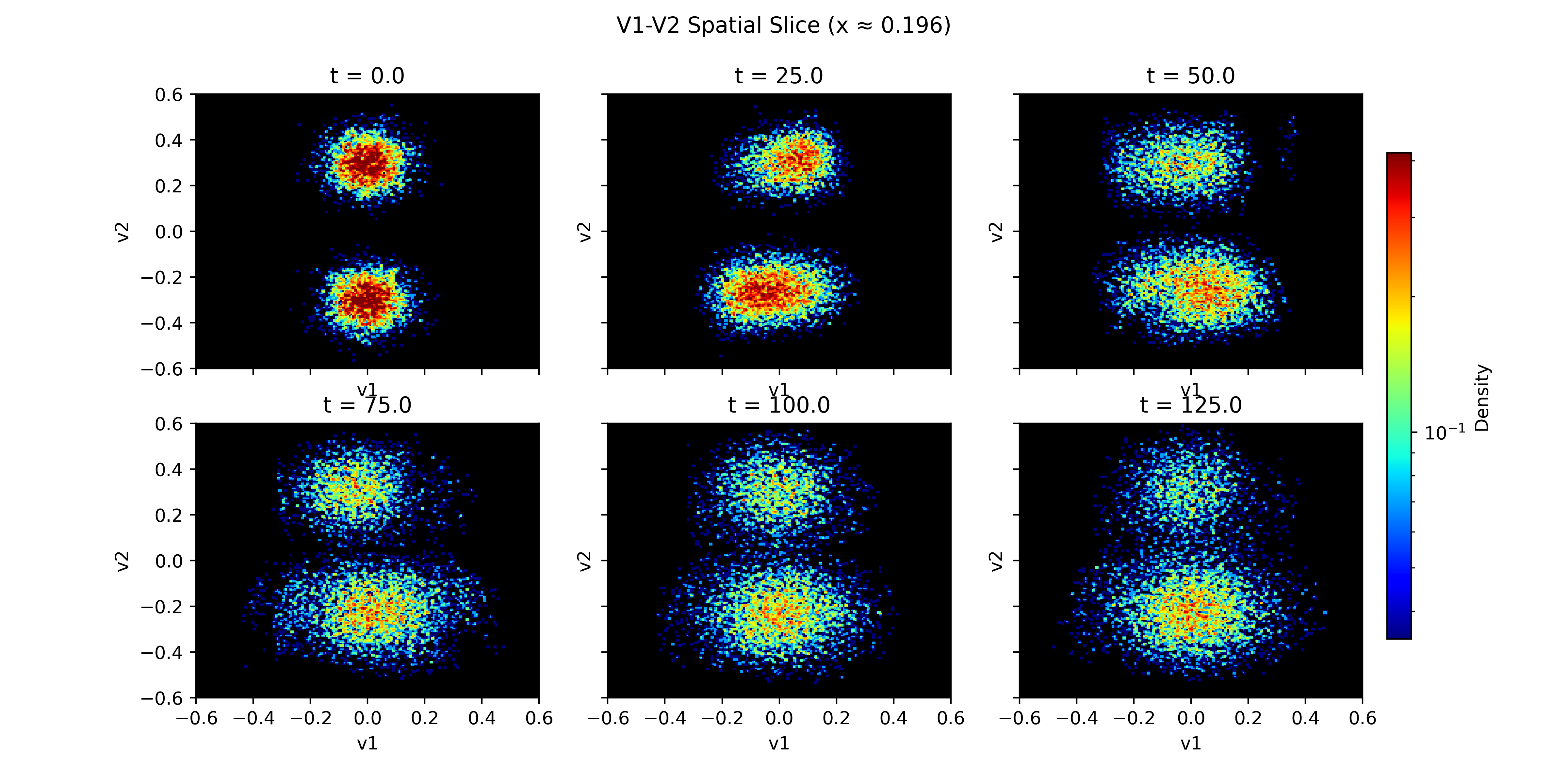}
\caption{Blob, $\nu = 10^{-4}$}
\end{subfigure}
\hfill
\begin{subfigure}{0.48\linewidth}
\includegraphics[width=\linewidth]{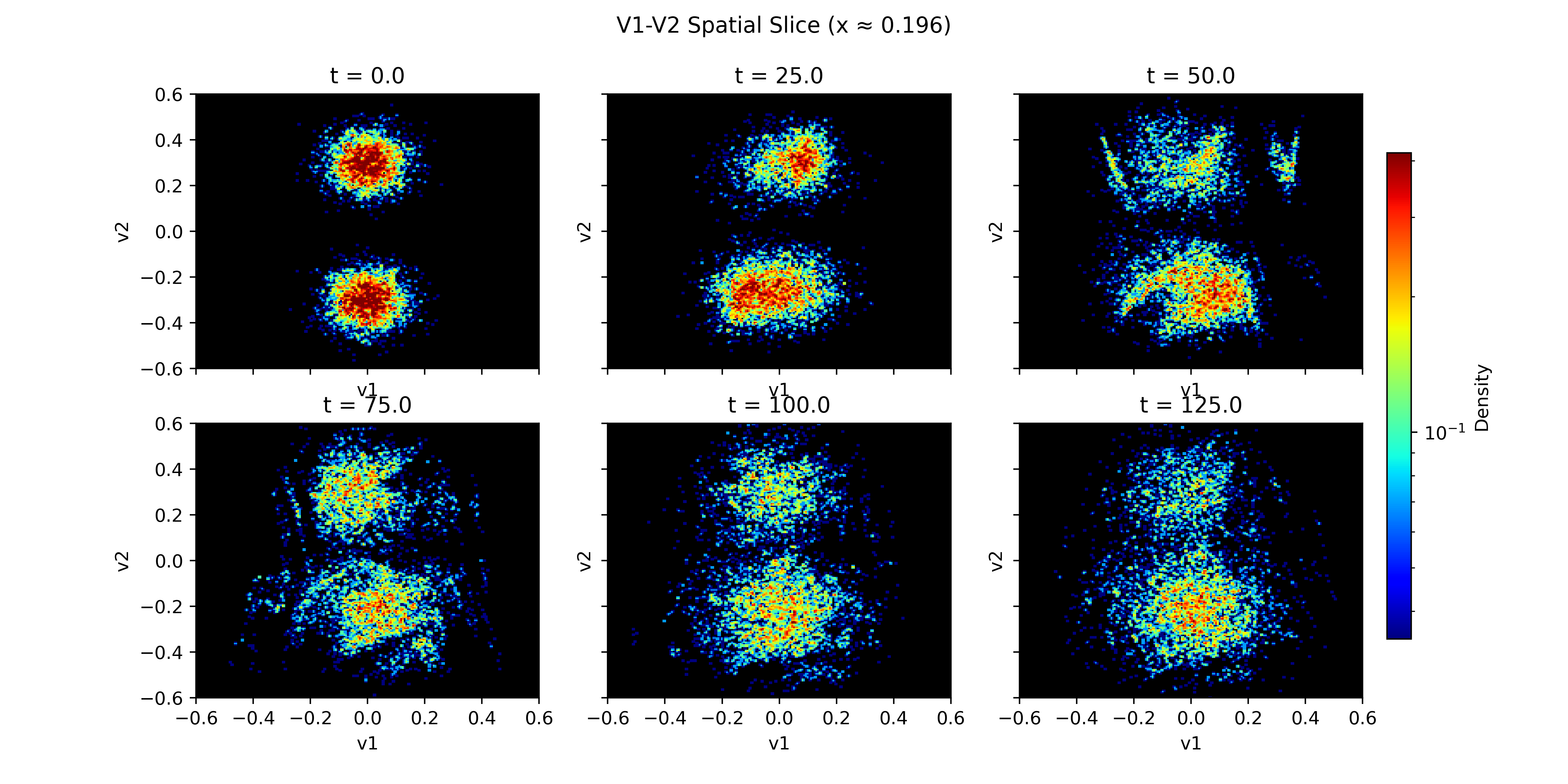}
\caption{SBTM, $\nu = 10^{-4}$}
\end{subfigure}
\caption{Weibel instability: $(v_1, v_2)$-marginal density.  (a)--(d): phase space snapshots at $\nu = 10^{-4}$ and $\nu = 8 \times 10^{-4}$.  (e)--(h): $(v_1, v_2)$ density slices at $x = 0$.  At $\nu = 0$ (e, f), both methods are identical.  At $\nu = 10^{-4}$ (g, h), the blob method over-smooths while SBTM preserves fine structures.}
\label{fig:weibel_phase}
\end{figure}

\textbf{Score estimation quality.}  Figure~\ref{fig:weibel_score} shows quiver plots of the estimated score in $(v_1, v_2)$ at the initial and final times.  At $t = 0$, SBTM achieves low MSE against the known analytic score; the blob method has substantially higher MSE.  At $t = t_{\mathrm{final}}$, SBTM scores are consistent with the expected linear score of a Gaussian ($s(x,v) \propto -v$), while the blob method produces noisy scores in the tails.

\begin{figure}[h]
\centering
\begin{subfigure}{0.48\linewidth}
\includegraphics[width=\linewidth]{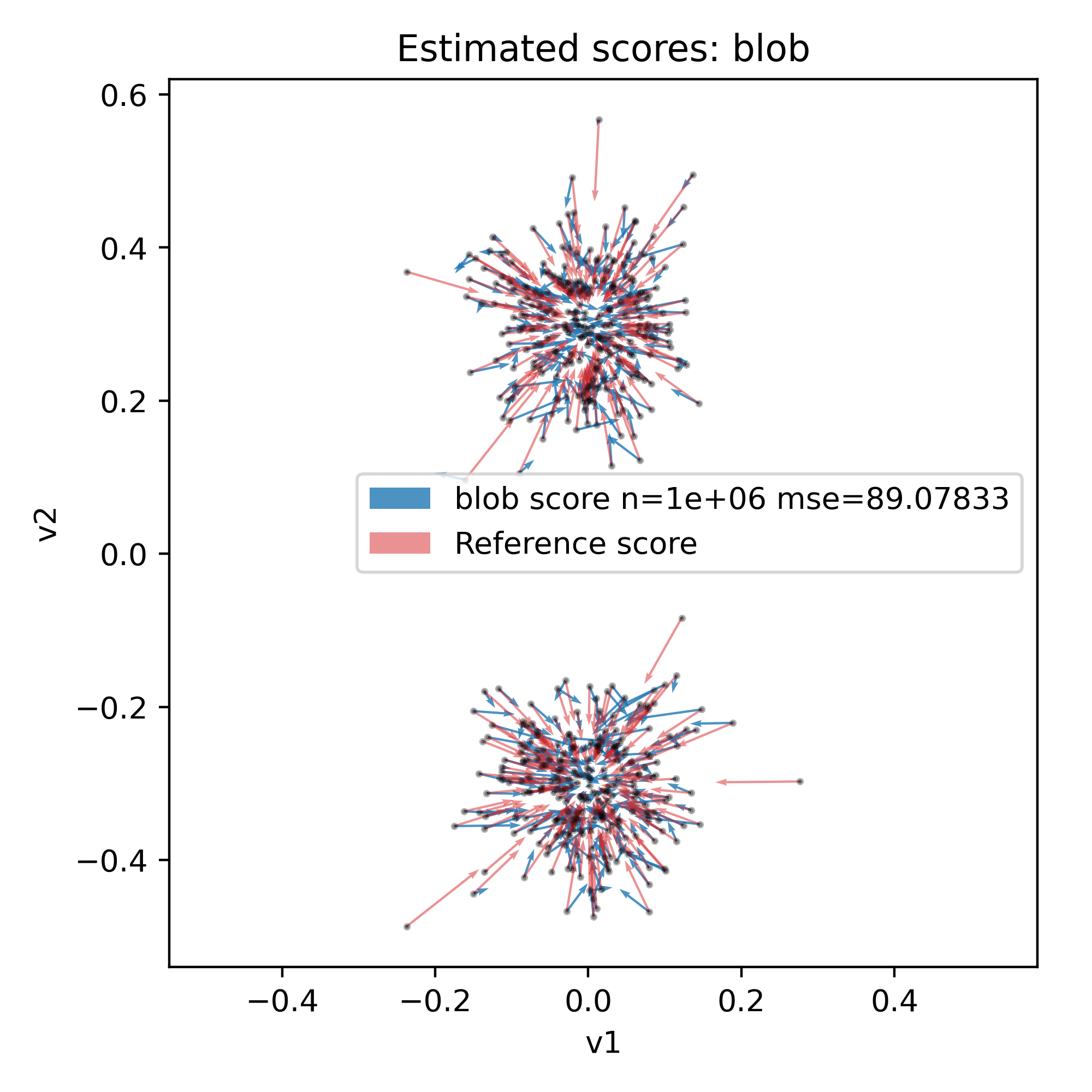}
\caption{Blob, $t = 0$}
\end{subfigure}
\hfill
\begin{subfigure}{0.48\linewidth}
\includegraphics[width=\linewidth]{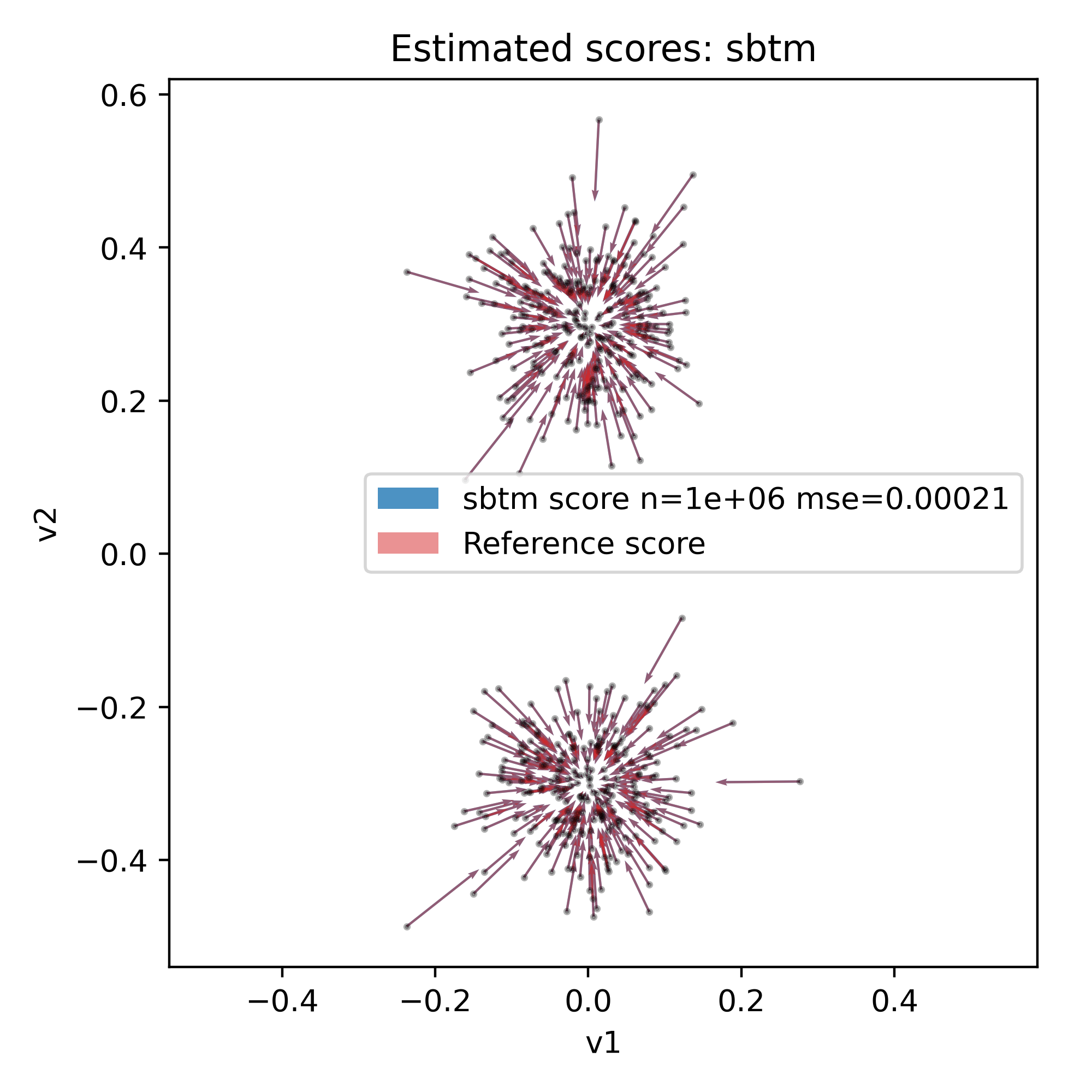}
\caption{SBTM, $t = 0$}
\end{subfigure}
\begin{subfigure}{0.48\linewidth}
\includegraphics[width=\linewidth]{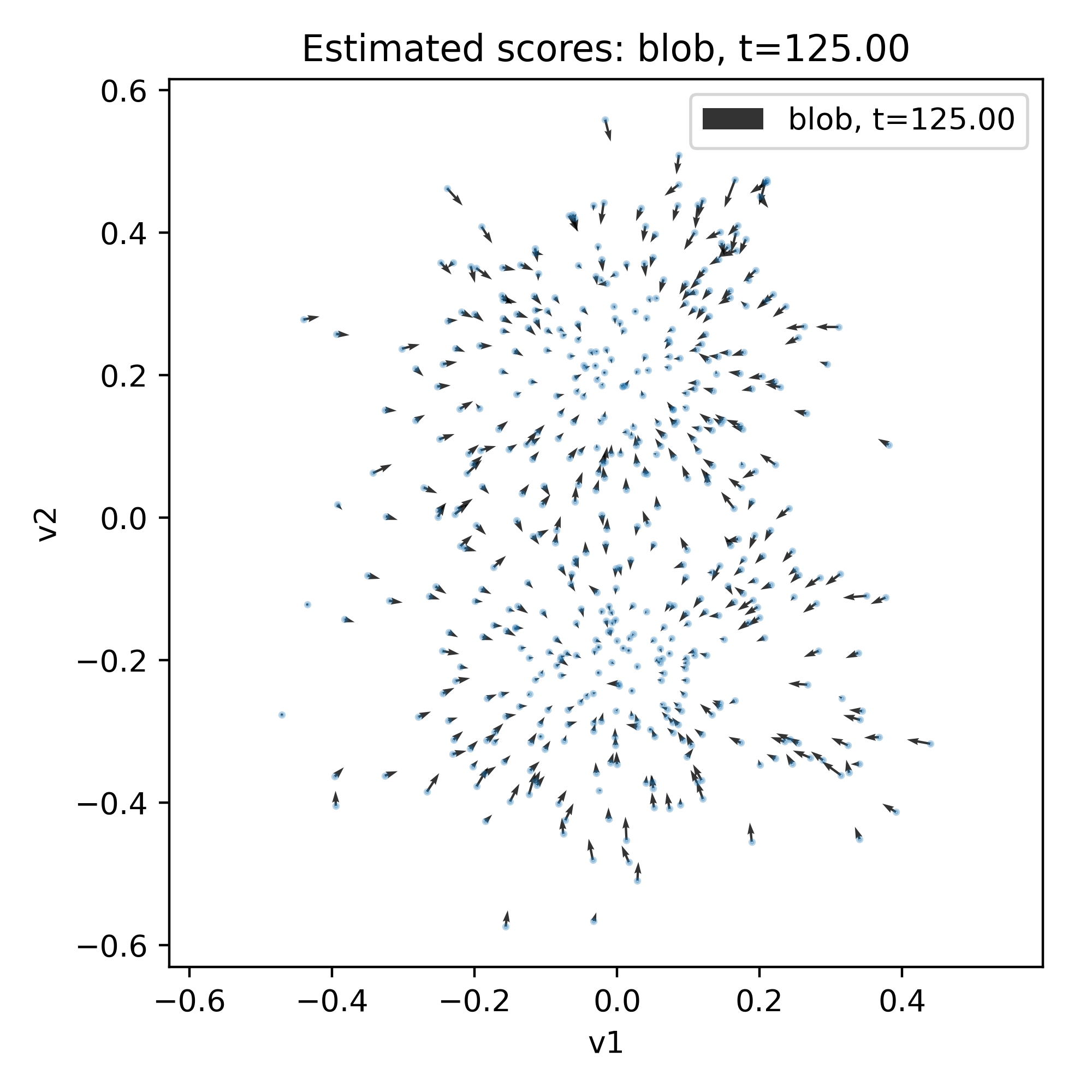}
\caption{Blob, $t = t_{\mathrm{final}}$}
\end{subfigure}
\hfill
\begin{subfigure}{0.48\linewidth}
\includegraphics[width=\linewidth]{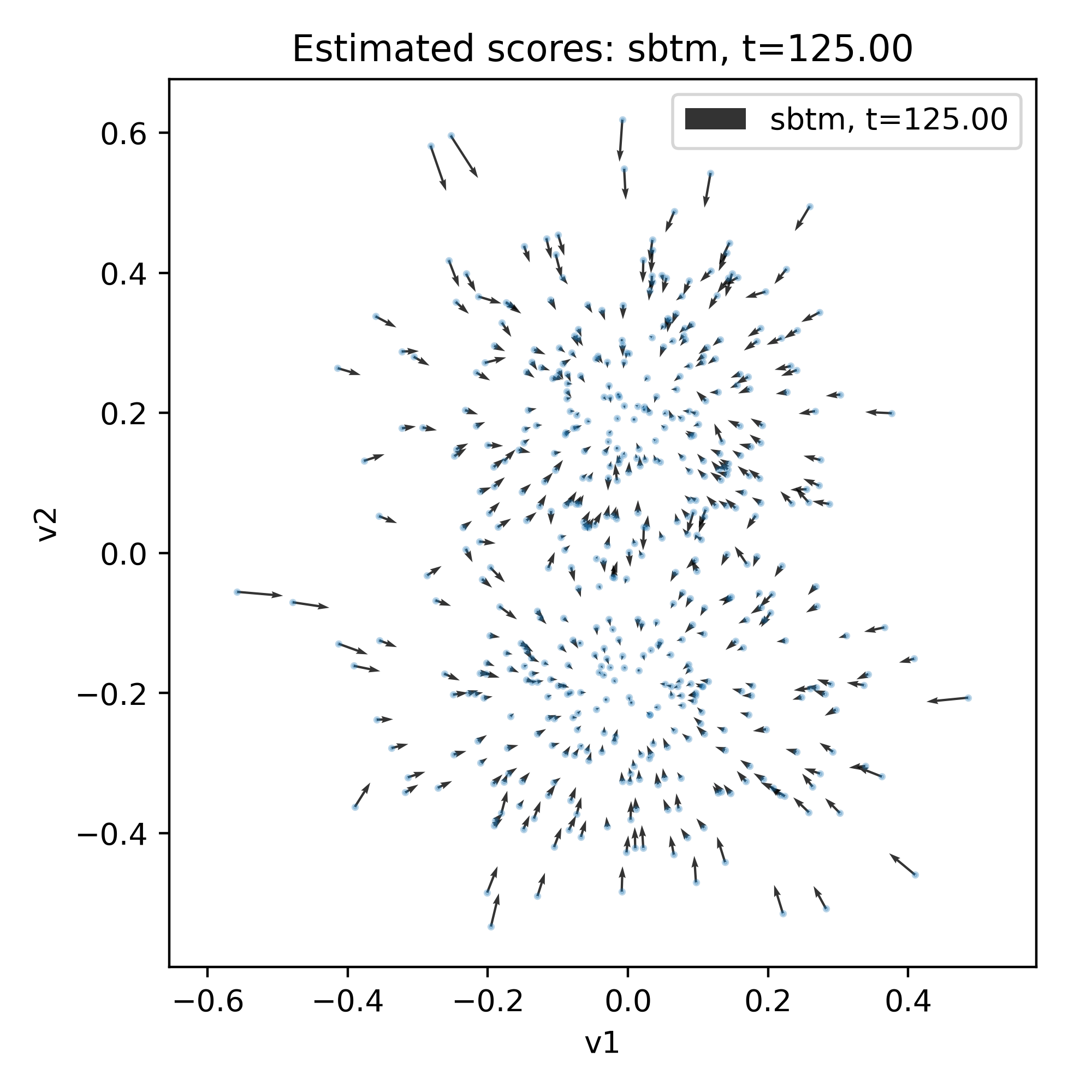}
\caption{SBTM, $t = t_{\mathrm{final}}$}
\end{subfigure}
\caption{Weibel instability: score quiver plots at $\nu = 10^{-4}$.  Top: initial time (true score in red; MSE in legend).  Bottom: final time (true score unavailable).  SBTM produces smooth, physically consistent scores; the blob method has high MSE at $t = 0$ and noisy scores at $t = t_{\mathrm{final}}$.}
\label{fig:weibel_score}
\end{figure}

\textbf{Estimated entropy production.}  Figure~\ref{fig:weibel_entropy} shows the estimated entropy production over time. SBTM correctly decays toward zero at all collision frequencies, with higher $\nu$ producing faster decay.  The blob method plateaus at nonzero values due to persistent score errors in low-density tail regions.  At fixed $\nu = 8\times 10^{-4}$, the blob plateau decreases with $n$, converging toward SBTM.

\begin{figure}[h]
\centering
\begin{subfigure}{0.48\linewidth}
\includegraphics[width=\linewidth]{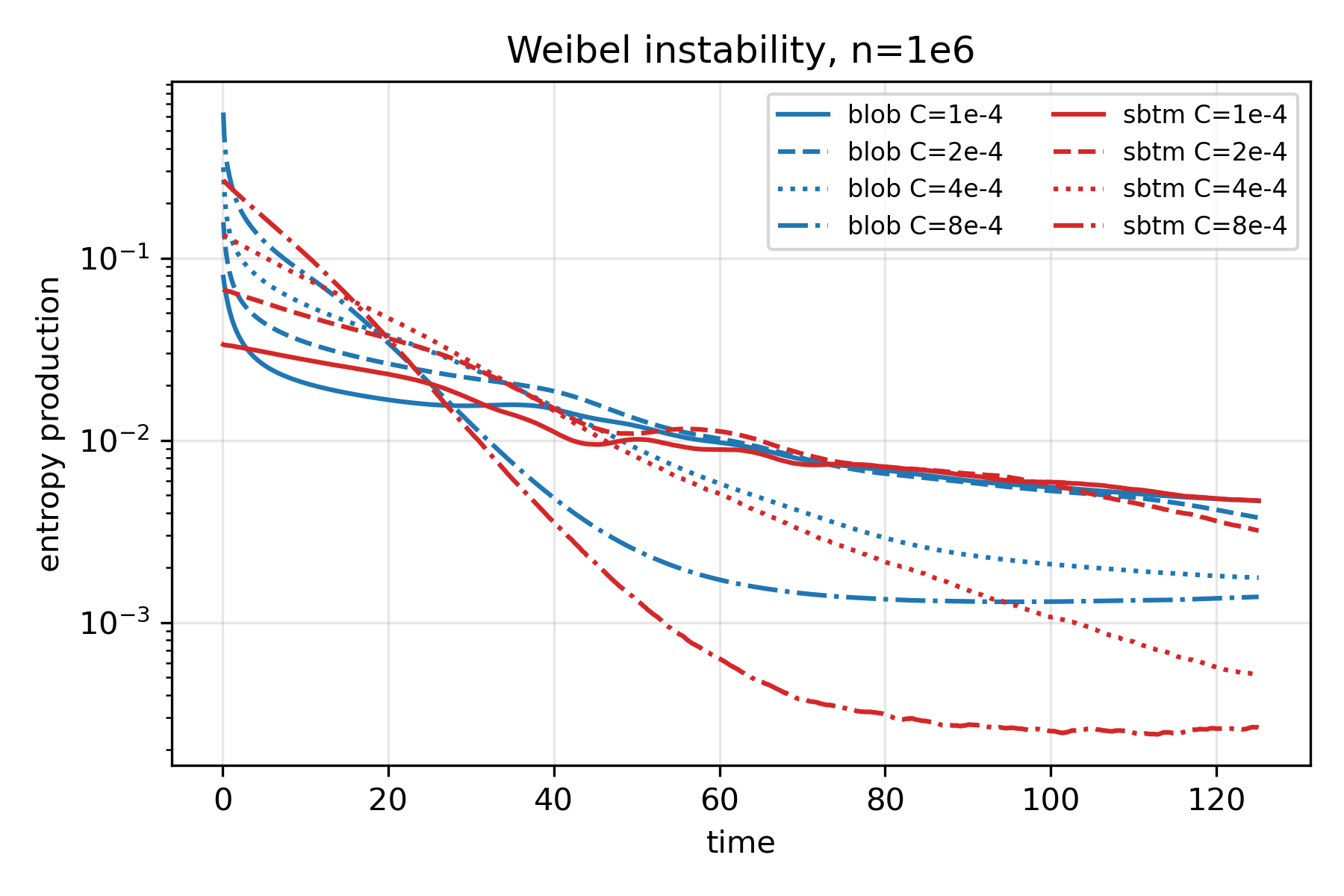}
\caption{Sweep over $\nu$ ($n = 10^6$)}
\end{subfigure}
\hfill
\begin{subfigure}{0.48\linewidth}
\includegraphics[width=\linewidth]{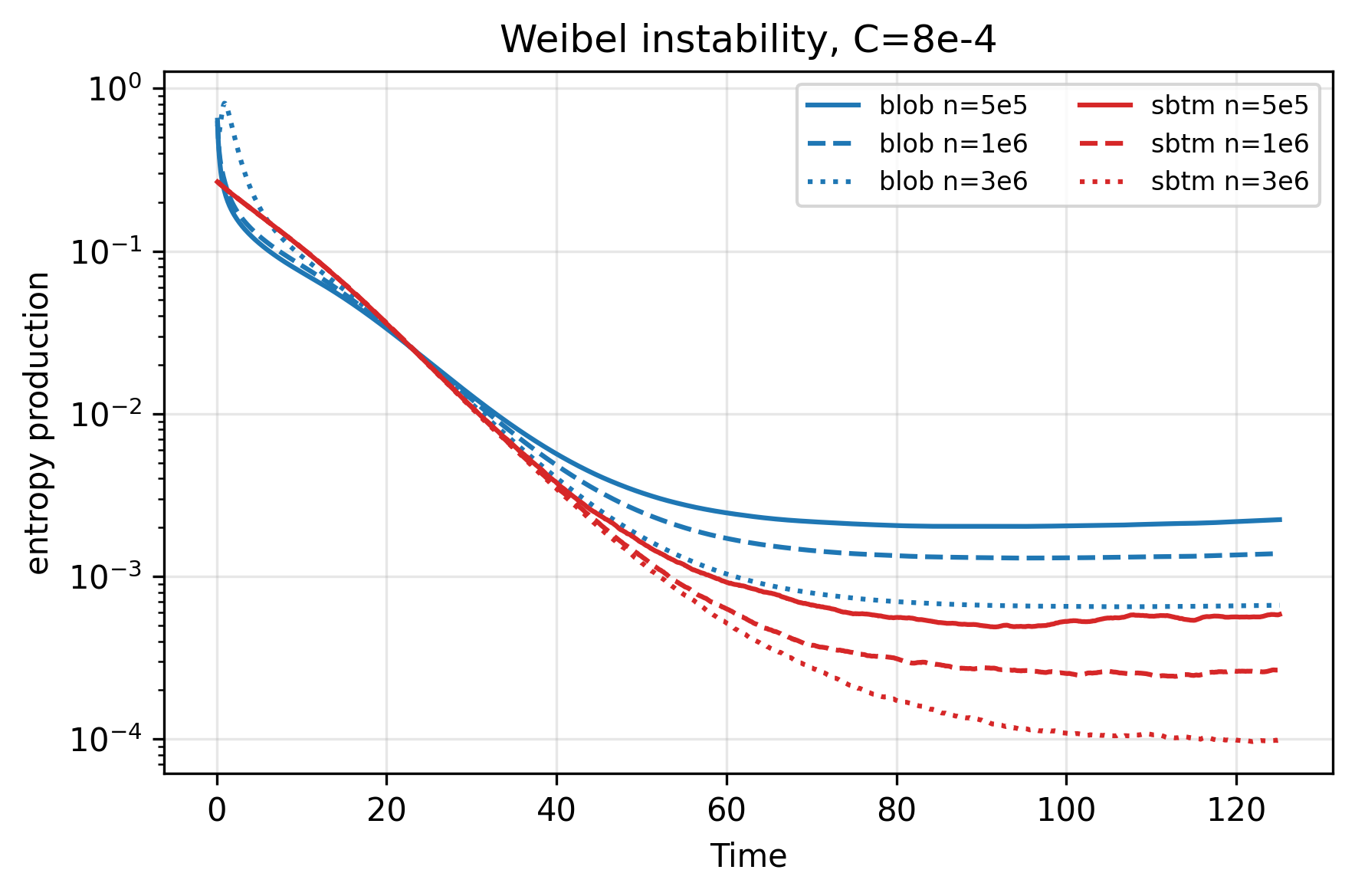}
\caption{Sweep over $n$ ($\nu = 8\times 10^{-4}$)}
\end{subfigure}
\caption{Weibel instability: estimated entropy production for blob (blue) and SBTM (red).  Left: across collision frequencies; SBTM correctly decays to zero while the blob method plateaus.  Right: across particle counts at $\nu = 8\times 10^{-4}$; the blob plateau decreases with $n$, converging toward SBTM.}
\label{fig:weibel_entropy}
\end{figure}

%

\textbf{Phase space sweep over $\nu$ and $d_v$.}  Figures~\ref{fig:weibel_sweep_dv2} and~\ref{fig:weibel_sweep_dv3} show $(v_1, v_2)$-marginal density snapshots for the Weibel instability across collision frequencies $\nu \in \{0, 10^{-4}, 2 \times 10^{-4}, 4 \times 10^{-4}, 8 \times 10^{-4}\}$ in $d_v = 2$ and $d_v = 3$ dimensions, respectively.  At $\nu = 0$, both methods are identical.  As $\nu$ increases, the blob method increasingly produces non-Gaussian densities, while SBTM equilibrates correctly.  The effect is more pronounced in $d_v = 3$, indicating better dimension scaling of SBTM.

\begin{figure}[h]
\centering
\begin{subfigure}{0.48\linewidth}
\includegraphics[width=\linewidth]{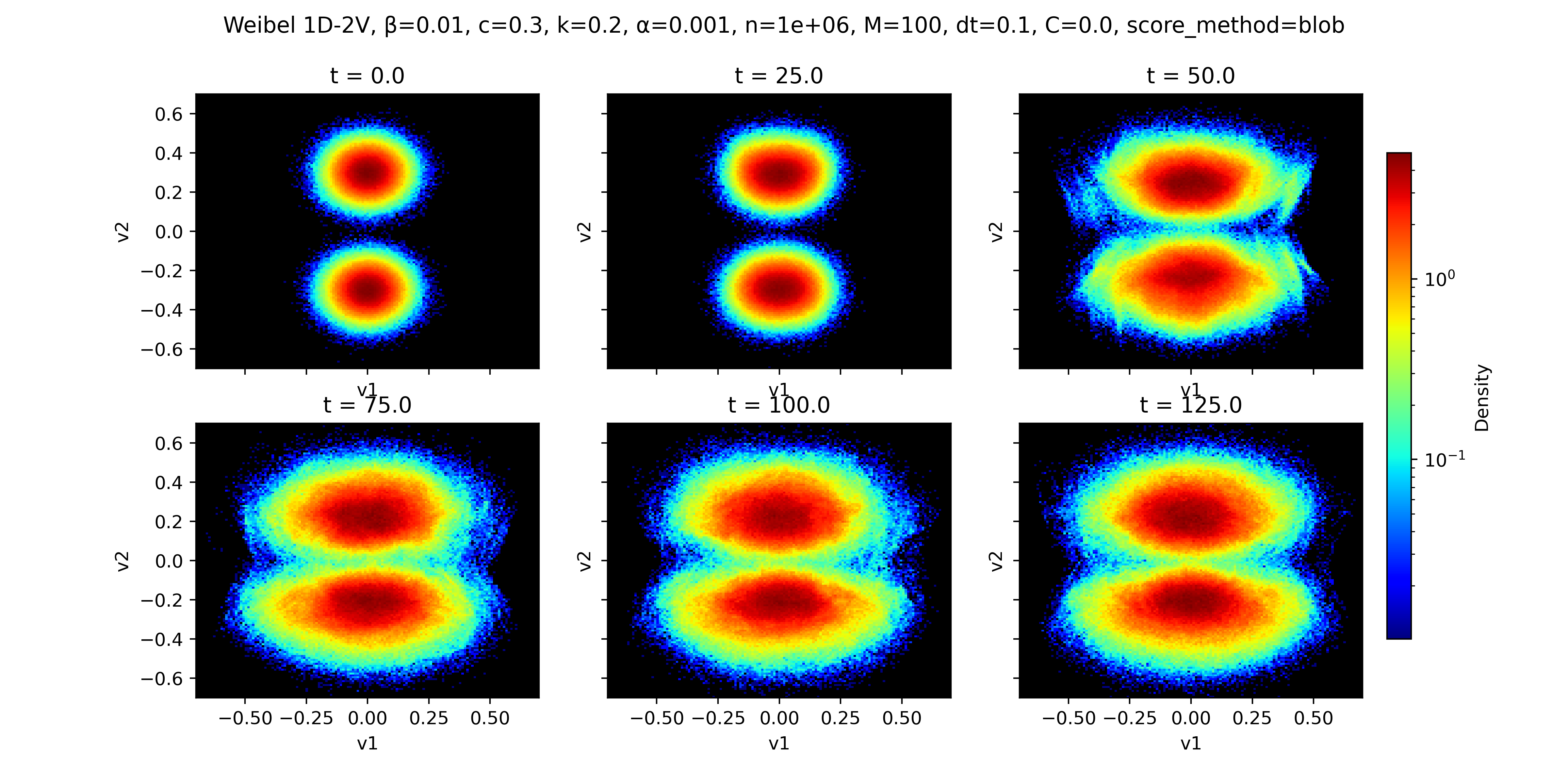}
\caption{Blob, $\nu = 0$}
\end{subfigure}
\hfill
\begin{subfigure}{0.48\linewidth}
\includegraphics[width=\linewidth]{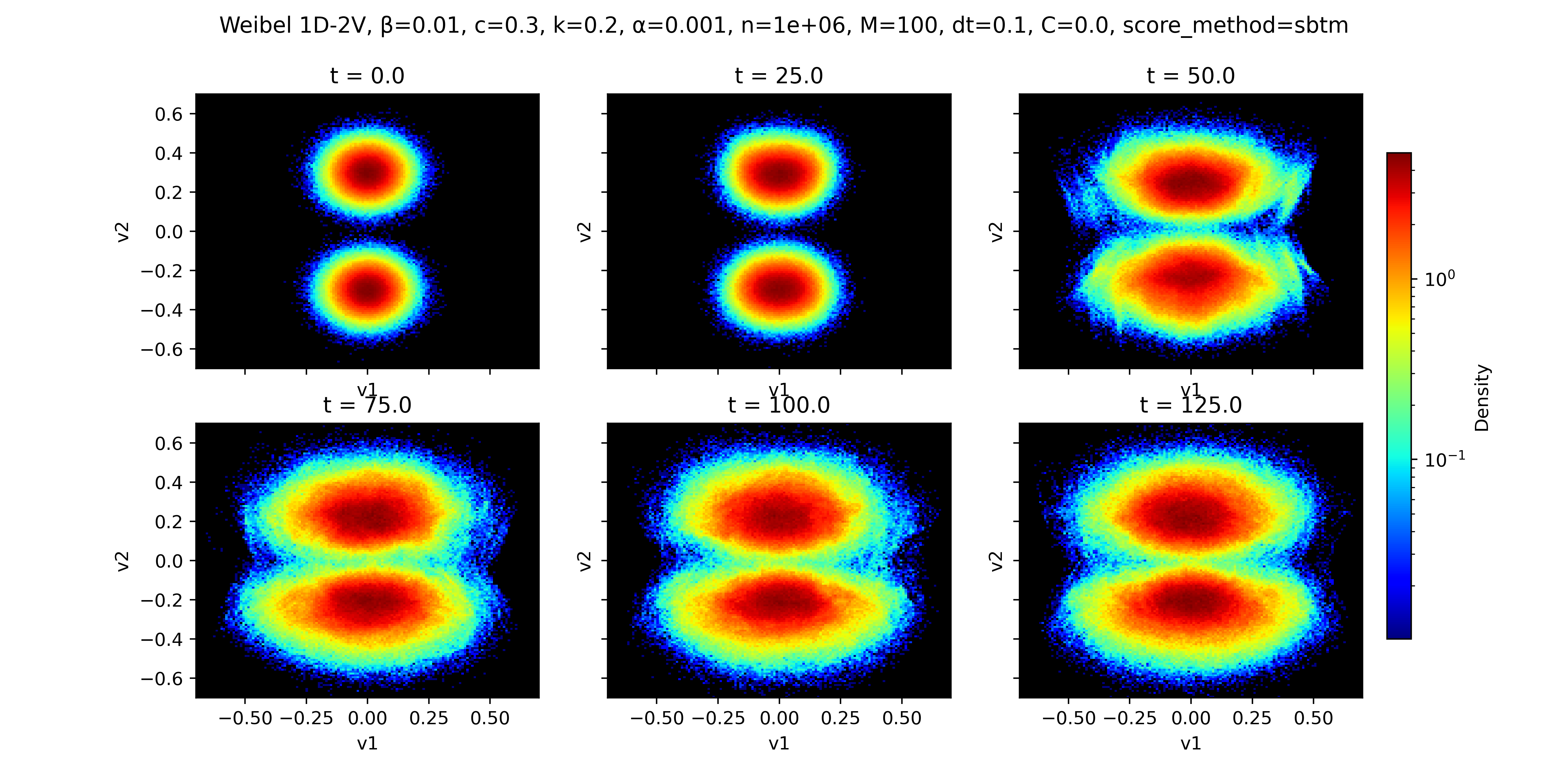}
\caption{SBTM, $\nu = 0$}
\end{subfigure}
\begin{subfigure}{0.48\linewidth}
\includegraphics[width=\linewidth]{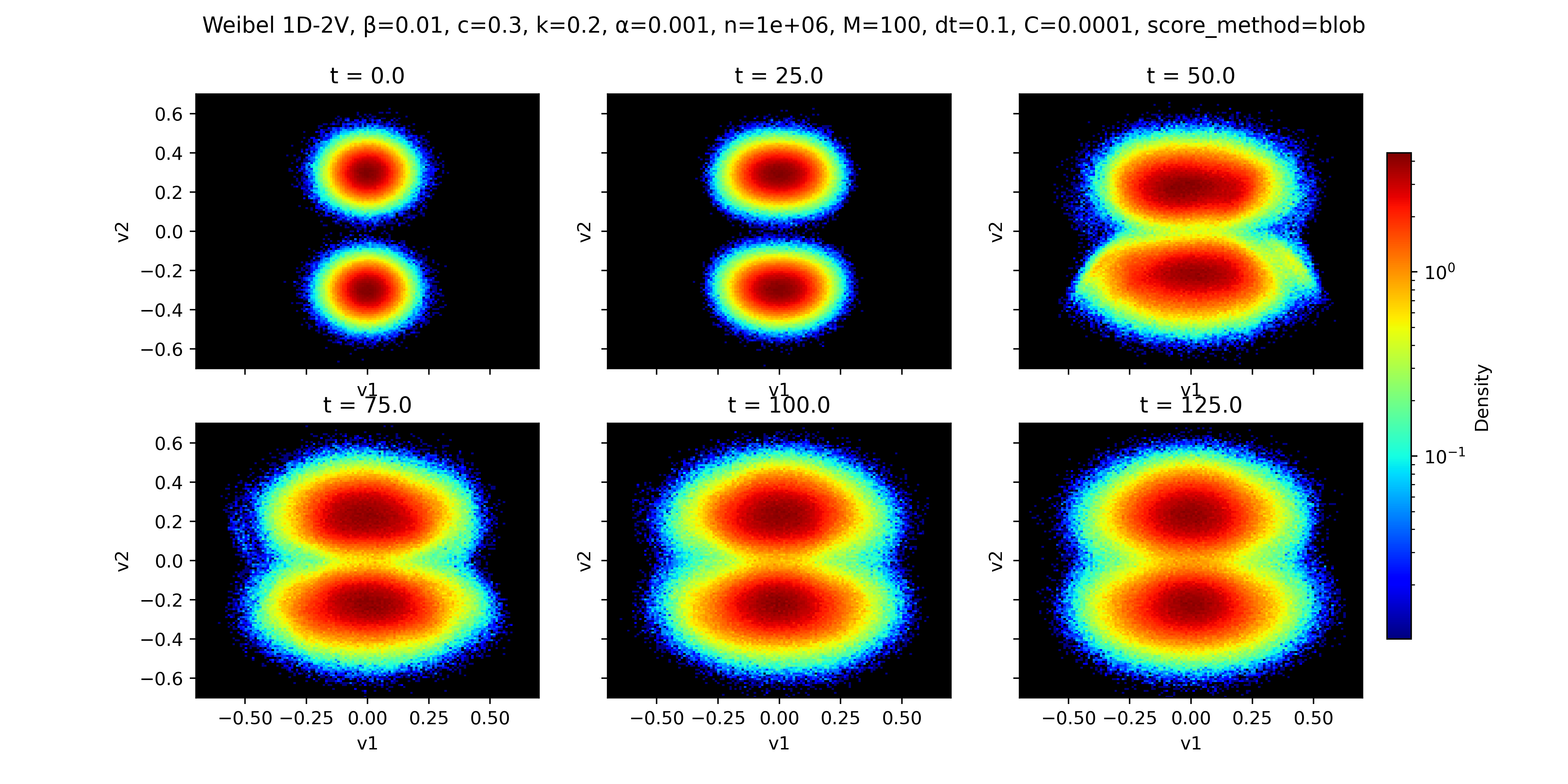}
\caption{Blob, $\nu = 10^{-4}$}
\end{subfigure}
\hfill
\begin{subfigure}{0.48\linewidth}
\includegraphics[width=\linewidth]{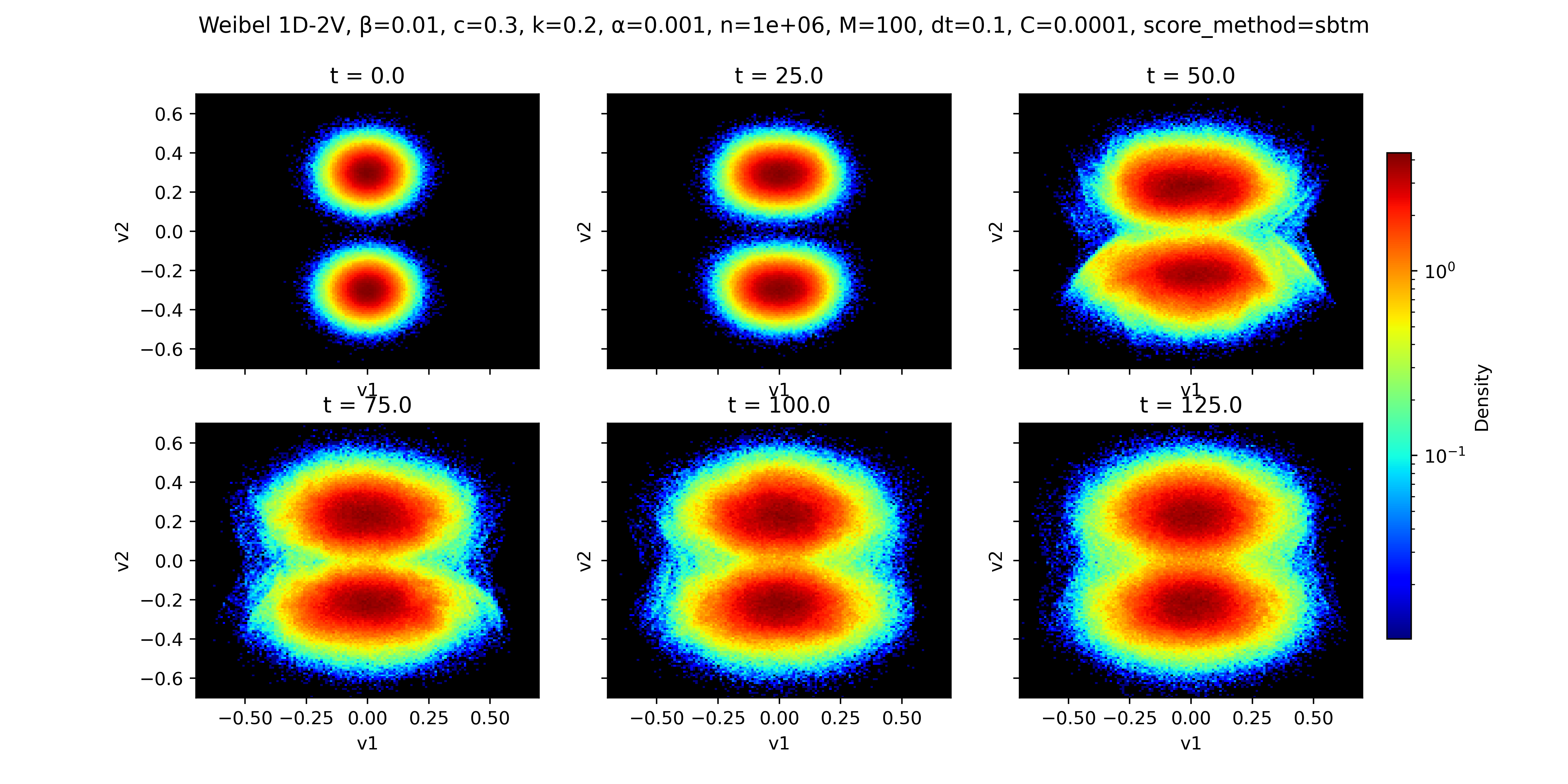}
\caption{SBTM, $\nu = 10^{-4}$}
\end{subfigure}
\begin{subfigure}{0.48\linewidth}
\includegraphics[width=\linewidth]{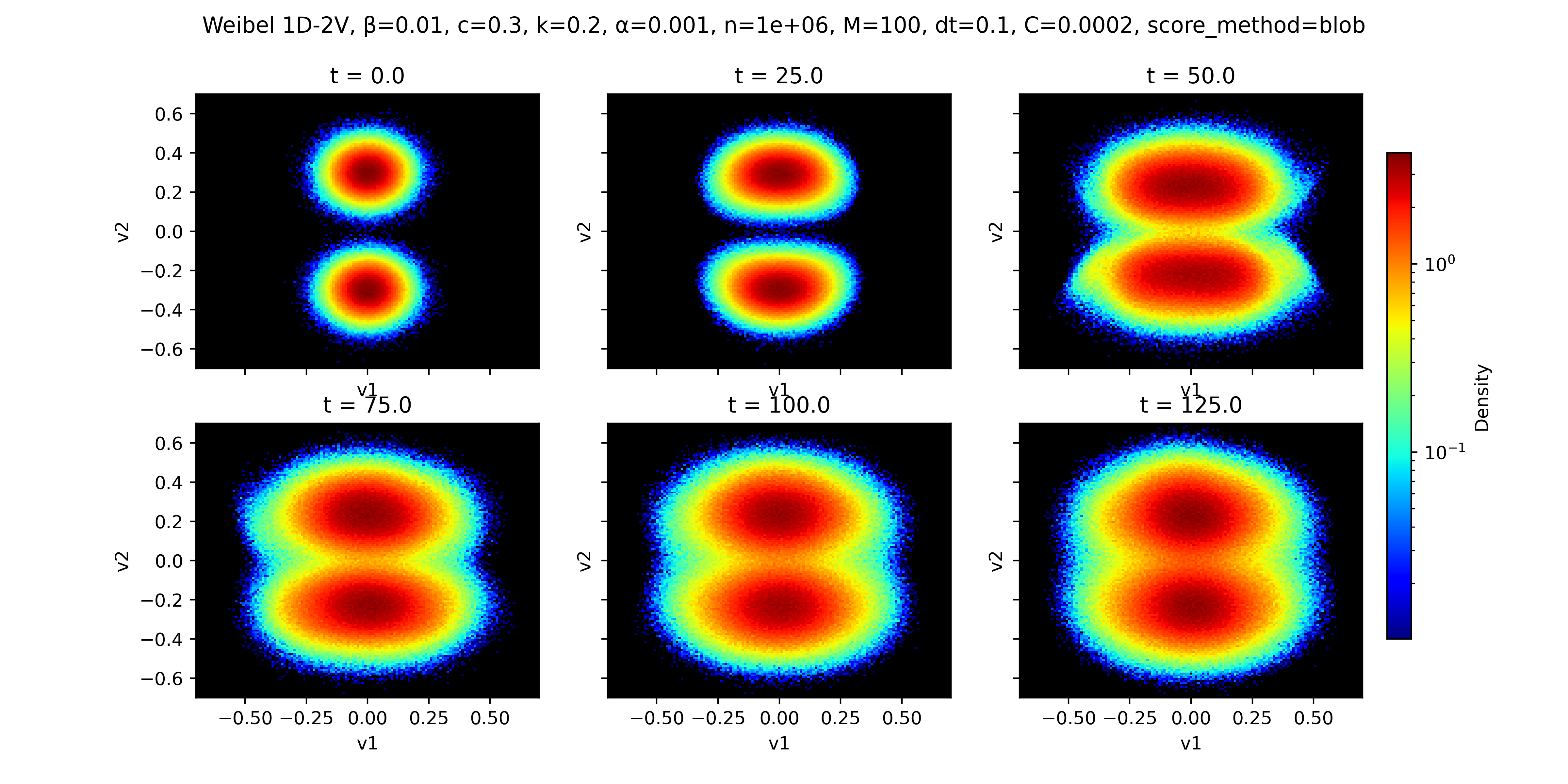}
\caption{Blob, $\nu = 2 \times 10^{-4}$}
\end{subfigure}
\hfill
\begin{subfigure}{0.48\linewidth}
\includegraphics[width=\linewidth]{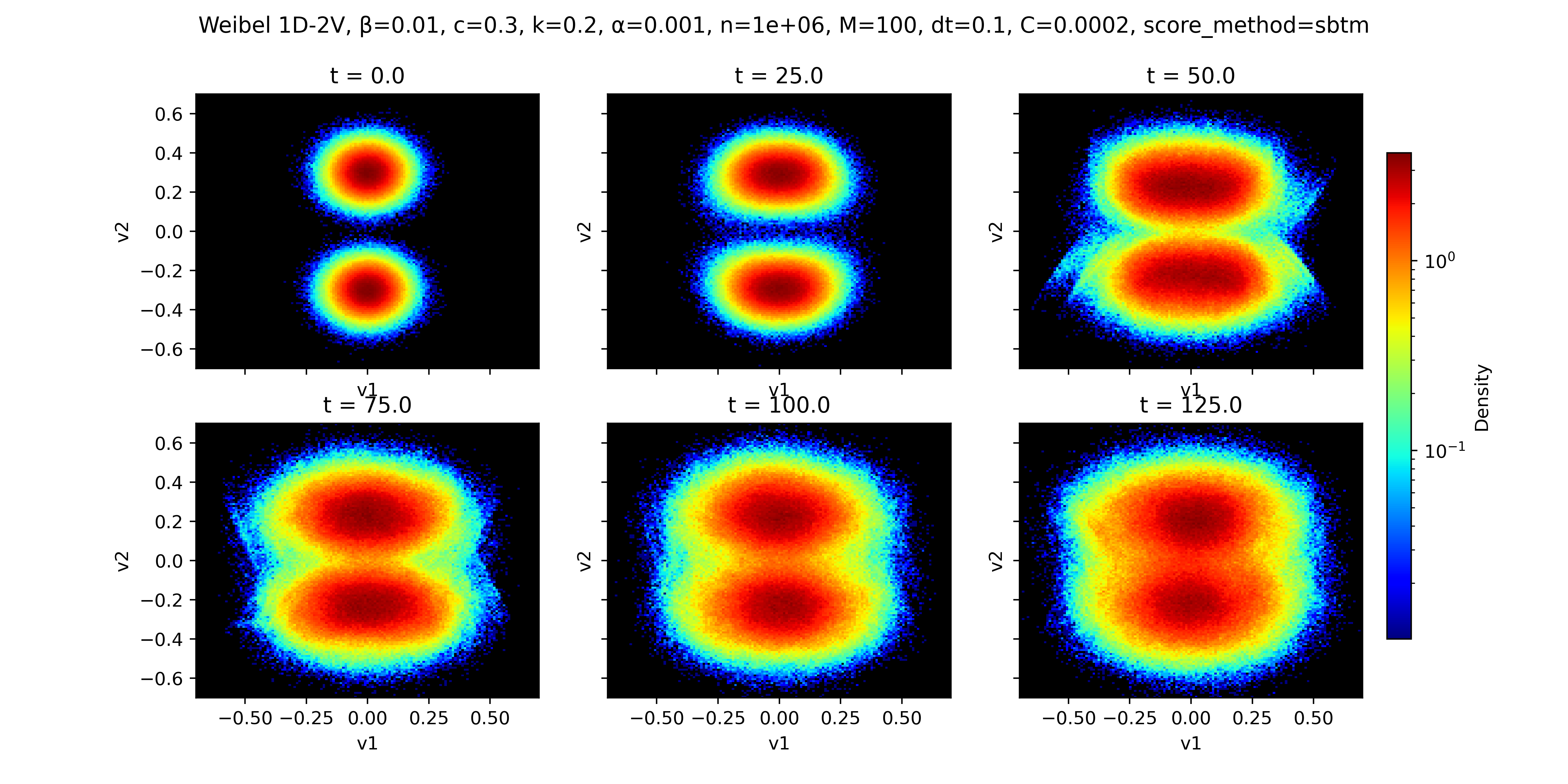}
\caption{SBTM, $\nu = 2 \times 10^{-4}$}
\end{subfigure}
\begin{subfigure}{0.48\linewidth}
\includegraphics[width=\linewidth]{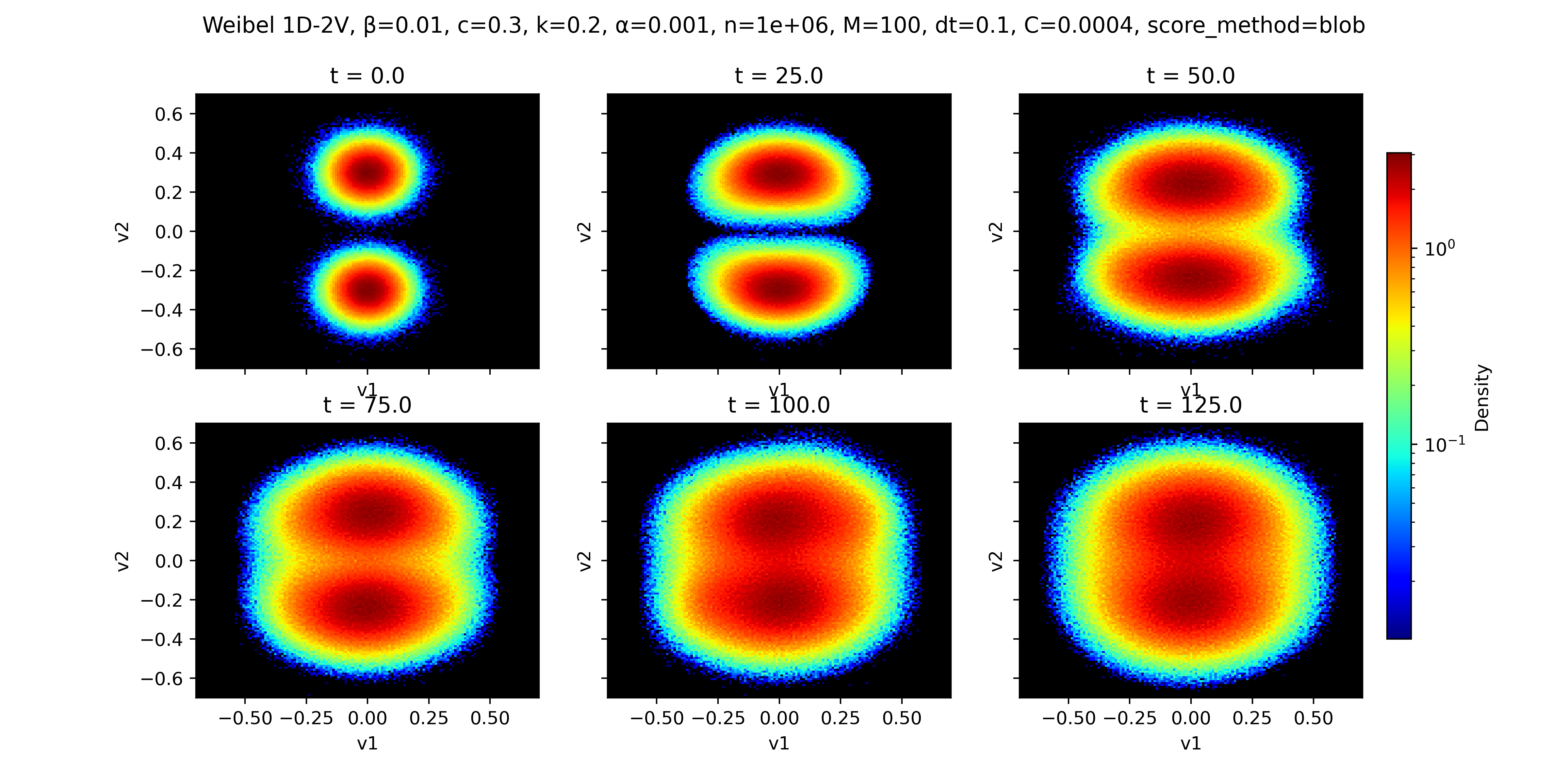}
\caption{Blob, $\nu = 4 \times 10^{-4}$}
\end{subfigure}
\hfill
\begin{subfigure}{0.48\linewidth}
\includegraphics[width=\linewidth]{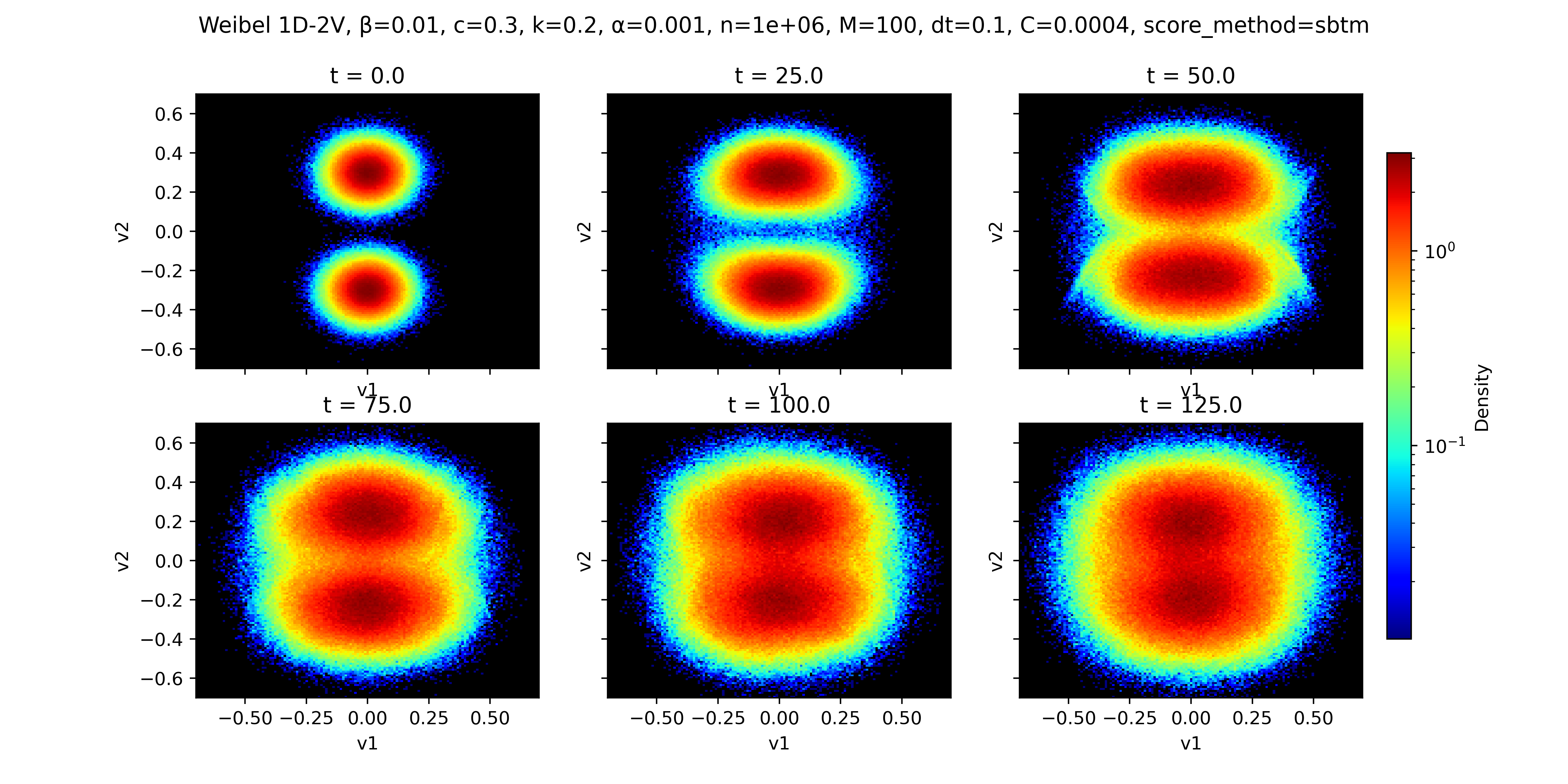}
\caption{SBTM, $\nu = 4 \times 10^{-4}$}
\end{subfigure}
\begin{subfigure}{0.48\linewidth}
\includegraphics[width=\linewidth]{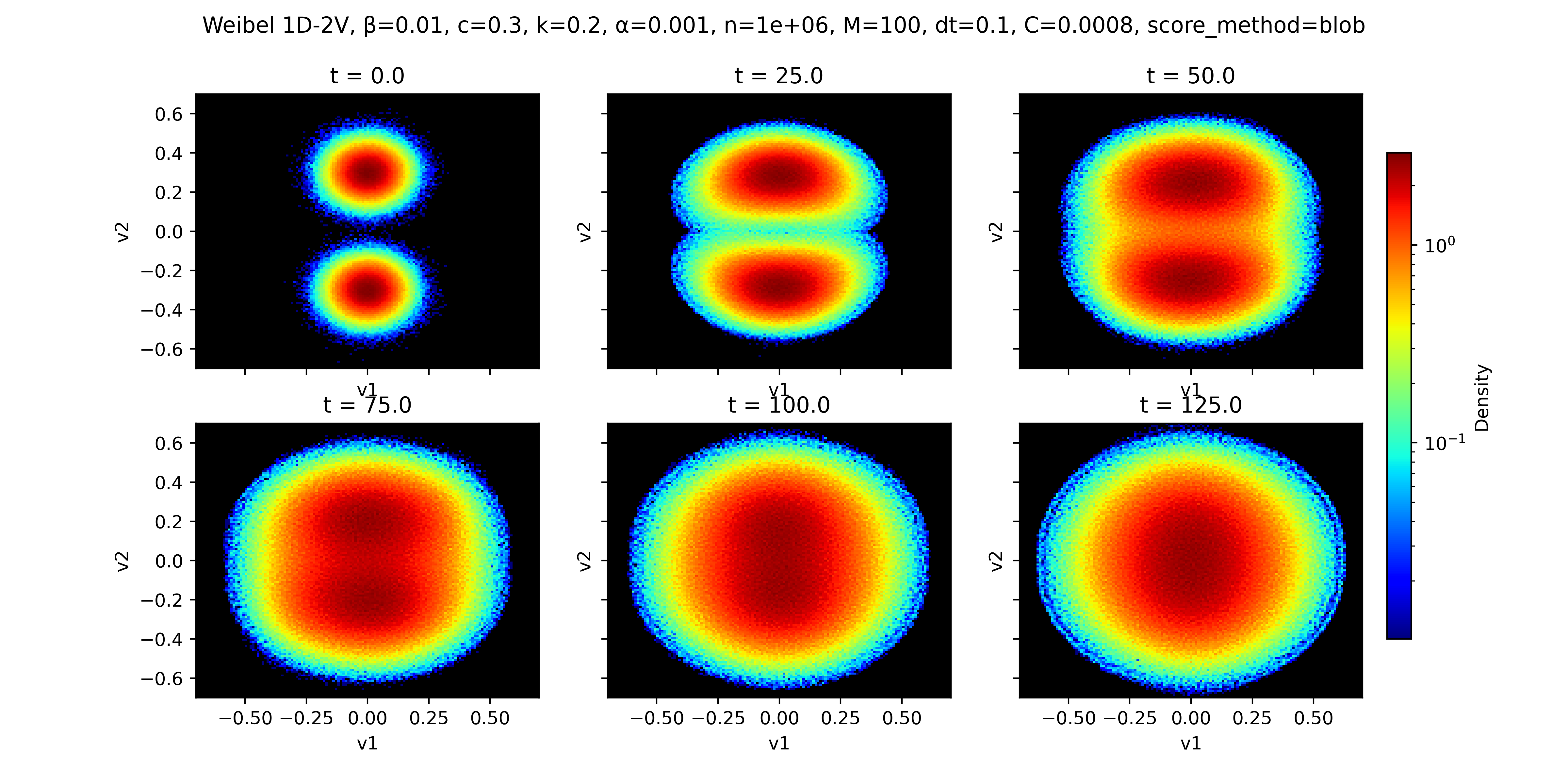}
\caption{Blob, $\nu = 8 \times 10^{-4}$}
\end{subfigure}
\hfill
\begin{subfigure}{0.48\linewidth}
\includegraphics[width=\linewidth]{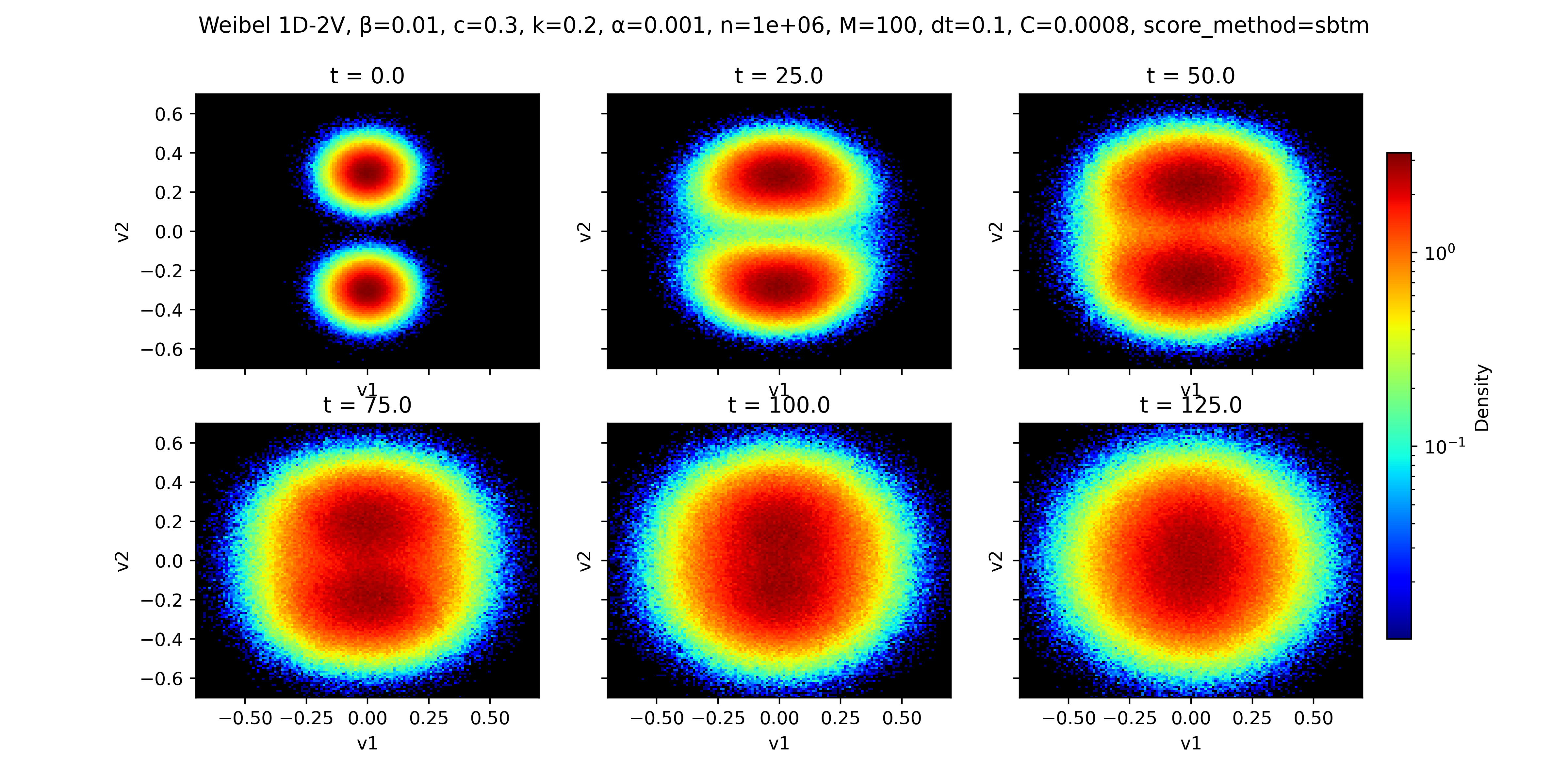}
\caption{SBTM, $\nu = 8 \times 10^{-4}$}
\end{subfigure}
\caption{Weibel instability, $d_v = 2$: $(v_1, v_2)$-marginal density across collision frequencies.  At $\nu = 0$ both methods are identical.  As $\nu$ increases, the blob method produces increasingly non-Gaussian densities, while SBTM equilibrates correctly.}
\label{fig:weibel_sweep_dv2}
\end{figure}

\begin{figure}[h]
\centering
\begin{subfigure}{0.48\linewidth}
\includegraphics[width=\linewidth]{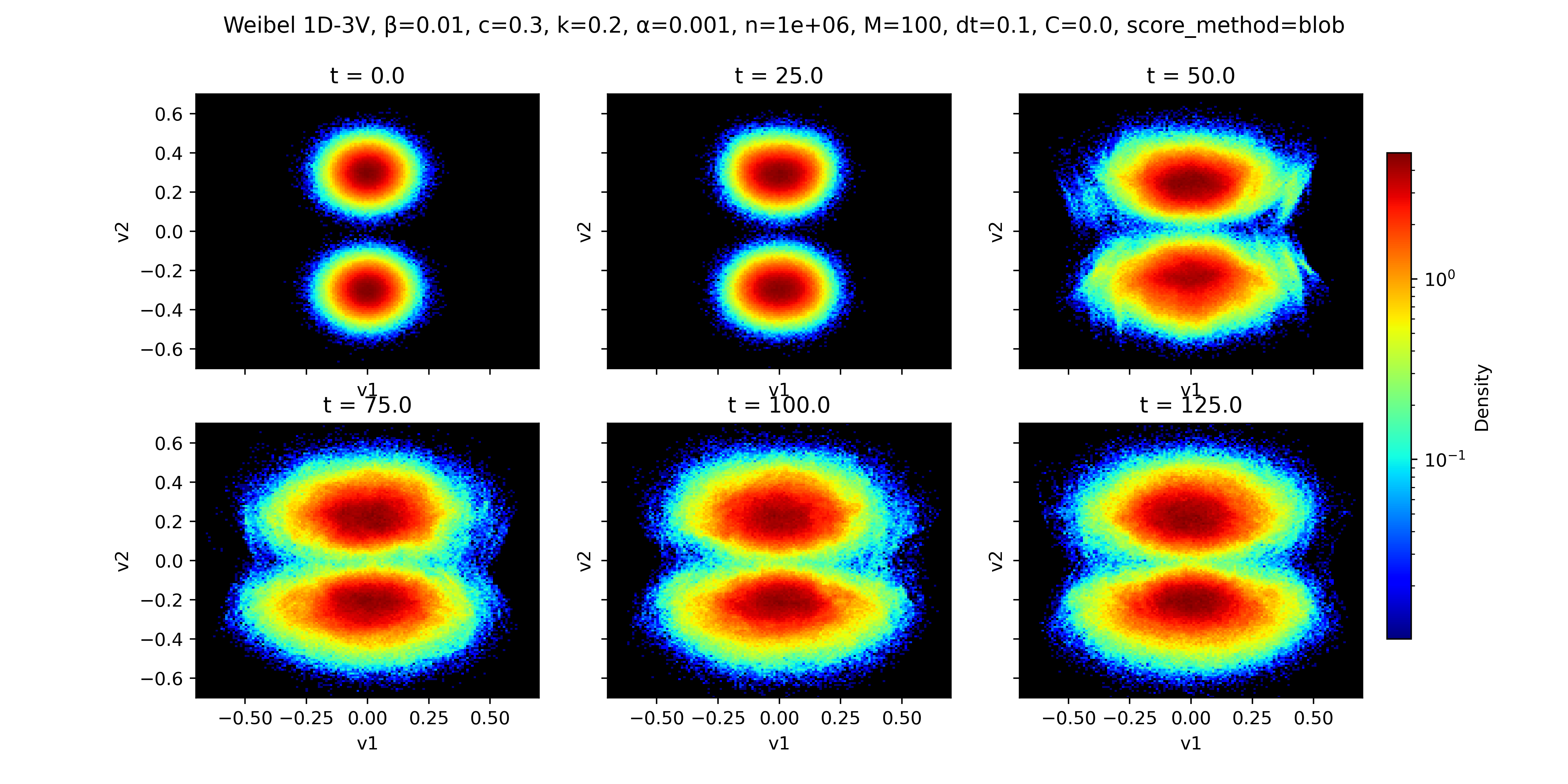}
\caption{Blob, $\nu = 0$}
\end{subfigure}
\hfill
\begin{subfigure}{0.48\linewidth}
\includegraphics[width=\linewidth]{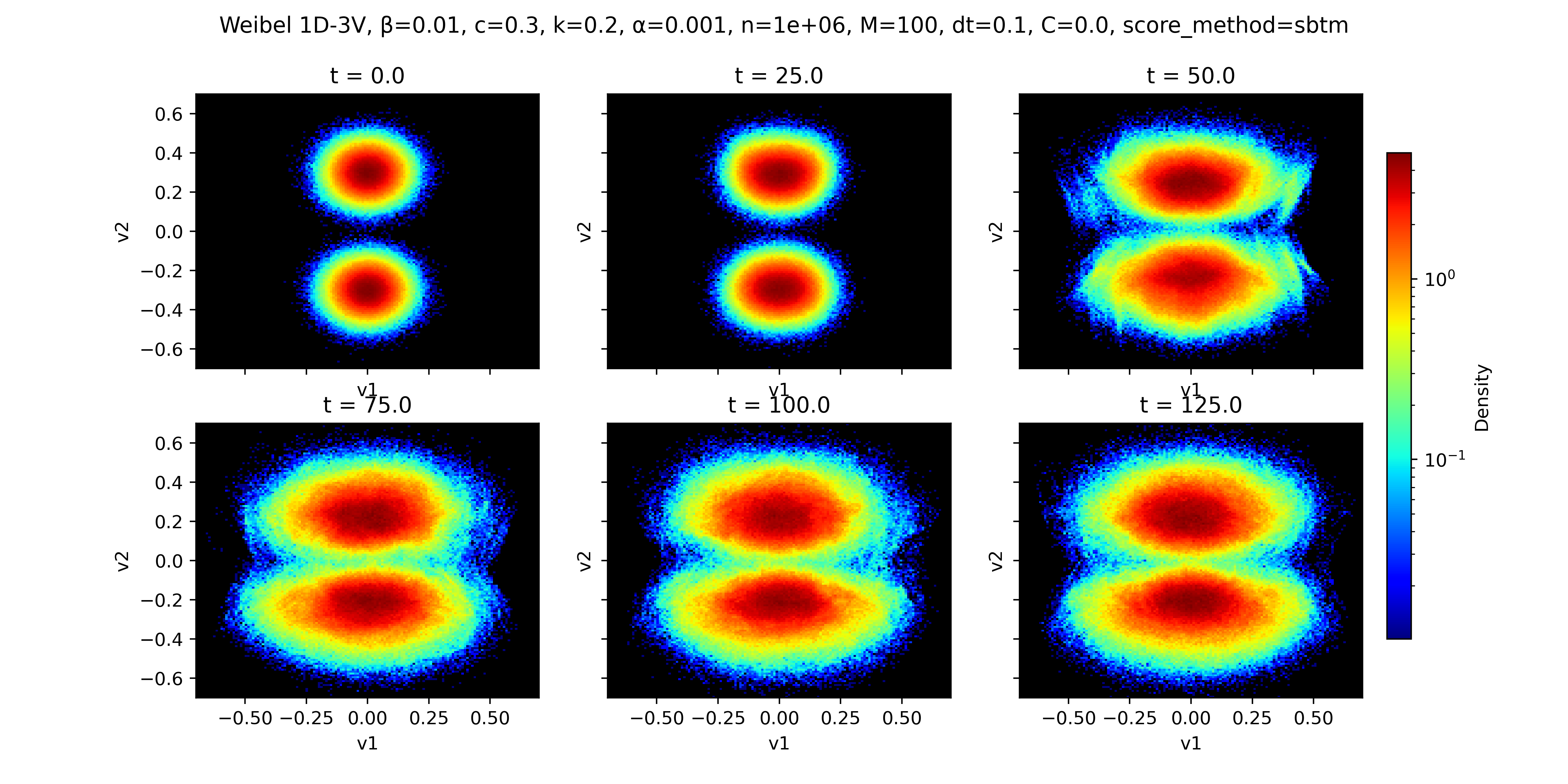}
\caption{SBTM, $\nu = 0$}
\end{subfigure}
\begin{subfigure}{0.48\linewidth}
\includegraphics[width=\linewidth]{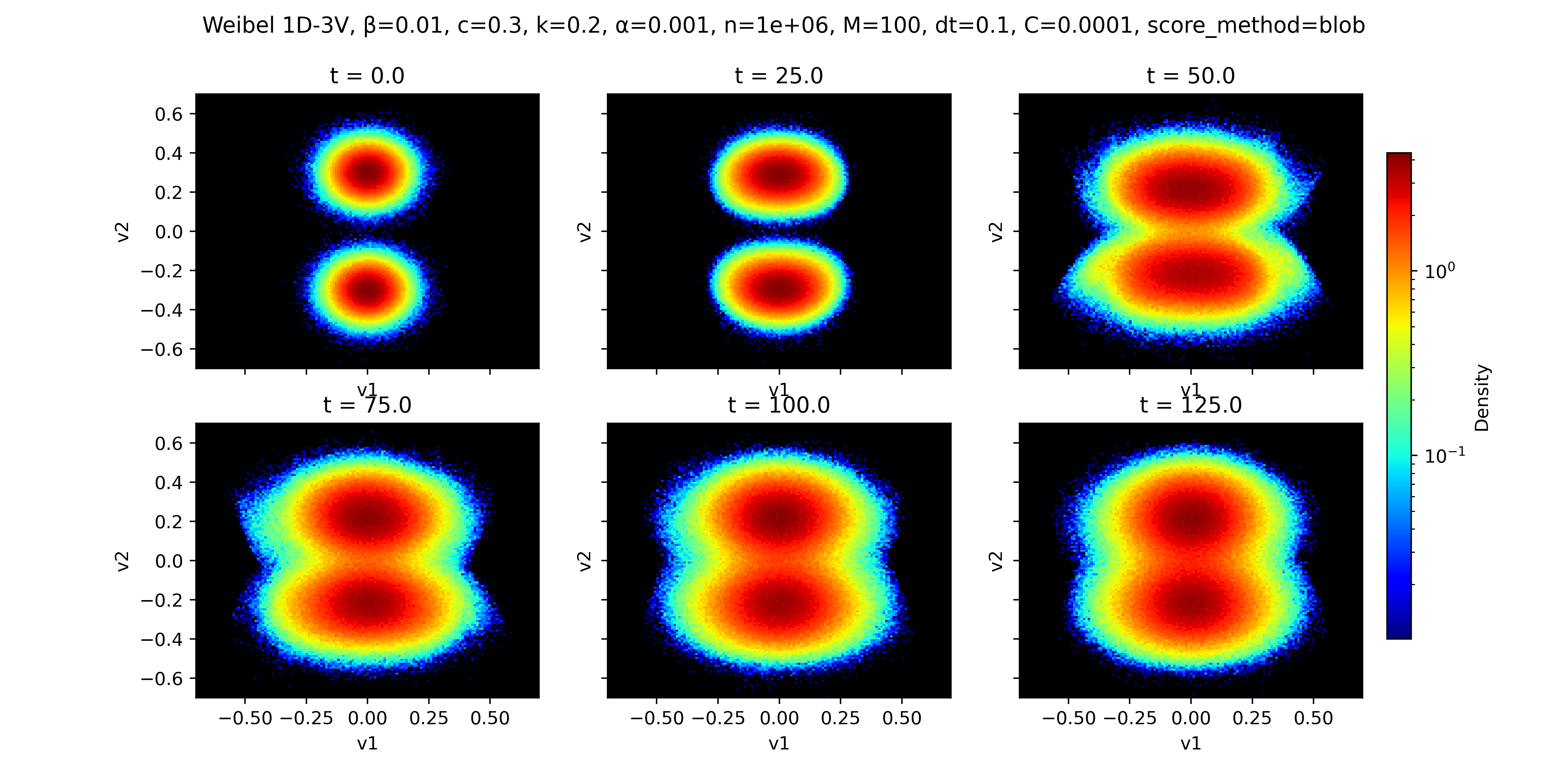}
\caption{Blob, $\nu = 10^{-4}$}
\end{subfigure}
\hfill
\begin{subfigure}{0.48\linewidth}
\includegraphics[width=\linewidth]{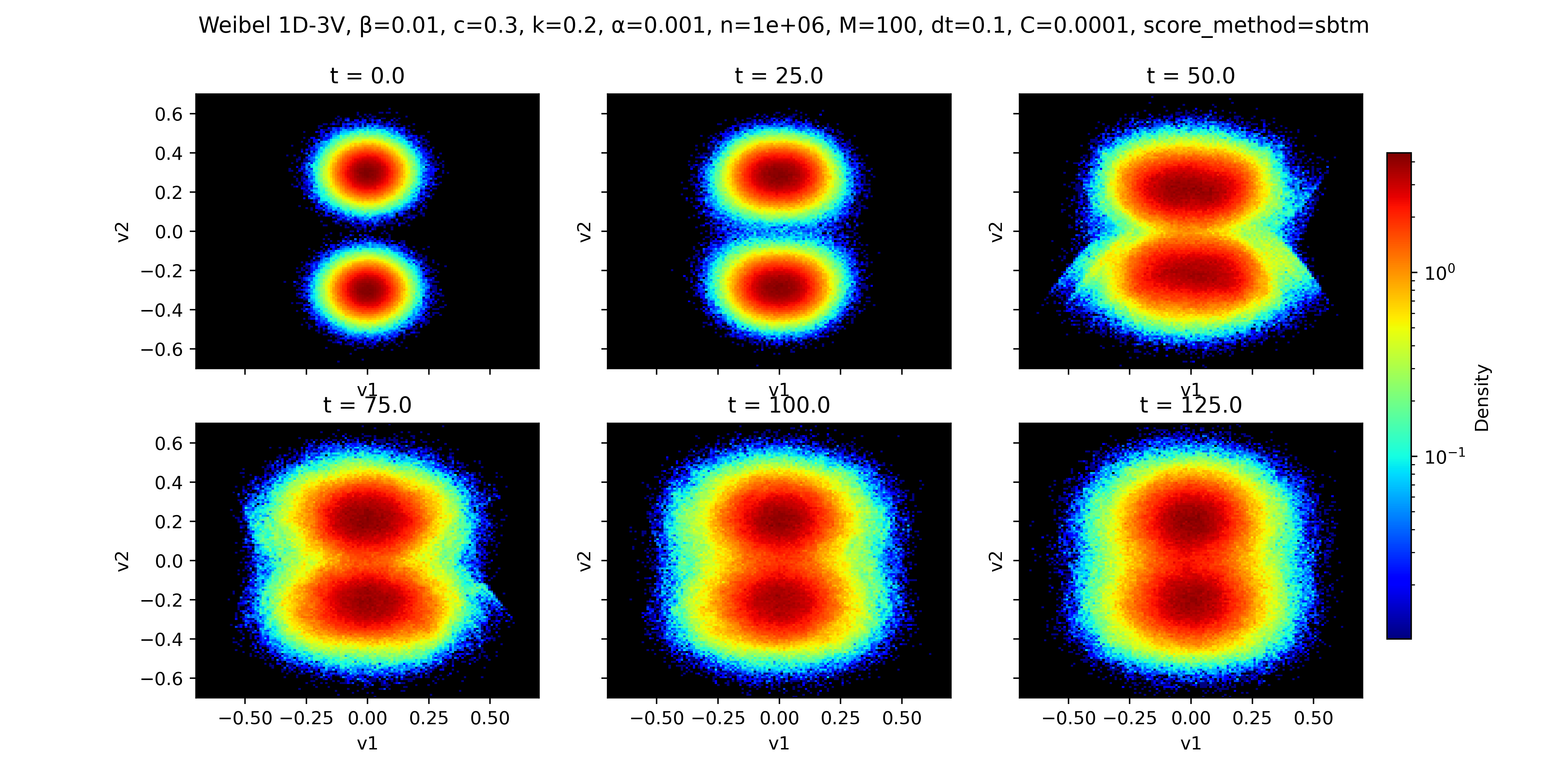}
\caption{SBTM, $\nu = 10^{-4}$}
\end{subfigure}
\begin{subfigure}{0.48\linewidth}
\includegraphics[width=\linewidth]{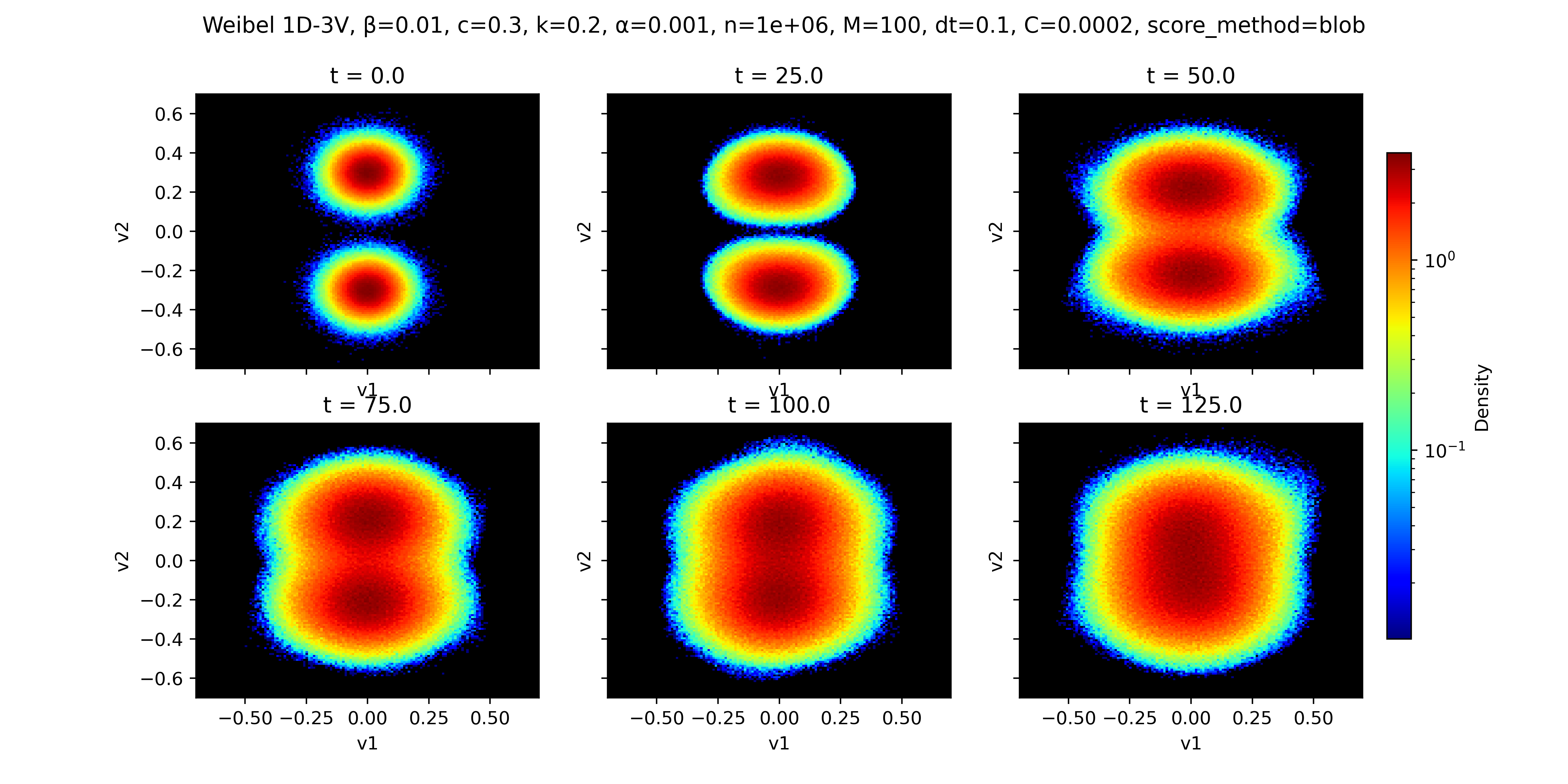}
\caption{Blob, $\nu = 2 \times 10^{-4}$}
\end{subfigure}
\hfill
\begin{subfigure}{0.48\linewidth}
\includegraphics[width=\linewidth]{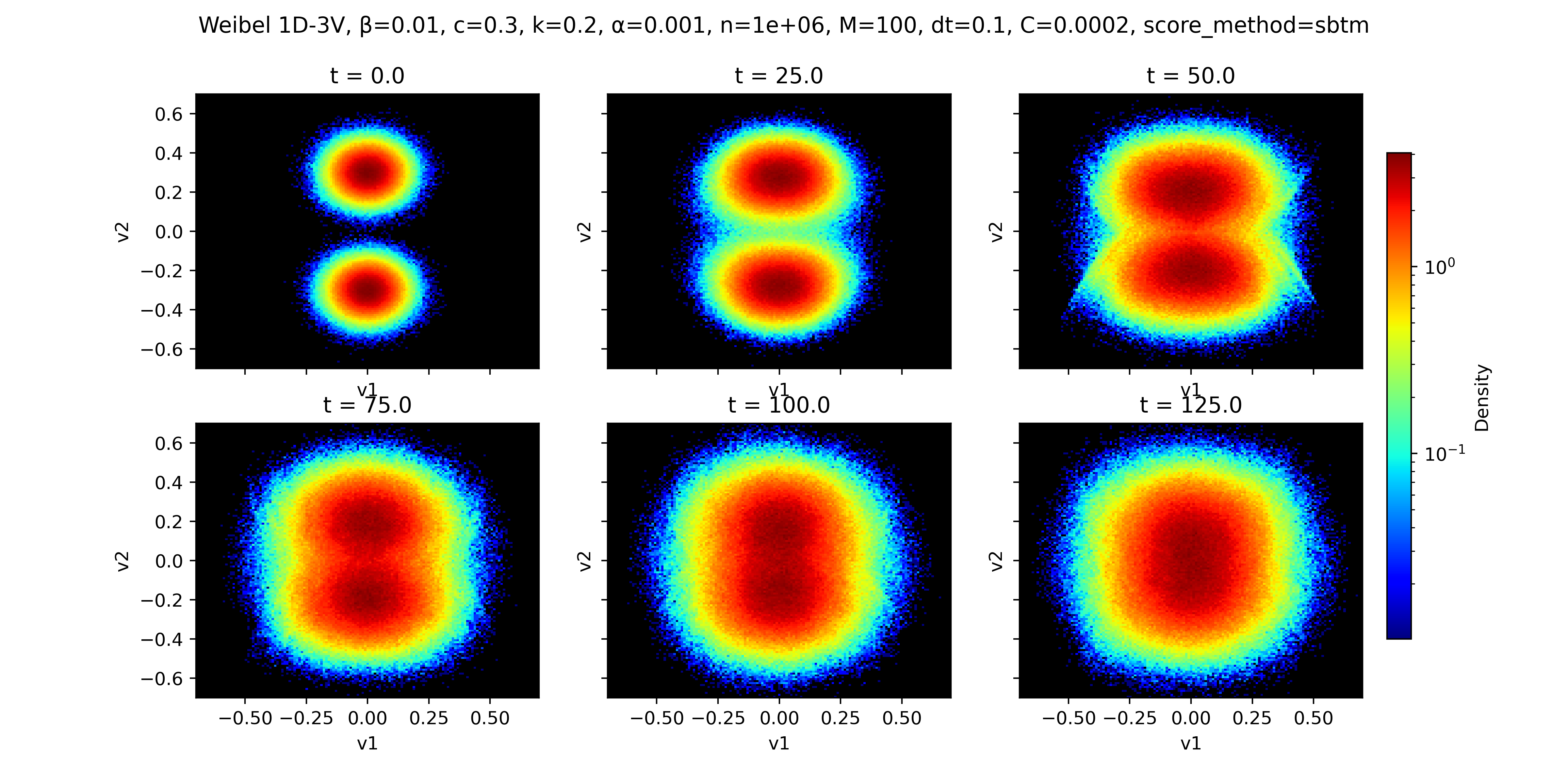}
\caption{SBTM, $\nu = 2 \times 10^{-4}$}
\end{subfigure}
\begin{subfigure}{0.48\linewidth}
\includegraphics[width=\linewidth]{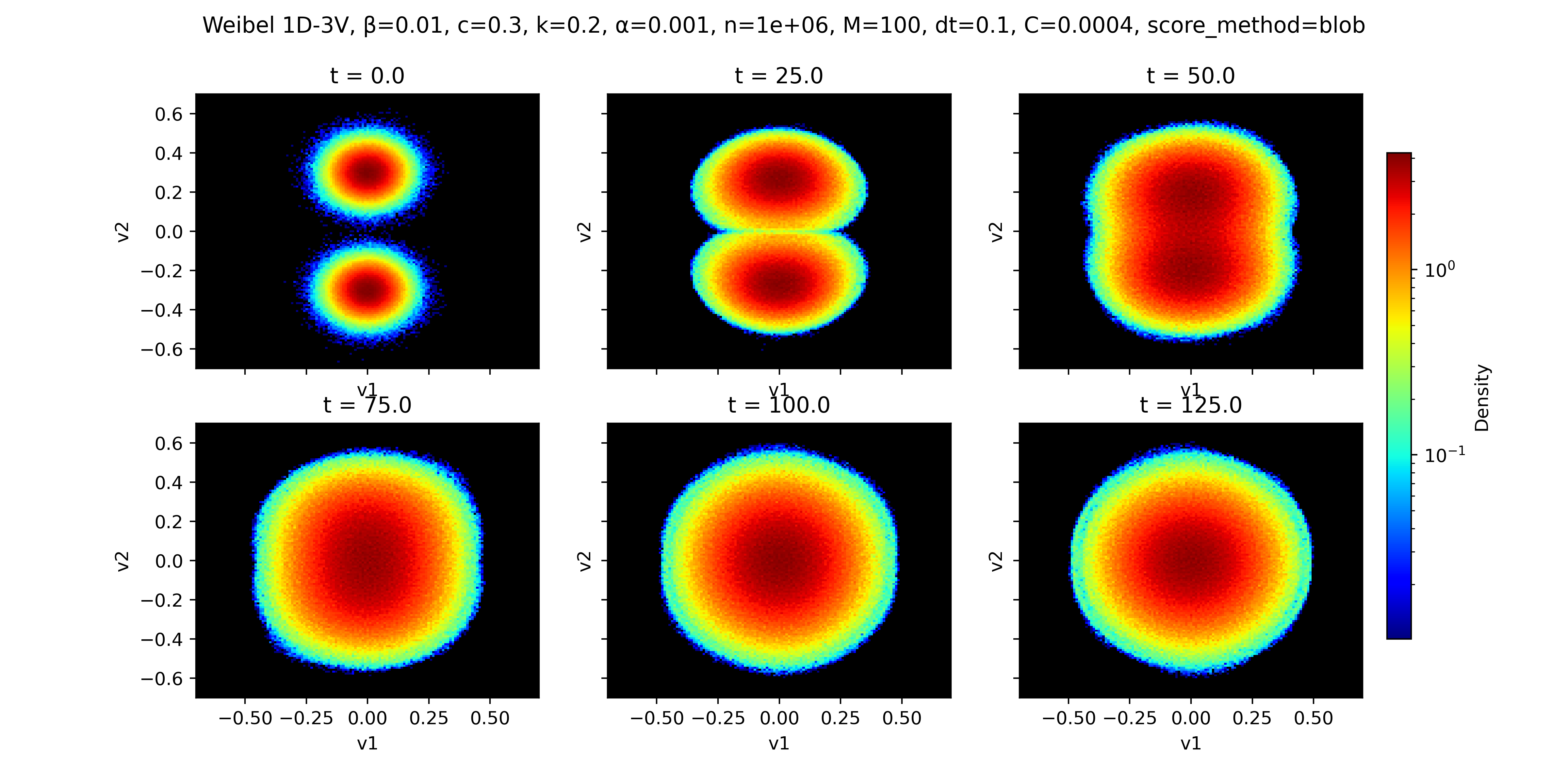}
\caption{Blob, $\nu = 4 \times 10^{-4}$}
\end{subfigure}
\hfill
\begin{subfigure}{0.48\linewidth}
\includegraphics[width=\linewidth]{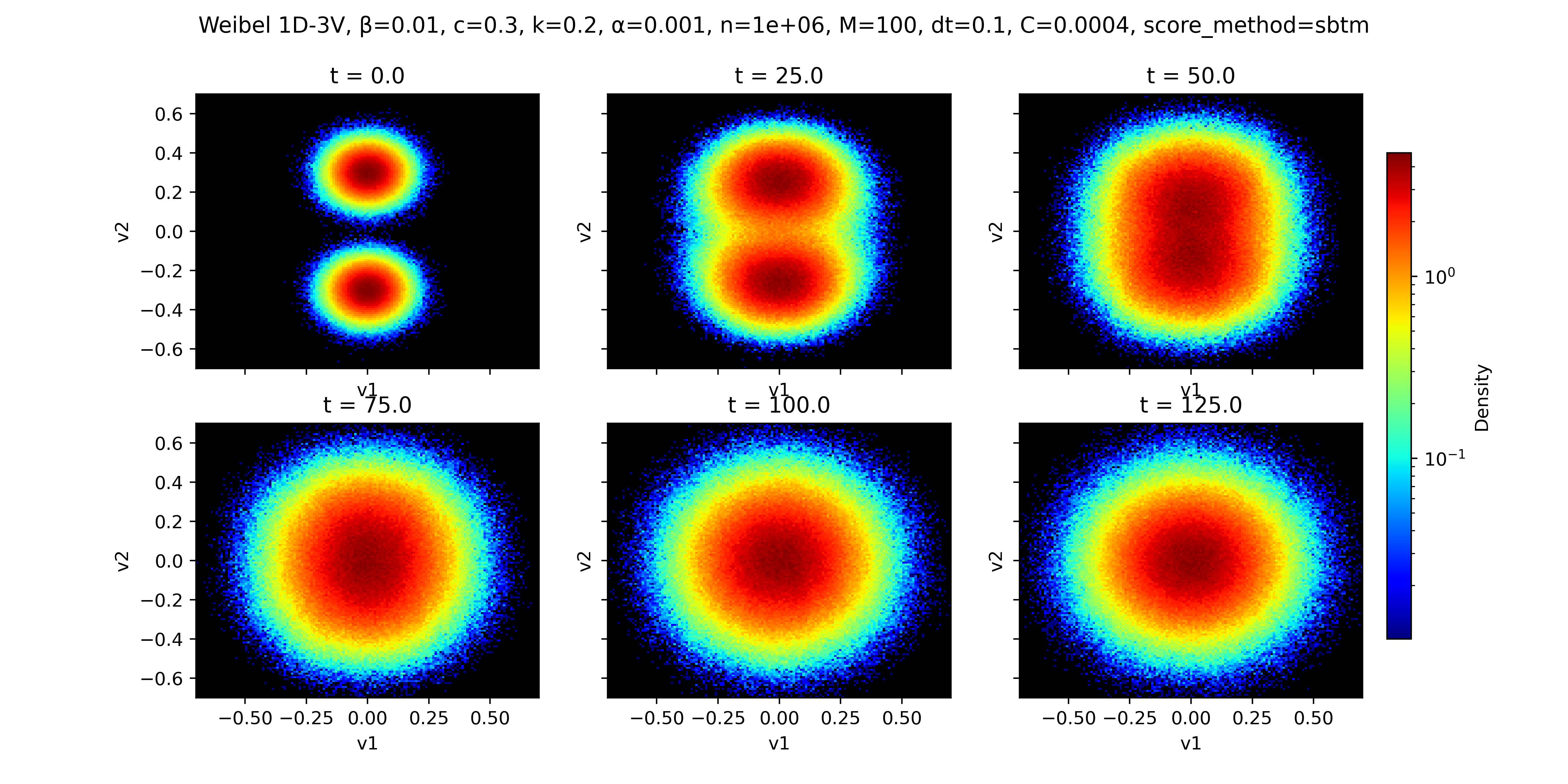}
\caption{SBTM, $\nu = 4 \times 10^{-4}$}
\end{subfigure}
\begin{subfigure}{0.48\linewidth}
\includegraphics[width=\linewidth]{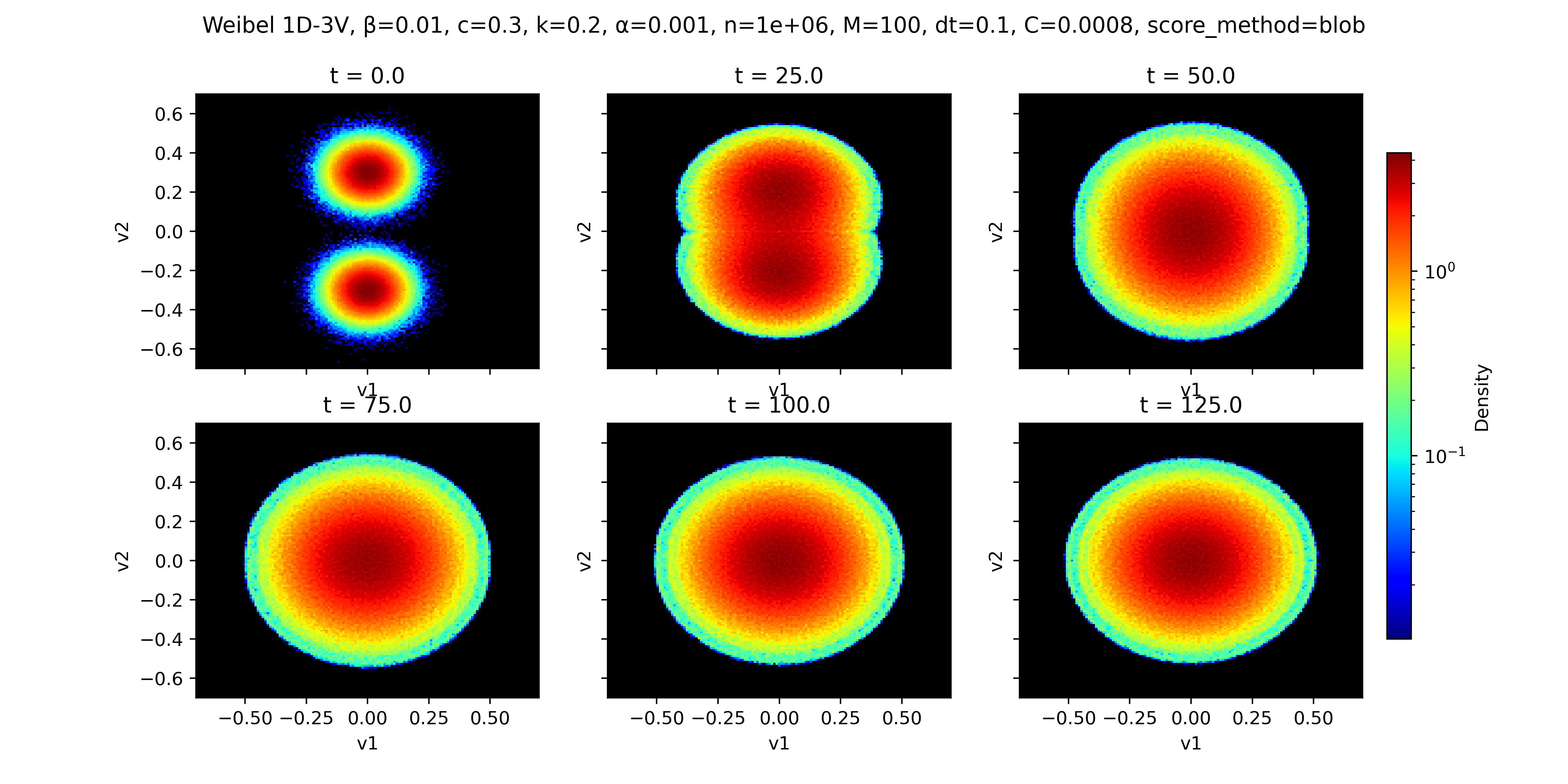}
\caption{Blob, $\nu = 8 \times 10^{-4}$}
\end{subfigure}
\hfill
\begin{subfigure}{0.48\linewidth}
\includegraphics[width=\linewidth]{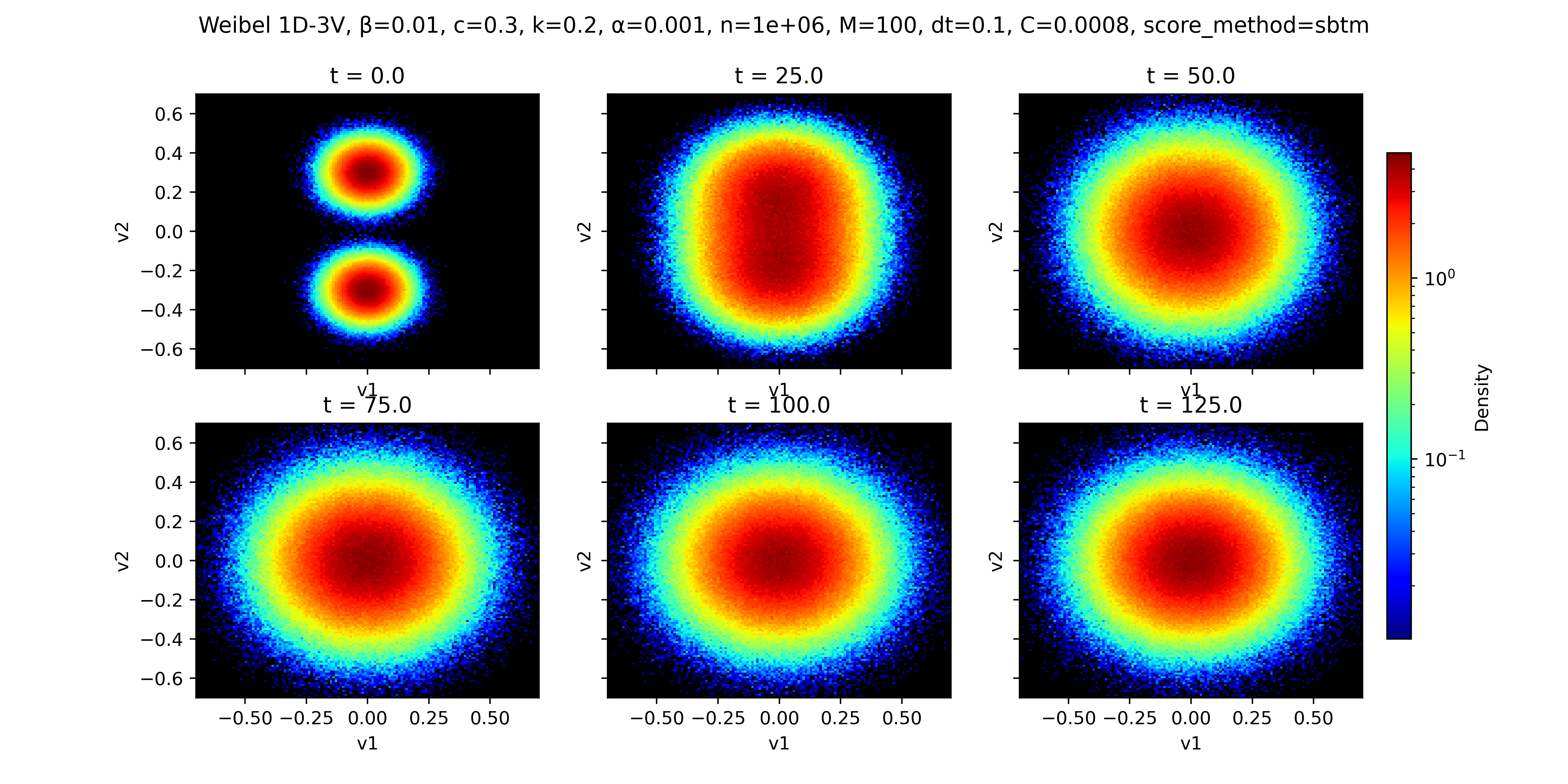}
\caption{SBTM, $\nu = 8 \times 10^{-4}$}
\end{subfigure}
\caption{Weibel instability, $d_v = 3$: $(v_1, v_2)$-marginal density across collision frequencies.  The blob method's equilibration failure is more pronounced than in $d_v = 2$ (Figure~\ref{fig:weibel_sweep_dv2}), indicating that SBTM scales better to higher dimensions.}
\label{fig:weibel_sweep_dv3}
\end{figure}

\FloatBarrier

\subsection{Computational Efficiency}
\label{sec:efficiency}

Table~\ref{tab:efficiency} compares runtime and peak GPU memory for the two-stream instability test case.  Both methods share the same $O(n^2)$ collision force evaluation (per spatial cell), which dominates the overall cost. The difference lies in score estimation: the blob method requires forming and storing pairwise kernel evaluations ($O(n^2)$ memory per cell), whereas SBTM amortizes the score via a neural network with $O(1)$ memory in $n$. As a result, SBTM achieves a $1.57\times$ speedup at $n = 10^6$ and uses $2$ -- $4\times$ less memory across all particle counts, with the memory advantage growing at larger~$n$.
\begin{table}[H]
\begin{center}
\begin{tabular}{lcccc}
\toprule
& \multicolumn{2}{c}{\textbf{Runtime} ($n = 10^6$)} & \multicolumn{2}{c}{\textbf{Peak Memory (GB)}} \\
\cmidrule(lr){2-3} \cmidrule(lr){4-5}
\textbf{Method} & Time & Speedup & $n = 10^6$ & $n = 3 \times 10^6$ \\
\midrule
Blob (KDE)  & 11h 27m & --          & 16.4 & 65.1 \\
SBTM        & 7h 17m  & $1.57\times$ & 4.4  & 17.2 \\
\bottomrule
\end{tabular}
\caption{Runtime and memory comparison for the two-stream instability test case ($M = 100$, $\Delta t = 0.05$, $d_v = 3$, $t_{\mathrm{final}} = 50$).}
\label{tab:efficiency}
\end{center}
\end{table}

\section{Conclusion}
\label{sec:conclusion}

We introduced SBTM as a drop-in replacement for the blob method in deterministic particle simulations of the Vlasov--Maxwell--Landau system.  The two methods share the same PIC framework and differ only in how the velocity score is estimated: a kernel density estimate (blob) versus a neural network trained on-the-fly via implicit score matching (SBTM).  This single change yields consistent improvements across all three benchmarks: SBTM relaxes correctly to Maxwellian equilibrium, converges with fewer particles, and runs faster with lower memory.  On the theoretical side, we proved that the discretized collision operator preserves momentum and kinetic energy for any score approximation, characterized the unique global steady states of the VML and VPL systems (Theorems~\ref{thm:equilibrium} and~\ref{thm:equilibrium_vpl}), and formally verified Theorem~\ref{thm:equilibrium} in the Lean~4 proof assistant \citep{ilin2026formalization}.

Several limitations of the current work suggest natural directions for future research.  The time integration uses forward Euler, which does not conserve energy at the fully discrete level; adopting the energy-conserving explicit scheme of \citep{yoo2025explicit} would address this at the cost of additional collision operator evaluations per time step.  All experiments are restricted to the 1D-in-space setting with a single electron species and an immobile ion background; extensions to multiple spatial dimensions and multi-species plasmas are important for practical applications.  

\section*{Declaration of generative AI use}
The proofs of Theorems~\ref{thm:equilibrium} and~\ref{thm:equilibrium_vpl} were obtained with the assistance of Gemini DeepThink.  

\section*{Acknowledgement}
The work of J. Hu was partially supported by DOE grant DE-SC0023164, NSF grants DMS-2409858
and IIS-2433957, and DoD MURI grant FA9550-24-1-0254.

\appendix
\section{Proofs of Equilibrium Theorems}
\label{app:proofs_equilibrium}

In this appendix, we provide complete proofs of Theorems~\ref{thm:equilibrium} and~\ref{thm:equilibrium_vpl}. The formal verification of Theorem \ref{thm:equilibrium} in the Lean~4 proof assistant is carried out in \citep{ilin2026formalization}.

\begin{proof}[Proof of Theorem~\ref{thm:equilibrium}]
The proof proceeds in two stages: (1)~the Landau entropy dissipation forces $f$ to be a local Maxwellian in~$v$, and (2)~the Vlasov--Maxwell equations constrain the spatial dependence, forcing a global Maxwellian.

\textbf{Step 1: Entropy dissipation and the local Maxwellian.}
At equilibrium, $\partial_t f = \partial_t E = \partial_t B = 0$, and the equation \eqref{eq:vlasov}  reduces to
\begin{equation}
\label{eq:steady_vlasov}
v \cdot \nabla_x f + (E + v \times B) \cdot \nabla_v f = \nu \, Q[f, f].
\end{equation}
Multiply by $\log f$ and integrate over phase space $\mathbb{T}^3 \times \R^3$:
\[
\iint_{\mathbb{T}^3 \times \R^3} \!\Big[ v \cdot \nabla_x f + (E + v \times B) \cdot \nabla_v f \Big] \log f \, dv \, dx = \nu \iint_{\mathbb{T}^3 \times \R^3} Q[f,f] \, \log f \, dv \, dx.
\]
Using $\nabla(f \log f - f) = \log f \, \nabla f$, the spatial transport term integrates to zero by the divergence theorem on the periodic domain $\mathbb{T}^3$.  Integrating the Lorentz force term by parts with respect to $v$ yields a factor $\nabla_v \cdot (E + v \times B)$; since $\nabla_v \cdot E = 0$ and $\nabla_v \cdot (v \times B) = 0$, the entire left-hand side vanishes, leaving the global entropy production constraint:
\begin{equation}
\label{eq:zero_dissipation}
\int_{\mathbb{T}^3} \!\left(\int_{\R^3} Q[f,f] \, \log f \, dv \right) dx = 0.
\end{equation}
By the H-theorem \eqref{eq:H-theorem} for the Landau operator, $\int_{\R^3} Q[f,f] \log f \, dv \leq 0$ pointwise in $x$.  Since its spatial integral is zero, it must vanish identically.  The null space of the Landau operator is spanned by the collisional invariants $\{1, v, |v|^2\}$ \citep{villani2002review}, meaning $\log f$ must be a linear combination of these:
\[
\log f(x,v) = a(x) + b(x) \cdot v + c(x)|v|^2.
\]
By mapping the coefficients $a(x)$, $b(x)$, and $c(x)$ to the macroscopic fluid fields, we rewrite this as a \emph{local Maxwellian} \citep[Ch.~II.7, Eq.~7.8]{cercignani1988boltzmann}:
\begin{equation}
\label{eq:local_maxwellian}
f(x, v) = \frac{\rho(x)}{(2\pi T(x))^{3/2}} \exp\!\left(-\frac{|v - u(x)|^2}{2T(x)}\right),
\end{equation}
with density $\rho(x) = \int_{\R^{3}} f \, dv > 0$, bulk velocity $u(x) = \frac{1}{\rho}\int_{\R^{3}} v f \, dv$, and temperature $T(x) = \frac{1}{3\rho}\int_{\R^{3}} |v - u|^2 f \, dv > 0$.

\textbf{Step 2: Vlasov kinematics (polynomial matching).}
Since $f$ is a local Maxwellian, $Q[f,f] = 0$.  The Vlasov equation simplifies to
\[
v \cdot \nabla_x (\log f) + (E + v \times B) \cdot \nabla_v (\log f) = 0.
\]
Substituting $\log f = a(x) + b(x) \cdot v + c(x)|v|^2$ and computing the gradients, noting that $(v \times B) \cdot v = 0$ and $(v \times B) \cdot b = v \cdot (B \times b)$, the equation reduces to a polynomial in $v$:
\begin{equation}
\label{eq:poly_v}
(v \cdot \nabla_x c)|v|^2 + \sum_{i,j=1}^{3} v_i v_j \, \partial_{x_i} b_j + v \cdot \!\left(\nabla_x a + 2c\, E + B \times b\right) + E \cdot b = 0.
\end{equation}
Since the magnetic component of the Lorentz force is perpendicular to $v$, it contributes nothing to the quadratic or cubic velocity terms.  Therefore, matching the highest-order coefficients follows the standard unforced kinematic derivation \citep[Ch.~III.10, Eqs.~10.6--10.7]{cercignani1988boltzmann}.
For this to hold for all $v \in \R^3$, the coefficient of every homogeneous degree of $v$ must vanish identically.

\textbf{Step 3: Enforcing macroscopic uniformity.}

\textbf{(a) $\mathcal{O}(|v|^3)$ terms: temperature is constant.}  The cubic terms give $v \cdot \nabla_x c = 0$, hence $\nabla_x c = 0$ \citep[Eq.~10.7]{cercignani1988boltzmann}.  Since $c = -1/(2T)$, we conclude $T(x) \equiv T_\infty$ is globally constant.

\textbf{(b) $\mathcal{O}(|v|^2)$ terms: bulk velocity is constant.}  With $c$ constant, the quadratic terms give $\sum_{i,j} v_i v_j \, \partial_{x_i} b_j = 0$.  This yields Killing's equation $\partial_{x_i} b_j + \partial_{x_j} b_i = 0$ \citep[Eq.~10.6]{cercignani1988boltzmann}.  Integrating the sum of squares of these derivatives over $\mathbb{T}^3$ and integrating by parts shows that $\partial_{x_i} b_j = 0$ everywhere.  Thus $b(x)$ is constant, meaning the bulk velocity $u(x) \equiv u_\infty$ is uniform.

\textbf{(c) $\mathcal{O}(|v|^1)$ terms: force balance.}  With constant $T_\infty$ and $u_\infty$, substituting $c = -1/(2T_\infty)$ and $b = u_\infty / T_\infty$ yields the macroscopic force balance:
\begin{equation}
\label{eq:force_balance}
\nabla_x \log \rho = \frac{1}{T_\infty}(E + u_\infty \times B).
\end{equation}

\textbf{Step 4: Amp\`ere's law and the nullification of bulk velocity.}
Since $u_\infty$ is constant, the current density is $J(x) = \rho(x) \, u_\infty$.  Amp\`ere's law at steady state gives $\nabla_x \times B = J = \rho(x) \, u_\infty$ \citep[cf.][Eqs.~60--62]{guo2003vlasov}.  Taking the dot product with $u_\infty$:
\[
u_\infty \cdot (\nabla_x \times B) = \rho(x) |u_\infty|^2.
\]
Using the vector identity $\nabla_x \cdot (B \times u_\infty) = u_\infty \cdot (\nabla_x \times B) - B \cdot (\nabla_x \times u_\infty)$ and noting $\nabla_x \times u_\infty = 0$:
\[
\nabla_x \cdot (B \times u_\infty) = \rho(x)|u_\infty|^2.
\]
Integrating over the periodic domain $\mathbb{T}^3$, the exact divergence vanishes:
\[
0 = |u_\infty|^2 \int_{\mathbb{T}^3} \rho(x) \, dx.
\]
Because $\rho(x) > 0$ strictly, the integral is positive.  Therefore $|u_\infty|^2 = 0$, i.e., the bulk velocity vanishes: $u_\infty = 0$.

\textbf{Step 5: Electrostatic balance and the maximum principle.}
With $u_\infty = 0$, the force balance~\eqref{eq:force_balance} reduces to $\nabla_x \log \rho = E / T_\infty$.  Taking the divergence and substituting Gauss's law $\nabla_x \cdot E = \rho - \rho_{\mathrm{ion}}$ yields an elliptic PDE for the density \citep[cf.][Eq.~142]{guo2012vlasov}:
\begin{equation}
\label{eq:poisson_boltzmann}
T_\infty \, \Delta_x (\log \rho) = \rho(x) - \rho_{\mathrm{ion}}.
\end{equation}
Because the periodic domain $\mathbb{T}^3$ is compact and $\rho(x)$ is smooth and strictly positive, $\rho$ achieves a global maximum and a global minimum.
\begin{itemize}
\item At the global maximum $x_{\max}$: $\Delta_x \log \rho \leq 0$, so $\rho(x_{\max}) \leq \rho_{\mathrm{ion}}$.
\item At the global minimum $x_{\min}$: $\Delta_x \log \rho \geq 0$, so $\rho(x_{\min}) \geq \rho_{\mathrm{ion}}$.
\end{itemize}
Combining: $\rho_{\mathrm{ion}} \leq \rho(x_{\min}) \leq \rho(x) \leq \rho(x_{\max}) \leq \rho_{\mathrm{ion}}$, forcing $\rho(x) \equiv \rho_{\mathrm{ion}}$.  Since the density is constant, $\nabla_x \log \rho = 0$, and plugging back into the force balance gives $E(x) = 0$ everywhere.

\textbf{Step 6: Uniformity of the magnetic field.}
With $u_\infty = 0$, the macroscopic current $J = 0$, so Amp\`ere's law reduces to $\nabla_x \times B = 0$.  Coupled with $\nabla_x \cdot B = 0$, the vector Laplacian is
\[
\Delta_x B = \nabla_x(\nabla_x \cdot B) - \nabla_x \times (\nabla_x \times B) = 0.
\]
Multiplying by $B$ and integrating over $\mathbb{T}^3$ yields $-\int |\nabla_x B|^2 \, dx = 0$, so $\nabla_x B = 0$.  The magnetic field is a global constant: $B(x) \equiv B_\infty$.

A uniform magnetic field exerts no net force on the isotropic Maxwellian equilibrium: since $\nabla_v f_\infty \propto v \, f_\infty$, the magnetic force term $(v \times B_\infty) \cdot \nabla_v f_\infty \propto (v \times B_\infty) \cdot v = 0$ vanishes identically for any $B_\infty \in \R^3$.  Thus the steady-state Vlasov equation admits a one-parameter family of equilibria indexed by $B_\infty$.

\textbf{Step 7: Determination of $T_\infty$ and $B_\infty$.}
The specific equilibrium selected by the dynamics is uniquely fixed by two conservation laws.  First, Faraday's law $\partial_t B = -\nabla_x \times E$ implies that the spatial mean of $B$ is conserved: $\frac{d}{dt}\int_{\mathbb{T}^3} B \, dx = -\int_{\mathbb{T}^3} \nabla_x \times E \, dx = 0$ by periodicity \citep[cf.][p.~596]{guo2003vlasov}.  At equilibrium $B \equiv B_\infty$, so
\begin{equation}
\label{eq:B_conservation}
B_\infty = \frac{1}{|\mathbb{T}^3|}\int_{\mathbb{T}^3} B_{\mathrm{init}}(x) \, dx.
\end{equation}
Second, with $E = 0$ and $B \equiv B_\infty$, the total energy is
\[
\mathcal{E} = \iint \frac{|v|^2}{2} f_\infty \, dv \, dx + \frac{|B_\infty|^2}{2} |\mathbb{T}^3| = \frac{3}{2}\, \rho_{\mathrm{ion}} \, T_\infty \, |\mathbb{T}^3| + \frac{|B_\infty|^2}{2} |\mathbb{T}^3|.
\]
Since $\mathcal{E}$ is conserved \eqref{eq:energy-consv} and $B_\infty$ is fixed by the initial data, $T_\infty$ is uniquely determined:
\begin{equation}
\label{eq:T_infty}
T_\infty = \frac{2}{3\, \rho_{\mathrm{ion}} \, |\mathbb{T}^3|} \left(\mathcal{E} - \frac{|B_\infty|^2}{2} |\mathbb{T}^3|\right).
\end{equation}
Positivity of $T_\infty$ is guaranteed by Jensen's inequality: $|B_\infty|^2 \leq \frac{1}{|\mathbb{T}^3|}\int |B_{\mathrm{init}}|^2 \, dx$, so the magnetic energy of the uniform field is at most the initial magnetic energy. \qedhere
\end{proof}

\begin{proof}[Proof of Theorem~\ref{thm:equilibrium_vpl}]
Steps~1--3 of the proof of Theorem~\ref{thm:equilibrium} carry over verbatim with $B = 0$: the magnetic field plays no role in the entropy dissipation argument (Step~1), and the magnetic force contributes nothing to the $\mathcal{O}(|v|^3)$ and $\mathcal{O}(|v|^2)$ terms of the polynomial matching (Steps~2--3).  Thus the steady state is again a local Maxwellian with globally constant temperature $T_\infty$ and uniform bulk velocity $u_\infty$, satisfying the force balance
\begin{equation}
\label{eq:force_balance_vpl}
\nabla_x \log \rho = \frac{E}{T_\infty}.
\end{equation}

\textbf{Step 4 breaks.}  In the VML system, Amp\`ere's law $\nabla_x \times B = \rho \, u_\infty$ forces $u_\infty = 0$ (by integrating over $\mathbb{T}^3$).  In the VPL system there is no magnetic field and no Amp\`ere's law.  Instead, the $\mathcal{O}(|v|^0)$ term of the polynomial~\eqref{eq:poly_v} (with $B = 0$) gives $E \cdot u_\infty / T_\infty = 0$.  But this imposes no constraint on $u_\infty$ once we establish $E = 0$.

\textbf{Step 5 carries over.}  The electrostatic balance and maximum principle argument is identical: from \eqref{eq:force_balance_vpl} and Gauss's law $\nabla_x \cdot E = \rho - \rho_{\mathrm{ion}}$, we obtain the same Poisson-Boltzmann equation \eqref{eq:poisson_boltzmann}, and the maximum principle forces $\rho \equiv \rho_{\mathrm{ion}}$ and $E = 0$.

\textbf{Step 6 is absent} (no magnetic field to constrain).

\textbf{Determination of $u_\infty$ and $T_\infty$.}  The VPL system is Galilean invariant (unlike the full VML system, which is Lorentz invariant), so a uniform bulk drift generates no back-reacting magnetic field.  Since the momentum is conserved \eqref{eq:moment-consv}, equating the equilibrium momentum to the initial value gives $u_\infty = \iint v \, f_{\mathrm{init}} \, dv \, dx / (\rho_{\mathrm{ion}}\, |\mathbb{T}^3|)$.  The equilibrium temperature then follows from conservation of total energy: $T_\infty = \frac{2}{3\, \rho_{\mathrm{ion}} \, |\mathbb{T}^3|}(\mathcal{E}_0 - \frac{1}{2}\rho_{\mathrm{ion}} |u_\infty|^2 |\mathbb{T}^3|)$.

\end{proof}

\bibliographystyle{plain}
\bibliography{references}

\end{document}